\documentclass[a4paper,11pt]{article}

\usepackage[utf8]{inputenc}
\usepackage[T1]{fontenc}
\usepackage[english]{babel}
\usepackage{amsmath}
\usepackage{amssymb}
\usepackage{amsthm} 
\usepackage{mathtools}
\usepackage{nomencl} 
\usepackage{gensymb}
\usepackage[title]{appendix}
\makenomenclature
\usepackage{makeidx}
\makeindex
\usepackage{array}
\usepackage{ifthen}
\usepackage{graphicx}
\usepackage{lmodern}
\usepackage{amssymb}
\usepackage{multirow}
\usepackage{appendix}
\usepackage{bbm}
\usepackage{ushort}
\usepackage{afterpage}
\usepackage{tikz-cd}
\usetikzlibrary{matrix,arrows}

\usepackage{here}
\usepackage{geometry}
\usepackage[pdftex]{hyperref}
\usepackage{pifont}
\usetikzlibrary{shapes}
\usetikzlibrary{calc}
\usetikzlibrary{arrows}
\usetikzlibrary{decorations.markings}

\DeclareMathOperator{\Ad}{Ad}
\DeclareMathOperator{\ad}{ad}
\DeclareMathOperator{\Aut}{Aut}
\DeclareMathOperator{\can}{can}

\DeclareMathOperator{\e}{e}

\DeclareMathOperator{\fppf}{fppf}

\DeclareMathOperator{\Grass}{Grass}
\DeclareMathOperator{\GL}{GL}
\DeclareMathOperator{\gl}{\mathfrak{gl}}

\DeclareMathOperator{\h}{h}
\DeclareMathOperator{\had}{had}

\DeclareMathOperator{\haut}{ht}

\DeclareMathOperator{\im}{Im}
\DeclareMathOperator{\Isom}{Isom}

\DeclareMathOperator{\Lie}{Lie}

\DeclareMathOperator{\nil}{Nil}
\DeclareMathOperator{\N}{N}

\DeclareMathOperator{\PGL}{PGL}

\DeclareMathOperator{\rad}{rad}
\DeclareMathOperator{\radu}{\mathfrak{rad_u}}
\DeclareMathOperator{\Rad}{Rad}
\DeclareMathOperator{\rank}{rank}
\DeclareMathOperator{\rg}{rg}

\DeclareMathOperator{\SL}{SL}

\DeclareMathOperator{\Spec}{Spec}
\DeclareMathOperator{\Ss}{ss}
\DeclareMathOperator{\Sc}{sc}

\newcommand{\cat}[1]{\underline{\mathrm{#1}}}
\newcommand{\newcategory}[2]{
    \newcommand{#1}{\cat{#2}}
}

\newcategory{\Ab}        {Ab}
\newcategory{\Alg}        {Alg}
\newcategory{\Bij}        {Bij}
\newcategory{\Cat}        {Cat}
\newcategory{\Field}    {Field}
\newcategory{\Fin}        {Fin}
\newcategory{\Graph}    {Graph}
\newcategory{\Grp}        {Grp}
\newcategory{\Mod}        {Mod}
\newcategory{\Mon}        {Mon}
\newcategory{\Op}        {Op}
\newcategory{\Pos}        {Pos}
\newcategory{\Ring}        {Ring}
\newcategory{\Set}        {Set}
\newcategory{\Top}        {Top}

\newcommand{\red}{\mathrm{red}}
\newcolumntype{P}[1]{>{\centering\arraybackslash}p{#1}}

\DeclareMathAlphabet{\mathpzc}{OT1}{pzc}{m}{it}
\newtheorem{theorem}{Theorem}[section]
\newtheorem*{theorem*}{Theorem}
\newtheorem{proposition}[theorem]{Proposition}
\newtheorem{statement}[theorem]{Statement}
\newtheorem{lemma}[theorem]{Lemma}
\newtheorem{corollary}[theorem]{Corollary}
\newtheorem*{corollary*}{Corollary}

\theoremstyle{definition}
\newtheorem{defn}[theorem]{Definition}

\theoremstyle{definition}
\newtheorem{remark}[theorem]{Remark}

\theoremstyle{definition}
\newtheorem{remarks}[theorem]{Remarks}

\theoremstyle{definition}

\theoremstyle{definition}
\newtheorem{exs}[theorem]{Examples}

\makeatletter
\tikzcdset{
open/.code={\tikzcdset{hook, circled};},
closed/.code={\tikzcdset{hook, slashed};},
open'/.code={\tikzcdset{hook', circled};},
closed'/.code={\tikzcdset{hook', slashed};},
circled/.code={\tikzcdset{markwith={\draw (0,0) circle (.375ex);}};},
slashed/.code={\tikzcdset{markwith={\draw[-] (-.4ex,-.4ex) -- (.4ex,.4ex);}};},
markwith/.code={
\pgfutil@ifundefined{tikz@library@decorations.markings@loaded}%
{\pgfutil@packageerror{tikz-cd}{You need to say %
\string\usetikzlibrary{decorations.markings} to use arrow with markings}{}}{}%
\pgfkeysalso{/tikz/postaction={/tikz/decorate,
/tikz/decoration={
markings,
mark = at position 0.5 with
{#1}}}}},
}
\makeatother
\begin{document}

\newcommand\blankpage{%
\null
    \thispagestyle{empty}%
    \addtocounter{page}{-1}%
    \newpage}
    
\renewcommand{\contentsname}{Table des matières}

\renewcommand{\labelenumii}{(\roman{enumii})}
\renewcommand{\nomname}{Notations}

\def\restriction#1#2{\mathchoice
              {\setbox1\hbox{${\displaystyle #1}_{\scriptstyle #2}$}
              \restrictionaux{#1}{#2}}
              {\setbox1\hbox{${\textstyle #1}_{\scriptstyle #2}$}
              \restrictionaux{#1}{#2}}
              {\setbox1\hbox{${\scriptstyle #1}_{\scriptscriptstyle #2}$}
              \restrictionaux{#1}{#2}}
              {\setbox1\hbox{${\scriptscriptstyle #1}_{\scriptscriptstyle #2}$}
              \restrictionaux{#1}{#2}}}
\def\restrictionaux#1#2{{#1\,\smash{\vrule height .8\ht1 depth .85\dp1}}_{\,#2}}

\newcommand{\closure}[2][3]{%
  {}\mkern#1mu\overline{\mkern-#1mu#2}}

\newcommand{\RR}{\ensuremath{\mathbb{R}}}
\newcommand{\KK}{\ensuremath{\mathbb{K}}}
\newcommand{\TT}{\ensuremath{\mathbb{T}}}

\newcommand{\identity}{\mathbbm{1}}
\newcommand{\cchi}{\protect\raisebox{2pt}{$\chi$}}

\newcommand
{\expp}[2]{\ensuremath{\exp \left({2 \pi i \{{#1 }\}_{#2}}\right)}}

\newcommand
{\expr}[2]{\ensuremath{\exp \left( {#1 }{2 \pi i {#2 }}\right)}}

\newcommand
{\uroot}[2]{\ensuremath{\exp \left(2 \pi i \frac{#1}{#2} \right)}}

\newcommand
{\fp}[1]{\ensuremath{\{#1\}_{p}}}

\newcommand
{\quo}[2]{\ensuremath{#1/\!\raisebox{-.65ex}{\ensuremath{#2}}}}

\newcommand\abs[1]{\lvert#1\rvert}
\newcommand\Abs[1]{\left|#1\right|}
\newcommand\bigAbs[1]{\bigl|#1\bigr|}

\def\stackbelow#1#2{\underset{\displaystyle\overset{\displaystyle\shortparallel}{#2}}{#1}}

      \title{Analogues of Morozov's Theorem in characteristic \texorpdfstring{$p>0$}{Lg}}
	\author{Marion Jeannin}
\maketitle

\abstract{Let $k$ be an algebraically closed field of characteristic $p>0$ and let $G$ be a reductive $k$-group. In this article we prove an analogue of Morozov's Theorem when $p$ is separably good for $G$ and under some extra assumptions. Morozov's Theorem characterises, in characteristic $0$, the Lie algebras of parabolic subgroups of $G$ by means of their nilradical. We also obtain an analogue in characteristic $p>0$ of a well known corollary of Morozov's Theorem, that characterises, in characteristic $0$, maximal proper subalgebras of a reductive Lie algebra.
If now $k$ is any field of characteristic $p\geq 0$, let $X$ be a smooth projective geometrically connected $k$-curve. Let us assume that $G$ is a reductive $k$-group which is the twisted form of a constant $X$-group. The existence of the aforementioned analogue in particular allows to adapt the construction of the canonical parabolic subgroup of $G$ proposed by M. Atiyah and R. Bott when $k$ is of characteristic $0$ (\cite{AB}) to the positive characteristic framework,.}

\section{Introduction}
\label{intro}
Let $k$ be a field and $G$ be a reductive $k$-group. In this article we adopt the convention of \cite[XIX]{SGA33}. In particular a reductive $k$-group is connected (see Définition 2.7 ibid.). We denote by $\mathfrak{g}:= \Lie(G)$ the corresponding Lie algebra. If $H\subseteq G$ is a subgroup we denote:
\begin{itemize}
\item by $\mathfrak{h}:= \Lie(H)$ the Lie algebra of $H$,
\item by $H_{\red}$ the reduced part of $H$,
\item by $H^0$ the connected component of the identity of $H$.
\end{itemize}  

Before going any further we remind the reader that is the context of characteristic $p>0$, some Lie algebras come endowed with an additional $p$-power map $\mathfrak{g} \rightarrow \mathfrak{g}, \ x  \mapsto x^{[p]}$ such that:
\begin{enumerate}
\item for any $\lambda \in k$ and any $x \in \mathfrak{g}$, the scalar multiplication is compatible with the $p$-power map, namely $(\lambda x)^{[p]} = \lambda^p x^{[p]}$,
\item for any $x \in \mathfrak{g}$, the Lie bracket is compatible with the $p$-power map, namely $\ad(x^{[p]}) = (\ad(x))^p$,
\item for any $x_1, \ x_2 \in \mathfrak{g}$ we have 
\[(x_1+x_2)^{[p]} = x_1^{[p]} + x_2^{[p]} -W(x_1,x_2),\] where we have set 
\[W(x_1,x_2) = \sum_{0<r<p}\frac{1}{r}\sum_u \ad x_{u(1)}\ad x_{u(2)} \cdots \ad x_{u(p-1)}(x_1)\] 
for $u$ going through all the maps $[1, p-1] \rightarrow \{0,1\}$ that take $r$ times the value $0$.
\end{enumerate}
Given a restricted $p$-Lie algebra, restricted $p$-subalgebras, respectively $p$-ideals, are subalgebras, respectively ideals that are stable under the $p$-power map.

For more details we refer the reader to \cite[II, \S7, n\degree 3]{DG} and \cite[chapter 2]{FS}. Let us only briefly remind here a few specificities and definitions:
\begin{enumerate}
\item Let $\mathfrak{g}$ be a restricted $p$-Lie algebra $\mathfrak{g}$ over $k$. An element $x \in \mathfrak{g}$ is $p$-nil if there is an integer $m \in \mathbb{N}$ such that $x^{[p]^m}= 0$. A restricted $p$-nil algebra is $p$-nil if all its elements are $p$-nilpotent. When $\mathfrak{g}$ is of finite dimension any restricted $p$-subalgebra which is $p$-nilpotent is also $p$-nil. Moreover, let $\h(G)$ be the Coxeter number of $G$. The values of such are listed in Table \ref{Hypothèses_sur_la caractéristiqu_ pour_un_groupe_simple} for any type of $G$. When $p>\h(G)$ G. McNinch shows in \cite{MAUS} that any nilpotent element $x \in \mathfrak{g}$ has an order of $p$-nilpotency equal to $1$ (that is $x^{[p]} = 0$).
\item When $\mathfrak{g}$ is a restricted $p$-subalgebra, we denote:
\begin{itemize}
\item by $\rad(\mathfrak{g})$ the solvable radical of $\mathfrak{g}$, namely its maximal solvable ideal,
\item by $\rad_p(\mathfrak{g})$ its $p$-radical, namely its maximal $p$-nilpotent ideal.
\item by $\nil(\mathfrak{g})$ its nilradical, that is, its maximal nilpotent ideal,
\item by $\mathfrak{z_g}$ its center,
\item by $\radu(\mathfrak{g})$ the Lie algebra of the unipotent radical of $G$ (if $\mathfrak{g}=\Lie(G)$ is the Lie algebra of a smooth connected algebraic group). 
\end{itemize} 
\noindent We briefly recall the following properties (see \cite[Lemma 2.1]{VAS} and \cite[Lemmas 2.6 and 2.13]{JEA1} together with \cite[Remarks 2.7]{JEA1}).
\begin{lemma}
Let $k$ be a field of characteristic $p>0$,  
\begin{enumerate}
\item let $G$ be a connected reductive $k$-group and assume that $p$ is separably good for $G$ (see \cite[Definition 2.2]{PevSta} or \ref{hypothèses_car} for a precise definition, in practice, this amounts to avoid very small characteristics and some factors of $A_{p-1}$-type in characteristic $p$). Then the equalities $\mathfrak{z_g}=\rad(\mathfrak{g}) = \nil(\mathfrak{g})$ hold true; 
\item let $H$ be a smooth connected $k$-group and assume that $p$ is separably good for the quotient $H/\Rad_U(H)$. Then the $p$-radical of $\mathfrak{h}$ is both the Lie algebra of the unipotent radical of $H$ and the set of all $p$-nilpotent elements of $\rad(\mathfrak{h})$.
\end{enumerate}
\label{reminder_rad}
\end{lemma}
\end{enumerate}  

The good analogues for both Lie algebras and their nilradical in the characteristic $p>0$ setting turn out to be respectively restricted $p$-Lie algebras and their $p$-radical. These objects come with their own specificities and bring new characteristic related difficulties. For instance, when $k$ is of characteristic $0$, a reductive Lie algebra is the Lie algebra of a reductive group or equivalently, is a Lie algebra that has a trivial nilradical. In characteristic $p>0$ these two properties are no longer equivalent for the analogous objects one needs to consider. This justifies to define the notion of $p$-reductivity use in Statement \ref{cor_morozov}:

\begin{defn}
A restricted $p$-algebra $\mathfrak{h}$ is $p$-reductive if its $p$-radical is trivial.
\label{def_pred}
\end{defn}

\begin{remark}
Let us note that if $k$ is a perfect field and $G$ is a reductive $k$-group such that $p$ is separably good for $G$, the Lie algebra of $G$ is $p$-reductive. Indeed, according to Lemma \ref{reminder_rad} the $p$-radical of $\mathfrak{g}$ is the Lie algebra of the unipotent radical of $G$, thus is trivial.
\end{remark}

Let $\mathfrak{u}$ be a $p$-nil subalgebra of $\mathfrak{g}$. 

One can associate to $\mathfrak{u}$ a tower of $p$-nil subspaces in $\mathfrak{g}$ (see \cite[XVIII, \S 10, Corollaire 2]{BOU}):
\begin{itemize} 
\item set $\mathfrak{u}_0 :=\mathfrak{u}$ and let $\mathfrak{q}_{1} := N_{\mathfrak{g}}(\mathfrak{u}_0)$ be the normaliser of $\mathfrak{u}_0$, 
\item the subspace $\mathfrak{u}_1$ is the $p$-radical of $\rad(\mathfrak{q}_1)$, 
\item the Lie algebra $\mathfrak{q}_2$ is the normaliser $N_{\mathfrak{g}}(\mathfrak{u}_1)$, 
\item and so on: let $\mathfrak{u}_i$ be the $p$-radical of $N_{\mathfrak{g}}(\mathfrak{u}_{i-1})$ and $\mathfrak{q}_{i+1} := N_{\mathfrak{g}}(\mathfrak{u}_i)$.
\end{itemize} 
\noindent This tower stabilises for dimensional reasons. We denote by $\mathfrak{q}_{\infty}$ the limit object of the tower of normalisers, and by $\mathfrak{u}_{\infty}$ the $p$-radical of $\rad(N_{\mathfrak{g}}(\mathfrak{u}_{\infty}))$. Note that in particular one has  $\mathfrak{q}_{\infty} = N_{\mathfrak{g}}(\mathfrak{u}_{\infty})$. 

In characteristic $0$, the $p$-radical is replaced by the nilradical. In this setting, the limit object $\mathfrak{q}_{\infty}$ can be shown to be a parabolic subalgebra of $\mathfrak{g}$. This result is attributed to V. Morozov and can be found for instance in \cite[VIII, \S10 Corollaire 2]{BOU}.

A similar construction is possible at the group level, this time by considering the tower of smooth connected normalisers of a unipotent smooth connected subgroup $U \subset G$. Namely, one needs to consider the smooth connected part of the normaliser at each iteration, rather than the whole normaliser. Once again the tower stabilises. Let us denote by $k^s$ the separable closure of $k$. Assume that for any field $k$ the unipotent subgroup $U$ is $k^s$-embeddable into the unipotent radical of a parabolic $k$-subgroup of $G$ (this condition is for instance always satisfied when the field $k$ is perfect). A theorem of B. Veisfeiler (see \cite{WEI}) and Borel--Tits (see \cite[Corollaire 3.2]{TB}) then states that the limit smooth connected normaliser is a parabolic subgroup of $G$, denoted by $P_G(U)$ (its Lie algebra is denoted by $\mathfrak{p_g}(U)$). 

Finally, the Hilbert--Mumford--Kempf--Rousseau theory (see \cite{K} and \cite[2.3]{M2}) allows to associate to any $p$-nil Lie subalgebra of $\mathfrak{g}$ an optimal parabolic subgroup (thus an optimal parabolic subalgebra), denoted by $P_G(\lambda_{\mathfrak{u}})$ (respectively $\mathfrak{p_g(\lambda_u)}$), where $\lambda_{\mathfrak{u}}$ is an optimal cocharacter for $\mathfrak{u}$. See section \ref{Kempf-Rousseau} for more details and the definition of the objects involved here.

In this article we aim to determine under which conditions on $p$ and $G$ the following statement holds true:

\begin{statement}[Analogous statement of Morozov's Theorem in characteristic $p>0$]
Let $k$ be an algebraically closed field of characteristic $p>0$ and $G$ be a reductive $k$-group. Let $\mathfrak{u} \subseteq \mathfrak{g}$ be a Lie subalgebra. Assume that $\mathfrak{u}$ is the $p$-radical of its normaliser $N_{\mathfrak{g}}(\mathfrak{u})$. Then:
	\begin{enumerate}
		\item the normaliser $N_\mathfrak{g}(\mathfrak{u})$ is a parabolic subalgebra of $\mathfrak{g}$,
		\item this parabolic subalgebra satisfies the following equalities $N_\mathfrak{g}(\mathfrak{u})= \mathfrak{p_g}(\lambda_{\mathfrak{u}})= \mathfrak{p_g}(U)$ 
where $U \subset G$ is a unipotent smooth connected subgroup such that $\Lie(U) = \mathfrak{u}$ (hypotheses on $p$ will in particular ensure the existence of such a subgroup).
					 
	\end{enumerate}
	\label{analogue_morozov}
	\end{statement}
In this paper we in particular determine explicit assumptions on $p$ and $G$ under which this statement turns out to be true. These are detailed in the remaining part of this introduction.

Moreover, in characteristic $0$ a well known corollary of Morozov's Theorem states that any proper maximal subalgebra of the Lie algebra of a reductive group is either reductive or parabolic (see \cite[VIII, \S 10, Corollaire 1]{BOU}). One may thus also wonder whether such an analogue of this corollary still hold true in characteristic $p>0$, namely:

\begin{statement}
Let $k$ be an algebraically closed field of characteristic $p>0$ and $G$ be a reductive $k$-group. Let $\mathfrak{q} \subseteq \mathfrak{g}$ be proper restricted $p$-Lie subalgebra. Then $\mathfrak{q}$ is either parabolic or $p$-reductive.
\label{cor_morozov}
\end{statement}

Note that A. Premet proved this statement to hold true when $G$ is simple and the characteristic is very good for $G$ (see \cite[Theorem 1.1]{P2}). Premet proof uses Weisfeiler's Theorems on graded Lie algebras and proceeds by a case-by-case analysis (disjunctions of cases coming from the classification of semisimple simply connected groups). We will provide a uniform proof of this statement when $p>0$ is separably good for $G$.

Before detailing the organisation of the present paper, let us briefly provide an overview of what has already been done in the literature:
\begin{itemize}
	\item in \cite[Proposition 4.7]{BDP}, V. Balaji, P. Deligne and A. J. Parameswaran show a result for which the first point of Statement \ref{analogue_morozov} is an immediate corollary, when $p>\h(G)$ and $G$ admits a representation which is:
\begin{enumerate}
\item almost faithful, meaning that its kernel is of multiplicative type,
\item of low height. More precisely, let $V$ be such a representation. Let us denote by $\haut_G(V):= \max\{\sum_{\alpha>0}\langle \lambda, \alpha^{\vee}\rangle\}$ the Dynkin height of the representation, where $\lambda$ is a weight for the action of $T$ on $V$. Then $V$ is of low height if $p > \haut_G(V)$. 
\end{enumerate}  

\item When $G$ satisfies the standard hypotheses (see \cite[2.9]{Jantzen2004}, the definition is reminded in section \ref{hypothèses_car}), A. Premet and D. I. Stewart provide a proof of the first point of Statement \ref{analogue_morozov} by means of a case-by-case study that relies on the classification of semisimple groups (see \cite[Corollary 1.4]{PS}). Note that these hypotheses in particular imply the semisimplicity of $\mathfrak{g}$. This explains why the authors consider the nilradical rather than the $p$-radical of the Lie algebra in their statements. Let us also note that when $\mathfrak{g}$ is semisimple these objects coincide, see for example Remark \ref{semisimple_rad_p_nil} in Section \ref{radicaux_p_alg_restreinte} below. 
\end{itemize}

Our assumptions on $G$ and $p$ (we explicit below) approach, in the most refined version of the analogue statement of Morozov's Theorem the level of generalities of \cite{PS}. Proofs proposed here are uniform and allow to characterise the parabolic subgroup obtained here by means of Geometric Invariant Theory. 

The article is organised as follows: after a quick reminder of some prerequisites Section \ref{Prerequisites}, we will study high characteristics in section \ref{sections_premiers_resultats}. When $p$ is big enough, obtaining an analogue of Morozov's theorem as stated above makes no difficulties but allows to better understand issues that are specific to the characteristic $p>0$ framework. More precisely:
\begin{enumerate}
\item in characteristic $p>0$, the existence of an integration of any restricted $p$-nil Lie subalgebra $\mathfrak{u} \subset \mathfrak{g}$ is not always satisfied. This means that, given a restricted $p$-nil Lie subalgebra $\mathfrak{u} \subset \mathfrak{g}$, one cannot always associate to $U$ a unipotent smooth connected subgroup $U \subset G$ such that $\Lie(U) = \mathfrak{u}$. Note that in characteristic $0$ the exponential map always makes possible an integration process. Furthermore, when such an integration exists, nothing ensures a priori that it will be compatible with the adjoint representation. This last condition would ensure the equality of the normalisers $N_G(U)$ and $N_G(\mathfrak{u})$. These two issues have been studied in a previous article (see \cite{JEA1}). They underline the existence of a huge gap between:
\begin{itemize}
\item the characteristic $p>\h(G)$ case, for which an integration (given by the truncated exponential) is guaranteed in any case, even if it does not satisfy all the properties of the integration process in characteristic $0$. In particular this integration is no longer compatible with the adjoint representation. This explains the condition on the existence of a certain kind of representation for $G$ required in \cite{BDP}. We will show that this is actually superfluous. When $G$ is simple and simply connected this leads to widen the range of admissible characteristics, without any additional assumptions on $G$. In particular we have some gains for types $F_4$, $E_6$, $E_7$ and $E_8$ (see \cite[p.14]{BDP}). 
\item The case of characteristics $p\leq \h(G)$ that are separably good for $G$. In this last framework only a punctual integration is guaranteed in any case. This is ensured by the existence of Springer isomorphisms $\phi : \mathcal{N}_{\red}(\mathfrak{g}) \rightarrow \mathcal{V}_{\red}(G)$, that are $G$-equivariant isomorphisms of reduced varieties, where $\mathcal{N}_{\red}(\mathfrak{g})$ stands for the reduced nilpotent $k$-scheme of $\mathfrak{g}$ and $\mathcal{V}_{\red}(G)$ is the reduced unipotent $k$-scheme of $G$. Such an isomorphism always exists under our assumptions as explained in \cite[Section 3.1]{JEA1}. Nevertheless when $\mathfrak{u}$ is such as in Statement \ref{analogue_morozov}, it is integrable into a unipotent smooth connected subgroup (as shown in \cite[Lemma 5.1]{JEA1}, see also \cite[Lemma 3.5]{JEA1} for the equality of normalisers);
\end{itemize} 
\item when $p>0$, given a subalgebra $\mathfrak{h} \subseteq \mathfrak{g}$, the normaliser $N_G(\mathfrak{h})$ is no longer smooth in general. In characteristic $0$ this smoothness condition for normalisers is automatically satisfied. This is a consequence of a theorem of P. Cartier (see for example \cite[II, \S 6, n\degree 1.1]{DG}). In characteristic $p>0$, counter-examples can be found in \cite[lemma 11]{HS}.
\end{enumerate}

\noindent Even if there are some similarities, we treat separately the case of characteristics $p>\h(G)$ (that is explained in section \ref{chapitre_morozov_p_sup}) and the one of separably good characteristics for $G$ (which is studied in section \ref{section_morozov_p_inf}). Let us stress out that the smoothness of normalisers of subsets of $\mathfrak{g}$ for the action of $G$ is a crucial issue in this article. This is explained with the following theorem, showed in section \ref{Kempf-Rousseau}. In the same section we also remind of the reader how Hilbert--Mumford--Kempf--Rousseau theory allows to associate to any $p$-nil subalgebra $\mathfrak{u} \subset \mathfrak{g}=\Lie(G)$ a destabilising parabolic subgroup $P_G(\lambda_{\mathfrak{u}})$, under suitable conditions on the characteristic $p>0$.

\begin{theorem}
Let $k$ be an algebraically closed field of characteristic $p>0$. Le $G$ be a reductive $k$-group and $\mathfrak{u} \subseteq \mathfrak{g}$ be a restricted $p$-Lie subalgebra such that $\mathfrak{u}$ is the set of $p$-nilpotent elements of $\rad(N_{\mathfrak{g}}(\mathfrak{u}))$. Assume that $p$ is not of torsion for $G$. Let also $P_G(\lambda_{\mathfrak{u}})$ be the optimal destabilizing parabolic subgroup of $\mathfrak{u}$ defined by the Hilbert--Mumford--Kempf--Rousseau method (see \cite[Theorem 3.4]{K} for the existence of such a subgroup). The following assertions are equivalent:
	\begin{enumerate}
		\item the algebra $N_{\mathfrak{g}}(\mathfrak{u})$ is the Lie algebra of the parabolic subgroup $P_G(\lambda_{\mathfrak{u}})$;
		\item the connected component $N_{G}(\mathfrak{u})^{0}$ is smooth;
		\item this connected component satisfies the equality $N_G(\mathfrak{u})^0 = P_G(\lambda_{\mathfrak{u}})$.
	\end{enumerate}
	\label{obstruction}
\end{theorem}

A tool to measure the lack of smoothness of infinitesimally saturated subgroups (see Definition \ref{def_sat_inf} below, normalisers considered in Statement \ref{analogue_morozov} satisfy this property) is provided by previous works of P. Deligne (see \cite[Théorème 2.7]{D1}, stated when $G = \GL(V)$ under some extra assumptions on $V$ and generalised in \cite[Theorem 2.5]{BDP} to any reductive group $G$ when $p>\h(G)$). In particular, when $p>\h(G)$ this theorem leads to the conclusion that if $\mathfrak{u} \subset \mathfrak{g}$ is a restricted $p$-nil $p$-Lie subalgebra, then $N_G(\mathfrak{u})^{0}_{\red}$ differs from $N_G(\mathfrak{u})$ only by a subgroup of multiplicative type. This comes from the fact that, in this context, the normaliser of any such $\mathfrak{u}$ in $G$ is infinitesimally saturated. As a reminder, the following notion of infinitesimal saturation generalises that of saturation introduced in \cite[\S 4]{SER1}. It has first been introduced by P. Deligne (see \cite[Définition 1.5]{D1}). Let us recall that, when $p>\h(G)$ the exponential map is well defined.

\begin{defn}

A subgroup $H \subseteq G$ is infinitesimally saturated if for any $p$-nilpotent element $x \in \mathfrak{h}:= \Lie(H)$ the $t$-power map
	\begin{alignat*}{3}
		\exp_x : \: & \mathbb{G}_a \: & \rightarrow \: & G,\\
			\:& t \: & \mapsto \: & \exp(tx),
	\end{alignat*} 
\noindent factors through $H$. In other words the situation is the following:
\begin{figure}[H]
\begin{center}
\[\begin{tikzpicture} 

 \matrix (m) [matrix of math nodes,row sep=2em,column sep=1.8em,minimum width=2em,  text height = 1.5ex, text depth = 0.25ex]
  {
    \mathbb{G}_a & G,\\
    H. & \\};
  \path[-stealth]
  	(m-1-1) edge
  			node [above] {$\exp_{x}(\cdot)$} (m-1-2)
    (m-1-1) edge[dashed] 
    		node [left] {$\exists$} (m-2-1) 
    (m-2-1) edge (m-1-2) ;
    	
\end{tikzpicture}\]
\end{center}
\end{figure}
\label{def_sat_inf}
\end{defn} 
We show in Section \ref{saturation_inf} that normalisers of a $p$-nil Lie subalgebra are infinitesimally saturated when the adjoint representation is of low height. This still holds true under the assumption $p>\h(G)$ (which is slightly less restrictive). This is a crucial key to prove, in Section \ref{chapitre_morozov_p_sup}, that analogues of Morozov's theorem (see Statement \ref{analogue_morozov}) hold true when $p>\h(G)$.

Section \ref{section_morozov_p_inf} is dedicated to prove the existence of analogues of Morozov's Theorem in separably good characteristics. The notion of infinitesimal saturation is replaced here by that of $\phi$-infinitesimal saturation (see \cite[Definition 4.1]{JEA1}), for $\phi : \mathcal{N}_{\red}(\mathfrak{g}) \rightarrow \mathcal{V}_{\red}(G)$ a Springer isomorphism. This allows us to define:
\begin{defn}
Let $\phi : \mathcal{N}_{\red}(\mathfrak{g}) \rightarrow \mathcal{V}_{\red}(G)$ be a Springer isomorphism for $G$. A subgroup $H \subseteq G$ is $\phi$-infinitesimally saturated if for any $p$-nilpotent element $x \in \mathfrak{h}$ the $t$-power map:
	\begin{alignat*}{3}
		\phi_x : \: & \mathbb{G}_a \: & \rightarrow \: & G,\\
			\:& t \: & \mapsto \: & \phi(tx),
	\end{alignat*} 
factors through $H$. In other words the situation is the following:
\begin{figure}[H]
\begin{center}
\[\begin{tikzpicture} 

 \matrix (m) [matrix of math nodes,row sep=2em,column sep=1.8em,minimum width=2em,  text height = 1.5ex, text depth = 0.25ex]
  {
    \mathbb{G}_a & G,\\
    H. & \\};
  \path[-stealth]
  	(m-1-1) edge
  			node [above] {$\phi_x$} (m-1-2)
    (m-1-1) edge[dashed] 
    		node [left] {$\exists$} (m-2-1) 
    (m-2-1) edge (m-1-2) ;
    	
\end{tikzpicture}\]
\end{center}
\end{figure}
\end{defn}

This notion leads to a natural extension of P. Deligne Theorem for infinitesimally saturated subgroups (see \cite[Theorem 1.1]{JEA1}) and still allow to measure the lack of smoothness of normalisers. Unfortunately it is not possible to adapt the reasoning provided in characteristic $p>\h(G)$ to show that $N_{\mathfrak{g}}(\mathfrak{u})$ is $\phi$-infinitesimally saturated, which justifies the additional assumption of $\phi$-infinitesimal saturation of the group normaliser in the statement of Theorem \ref{morozov_p_inf} below.

Analogues of Morozov's Theorem in characteristic $p>0$ are in particular useful to extend in characteristic $p>0$ (under some assumptions on $p$) Atiyah-Bott's construction of the canonical parabolic subgroup of twisted form of a constant reductive group over a Riemannian variety. This is detailed in the last section of this article (see Section \ref{parab_canonique}). This is a practical application of Theorem \ref{Morozov_p_sup} (or more precisely of one of its reformulation). Let us briefly expose the context: let $k$ be a field and $C$ be a projective smooth and geometrically connected $k$-curve. Let also $G$ be a reductive $C$-group which is the twisted form of a constant reductive $C$-group $G_0$ (namely $G_0$ is obtained from a reductive $k$-group by base change). In other words there is a $G_0$-torsor $E$ for which $G= \prescript{E}{}{G_0}$. When $k$ is of characteristic $0$, M. Atiyah and R. Bott consider the Harder--Narasimhan filtration of the Lie algebra $\mathfrak{g}$ (seen as a vector bundle over $C$):
\[0 \subsetneqq E_{-r} \subsetneqq \hdots \subsetneqq E_{-1} \subsetneqq E_{0} \subsetneqq E_{1} \subsetneqq \hdots \subsetneqq E_{l} = \mathfrak{g},\] 
where the indices are chosen in such a way that the term $E_0$ is the one for which the semistable quotient is of slope $0$. The authors show in \cite[\S 10]{AB} that $E_0$ is a parabolic subalgebra of $\mathfrak{g}$ and they call ``canonical parabolic subgroup of $G$'' the subgroup from which it derives. If now $k$ is any field and $G$ is a reductive $C$-group, K. A. Behrend extends in \cite{BEH} the notion of canonical parabolic subgroup when $k$ is any field. Its definition does no longer involves the Harder--Narasimhan filtration of $\mathfrak{g}$. In \cite[Proposition 3.4]{MS}, V. B. Mehta and S. Subramanian show that the definition given by M. Atiyah and R. Bott (in characteristic $0$) still makes sense in the aforementioned framework when $k$ is of characteristic $p> \max(2 \dim(G), 4(h(G)-1))$\footnote{The original result is stated with the height of $G$ which is equal to the Coxeter number of $G$ minus $1$.}. See also \cite[Theorem 2.6]{Meh} for a refinement of the statement, that only requires $p>2 \dim(G)$. Their proof relies on techniques coming from representation theory. Statement \ref{analogue_morozov} is true when $p>\h(G)$ (see Theorem \ref{Morozov_p_sup}) and allows to obtain a construction of the canonical parabolic subgroup of $G$ when $p> 2\dim(G)-2$ that mimics exactly Atiyah--Bott's construction in characteristic $0$ (see \cite{AB}), and to compare this construction with Behrend's definition of the canonical parabolic subgroup. In section \ref{parab_canonique} we make use of Corollary \ref{cor_killing} to show:

\begin{proposition}
If $k$ is of characteristic $0$ or $p > 2 \dim G - 2$ and if $G$ is endowed with a non degenerate $G$-equivariant symmetric bilinear form then: 
\begin{enumerate}
\item the subbundle $E_0 \subseteq \mathfrak{g}$ is the Lie algebra of a parabolic subgroup $Q \subseteq G$ and $E_{-1}$ is the Lie algebra of its unipotent radical (denoted by $\radu(Q)$), 
\item this parabolic subalgebra satisfies the equalities $E_0 = \mathfrak{p_g^{\can}}$ and $E_{-1} = \radu(P_G^{\can})$, where $P_G^{\can}$ is the canonical parabolic subgroup of $G$ and $\mathfrak{p_g^{\can}}$ is its Lie algebra. Let $K$ be the function field of $C$ and $\bar{K}$ its algebraic closure. Over the geometric generic fiber we have that $(E_{0})_{\bar{K}} = \mathfrak{p_g}(\lambda_{{(E_{-1})}_{\bar{K}}}),$
\item the subbundles $E_{i}$ are $p$-nil $p$-ideal of $E_0 = \mathfrak{p_g^{\can}}$ for any $-r \leq i < 0$.
\end{enumerate}
\label{E0_canonique}
\end{proposition}
Let us finally mention two other works that provide better bounds for the coincidence of both approaches:
\begin{itemize}
\item A. Langer shows that the definitions of canonical parabolic subgroup are the same when the $G_0$-torsor $E$ admits a strong Harder--Narasimhan filtration (see \cite[Definition 3.1 and Proposition 3.3]{LAN}),
\item I. Biswas and Y. I. Holla show in \cite{BH} that the approaches coincide when $p>\h(G)$.
\end{itemize}  
Some results exposed here come from the author's Ph.D. manuscript (see \cite{JEAthese}).

\section{Prerequisites}
\label{Prerequisites}
\subsection{Context}
	\label{hypothèses_car}
	
In this section and unless otherwise stated $k$ is a field of characteristic $p>0$ and $G$ is a reductive $k$-group. Let $T\subset G$ be a maximal torus, denote by $W= W(G,T)$ the Weyl group and $\Phi=\Phi(G,T)$ the associated root system. Let then $\Phi^{+}$ be a positive root system and $\Delta:= \{\alpha_1, \cdots, \alpha_n \}$ be the corresponding root basis that defines the corresponding Borel subgroup $T \subset B \subset G$.

Assume for the moment that $\Phi$ is irreducible. Let $\alpha := \sum_{i=1}^{n}a_i\alpha_i$ be its highest root. The Coxeter number of $G$, denoted by $\h(G)$, is the height of $\alpha$ plus $1$, whence the equality $\h(G) = \sum_{i=1}^{n} a_i +1$.

For any root $\beta \in \Phi$, denote by $\beta^{\vee}$ the corresponding coroot. Let $\alpha^{\vee}:= \sum_{i=1}^{n} b_i \alpha_i^{\vee}$ be the coroot associated to the highest root. The characteristic $p$ of $k$ is:
	\begin{itemize}
		\item \textit{of torsion} for $\Phi$ if there exists $i \in [1,\cdots,n]$ such that $p$ divides $b_i$;
		\item \textit{bad} for $\Phi$ if there exists an integer $i \in [1, \cdots, n]$ such that $p$ divides $a_i$;
		\item \textit{good} for $\Phi$ if its not bad for $\Phi$;
		\item \textit{very good} for $\Phi$ if it is good and if $p$ does not divide $n+1$ when $\Phi$ is of type $A_n$.
	\end{itemize}

\noindent For a more detailed discussion on torsion integers for $G$ see for example \cite{Stei_prime}, good characteristics are for instance studied in \cite[\S 0.3]{SPR}. The values of torsion, good and very good integers are listed in the table below, the rank of the algebraic $k$-group $G$ is also mentioned. The latter, denoted by $\rg(G)$, is equal to the dimension of a maximal torus of $G$ (which is the same for any maximal torus of $G$, see \cite[XIX 1.5]{SGA33}).	
	
\begin{table}[H]
\begin{center}
    \begin{tabular}{ |c| P{2.8cm}| P{2.8cm} |P{2.8cm} | c | c | }
    \hline
    Type of $G$ & Torsion integers & Good characteristics & Very good characteristics & $\rg(G)$ & $\h(G)$ \\ \hline
    $A_n$ & $1$ & any & $p \nmid n+1$ & $n$ & $n$ \\ \hline 
    $B_n$, with $n\geq 2$ & $2$ & $>2$ & $>2$ &  $n$ & $2n$ \\ \hline
    $C_n$, with $n \geq 2$ & $1$ & $>2$ &$>2$ & $n$ & $2n$ \\ \hline
    $D_n$, with $n \geq 3$ & $2$ & $>3$ &$>3$ & $n$ & $2(n-1)$ \\ \hline
    $E_6$ & $3$ & $>3$ &$>3$ & $6$ & $12$ \\ \hline
    $E_7$ & $4$ & $>3$ & $>3$ & $7$ & $18$ \\ \hline
    $E_8$ & $6$ & $>5$ & $>5$ & $8$ & $30$ \\ \hline
    $F_4$ & $3$ & $>3$ & $>3$ & $4$ & $12$ \\ \hline
    $G_2$ & $2$ & $>3$ & $>3$ & $2$ & $6$ \\ \hline

    \end{tabular}
\end{center}
 \caption[Table caption text]{Assumptions on the characteristic for a simple group.}	
 \label{Hypothèses_sur_la caractéristiqu_ pour_un_groupe_simple}
	\end{table}

\noindent Note that when the characteristic is very good for $G$ the Lie algebra of $G$ is endowed with a non degenerate Killing form (see for example \cite[I, 5.3]{SprSte}, \cite[1.16]{CAR}). 

When $G$ is any reductive group its root system $\phi$ is no longer irreducible. In this setting the Coxeter number of $G$ is the highest Coxeter number of the irreducible components of $\phi$. The characteristic $p$ of $k$ is:
	\begin{itemize}
		\item \textit{of torsion} if $p$ is of torsion for one irreducible component of $\Phi$ or if $p$ divides the order of the fundamental group of $G$; 
		\item \textit{bad} for $G$ if it is bad for one irreducible component of $\Phi$; 
		\item \textit{good}, respectively \textit{very good}, for $G$ if it is good, respectively very good, for each irreducible component of $\Phi = \Phi(G,T)$.
	\end{itemize}
	
Table \ref{Hypothèses_sur_la caractéristiqu_ pour_un_groupe_simple} allows to check that if the characteristic is very good for $G$ then it is not of torsion for $G$. Moreover, any $p >\h(G)$ is very good for $G$. 

Finally, we remind of the reader that a reductive $k$-group $G$ satisfies \textit{the standard hypotheses} (see \cite[2.9]{Jantzen2004}) if:
\begin{enumerate}
		\item[(H1)] the derived group of $G$ is simply connected, 
		\item[(H2)] the characteristic of $k$ is good for $G$,
		\item[(H3)] there exists a non degenerate $G$-equivariant symmetric bilinear form over $\mathfrak{g}$.
	\end{enumerate} 
\noindent Let us recall that according to \cite[I, \S8, Corollary]{SEL}, if $\mathfrak{g}$ is endowed with a non degenerate $G$-equivariant symmetric bilinear form, the Lie algebra is semisimple (i.e. its solvable radical is trivial, see \cite[1.7, Definition]{FS}). Let us remark that any  simple and simply connected $k$-group $G$ for which $k$ is of very good characteristic satisfies the standard hypotheses. Let us finally note that J. C. Jantzen provides:
\begin{itemize}
\item some explicit characterisation for Lie algebras of reductive groups that satisfy the standard hypotheses (see \cite[2.9]{Jantzen2004}), 
\item some criteria for adjoint reductive groups to satisfy these (see \cite[B.6]{JANtokyo}). 
\end{itemize} 

Let $k$ be an algebraically closed field of characteristic $p>0$ and $G$ be a semisimple $k$-group. When $p$ is not of torsion for $G$ the following corollary can be deduced from \cite[Theorem 2.2 and Remark a)]{LMT} :

\begin{corollary}[of {\cite[Theorem 2]{LMT}}, see {\cite[Remark 2.2]{JEA1}}]
Let $k$ be an algebraically closed field of characteristic $p>0$ and $G$ be a reductive $k$-group. Assume that $p$ is not of torsion for $G$. Let $\mathfrak{u} \subset \mathfrak{g}$ be a restricted $p$-nil $p$-subalgebra. Then $\mathfrak{u}$ is a subalgebra of the Lie algebra of the unipotent radical of a Borel subgroup $B \subseteq G$.
\label{corollaire_LMT}
\end{corollary}

Assume now that $G$ is a reductive $k$-group. When $G$ is semisimple the characteristic of $k$ is separably good for $G$ if:
\begin{enumerate}
	\item the integer $p$ is separably good for $G$,
	\item the morphism $G^{\Sc} \rightarrow G$, where $G^{\Sc}$ is the simply connected cover of $G$, is separable.
\end{enumerate}
When $G$ is a reductive group, the integer $p$ is separably good if it is separably good for the derived group $[G,G]$ (see \cite[Definition 2.2]{PevSta}). As underlined in the paragraph that follows \cite[Definition 2.2]{PevSta}, if $p$ is really good for $G$ then it is separably good. This last condition is nevertheless less restrictive when $G$ is of type $A$, which is the only type for which good, separably good, and very good integers do not coincide. As an example $p$ is separably good but not very good for $\SL_p$ or $\GL_p$. However, it is not separably good nor very good for $\PGL_p$. 
%
We briefly show here two technical lemmas that will allow us to compare the different bounds on the characteristic in what follows. Here $k$ is assumed to be algebraically closed.

\begin{lemma}
Let $d$ be the dimension of a minimal faithful representation of $G$. Then one has $d > \rg(G)$.
\label{dim_et_rang}
\end{lemma}

\begin{proof}
Let $\rho : G \rightarrow \GL_n$ be a minimal faithful representation of $G$. The representation $\rho$ is also a faithful representation of a maximal torus $T \subseteq G$. If we denote by $n$ the dimension of $T$ then one necessary has $d\geq n= : \rg(G)$.
\end{proof}

%

\begin{lemma}
The inequality $2\h(G)-2 \leq 2^{2d}$ holds true for any reductive $k$-group $G$. 
\label{borne_p_grand}
\end{lemma}

\begin{proof}
When $G$ is simple the values of $\h(G)$ and $\rg(G)$ are listed in Table \ref{Hypothèses_sur_la caractéristiqu_ pour_un_groupe_simple} and one can check that the inequalities $2\h(G)-2 < 2^{2\rg(G)} < 2^{2d}$ are indeed satisfied (according to Lemma \ref{dim_et_rang}).

Let us denote by $G_i$ the simple groups defined by the irreducible components of $\Phi$. The restriction of a faithful representation of $G$ to $G_i$ is still a faithful representation for $G_i$. If now $d_i$ is the dimension of a minimal faithful representation of $G_i$ one has $d \geq d_i \geq \h(G_i)$ for any $G_i$ according to what precedes. As there exists an integer $i$ for which $\h(G) = \h(G_i)$ the inequalities $2\h(G)-2 \leq 2^{2\sum_{i=1}^m d_i} \leq 2^{2d}$ then hold true, whence the expected result. 
\end{proof}
 
\subsection{Parabolic subalgebras}
\label{ssalg_parab}

A parabolic subalgebra $\mathfrak{p} \subseteq \mathfrak{g}$ is the Lie algebra of a parabolic subgroup $P \subseteq G$. This is equivalent to requiring the existence of a Cartan subalgebra $\mathfrak{t}\subset \mathfrak{g}$ such that $\mathfrak{p} = \mathfrak{t} \oplus \bigoplus_{\alpha \in \Phi'} \mathfrak{g}_{\alpha}$, where the subset $\Phi' \subseteq \Phi$ is a parabolic subset of $\Phi$ (see \cite[XXVI, Proposition 1.4]{SGA33}). Note that this description allows us to naturally define the Coxeter number of a parabolic subgroup of $G$: it is the height plus one of the highest root of the associated parabolic subset. In particular one has $\h(G) = \h(B)$ for any Borel subgroup $B\subset G$.
\noindent When $k$ is of characteristic $0$ these conditions are satisfied if and only if $\mathfrak{p}$ contains the Lie algebra of a Borel subgroup of $G$. When $k$ is a field of characteristic $p>0$ such a characterisation does no longer hold true: a parabolic subalgebra always contains a Borel subalgebra (because a parabolic subgroup always contains a Borel subgroup), but there are some pathological characteristics for which one can find Lie subalgebras of $\mathfrak{g}$ that contain a Borel subalgebra but do not derive from a parabolic subgroup of $G$. 

\begin{proposition}
If $k$ is of characteristic $p \neq 2, 3$, a subalgebra $\mathfrak{p} \subseteq \mathfrak{g}$ is parabolic if and only if it contains a Borel subalgebra.
\label{parabolic_algebra}
\end{proposition}

\begin{proof}
Let $\mathfrak{p} \subseteq \mathfrak{g}$ be a Lie subalgebra that contains a Borel subalgebra $\mathfrak{b}$. One needs to show that $\mathfrak{p}$ is of the form $\mathfrak{t} \oplus \bigoplus_{\alpha \in \Phi'} \mathfrak{g}_{\alpha}$ with $\mathfrak{t}$ and $\Phi'$ as in the preamble of this section. As $\mathfrak{p}$ contains a Borel subalgebra $\mathfrak{b}$, it contains the Lie algebra of a maximal torus $\mathfrak{t}$. This Cartan subalgebra acts on $\mathfrak{p}$ via the Lie bracket. The roots of $\mathfrak{t}$ define a family of endomorphisms of $\mathfrak{t}$ that are diagonalisable and that commute with each other. They are therefore simultaneously diagonalisable and $\mathfrak{p}$ decomposes as a sum of weight spaces for this action. Namely one has $\mathfrak{p} = \mathfrak{t} \oplus \bigoplus_{\alpha \in \Phi'} \mathfrak{g}_{\alpha}$. It remains to show that:
	\begin{itemize}
		\item the subset $\Phi' \subseteq \Phi$ contains a positive root system. Which is immediate because $\mathfrak{p}$ contains the Lie algebra of a Borel subgroup  by assumption,
		\item the subset $\Phi' \subseteq \Phi$ is of type $(R)$ (see \cite[XXII, Définition 5.4.2]{SGA33}). According to \cite[XXII, Théorème 5.4.7]{SGA33} one needs to remark that $\Phi'$ is closed. In other words, for any roots $\alpha, \ \beta \in \Phi'$ such that $\alpha + \beta \in \Phi$ one needs to show that $\alpha + \beta \in \Phi'$. This last condition is satisfied as soon as $p \neq 2,3$ according to \cite[XXIII, Corollaire 6.6]{SGA33}. 
	\end{itemize}
\end{proof}

\begin{remark}
If $G$ is of type $A_1$, one only needs to require $p \neq 2$ in Proposition \ref{parabolic_algebra}.
\end{remark}

\begin{exs}
We provide here examples of Lie subalgebras that contain a Borel subalgebra and do not derive from a parabolic subgroup of $G$.
\begin{enumerate}
	\item When $G = \PGL_3$ and $k$ is of characteristic $3$, let us consider the Lie subalgebra 
	\[\mathfrak{p}:= \left\lbrace \begin{pmatrix}
a & b & c \\
d & e & f  \\
0 & td & g
\end{pmatrix} \mid d, t \in k^{\times}\right\rbrace \subseteq \mathfrak{pgl}_3(k).\] It contains the Borel subalgebra of upper triangular matrices with trace $0$. Notwithstanding this, the Lie algebra $\mathfrak{p}$ cannot be written as the Lie algebra of a parabolic subgroup $P \subseteq \PGL_3$ (as this can be checked when $k$ is of characteristic $0$ because no parabolic subgroup admits $\mathfrak{p}$ as a Lie algebra).
	\item We still assume that $G = \PGL_3$ and $k$ is an algebraically closed field of characteristic $p=3$. We consider the example of \cite{LMT}: we set $X = \begin{pmatrix}
		0 & 1 & 0 \\
		0 & 0 & 1 \\
		0 & 0 & 0	
	\end{pmatrix}
	\text{ and } Y = \begin{pmatrix}
		0 & 0 & 0 \\
		1 & 0 & 0 \\
		0 & 2 & 0
		\end{pmatrix}.$
Let $\mathfrak{u} := \left< X,Y \right> \subsetneq \mathfrak{pgl}_3$ be the subalgebra generated by $X$ and $Y$. It is an $\ad$-nilpotent subalgebra of $\mathfrak{pgl}_3$. Its normaliser is given by $\N_{\mathfrak{pgl}_3}(\mathfrak{u}) = \left\lbrace
\begin{pmatrix}
a & b & c \\
d & e & b \\
g & -d & i
\end{pmatrix}\right\rbrace$ that contains a Borel subalgebra, namely the Lie algebra of the Borel subgroup obtained by conjugating the Borel subgroup of upper triangular matrices by $\begin{pmatrix}
0 & 1 & 0 \\
1 & 0 & 0 \\
0 & 0 & 1
\end{pmatrix}$. However $\mathfrak{p}$ does not derive from a parabolic subgroup of $G$.
\end{enumerate}
\end{exs}

\subsection{Lie algebras and normalisers}
The formalism used in this section is developed in \cite[II, \S 4]{DG}. We especially refer the reader to \cite[II, \S 4, 3.7]{DG} for notations. Let $A$ be a ring and $G$ be an affine $A$-group functor. As a reminder: 
	
\begin{enumerate}
	
	\item if $R$ is an $A$-algebra, we denote by $R[t]$ the algebra of polynomials in $t$ and by $\epsilon$ the image of $t$ via the projection $R[t] \rightarrow R[t]/(t^2)=:R[\epsilon]$.
\noindent We associate to $G$ a functor in Lie algebras, denoted by $\mathfrak{Lie}(G)$, which is the kernel of the following exact sequence: 
	
	\begin{figure}[H]
	\begin{center}
\[\begin{tikzpicture} 

 \matrix (m) [matrix of math nodes,row sep=2em,column sep=1.8em,minimum width=2em,  text height = 1.5ex, text depth = 0.25ex]
  {
    1 & \mathfrak{Lie}(G)(R)& G(R[\epsilon]) & G(R) & 1.\\};
  \path[-stealth]
  	(m-1-1) edge (m-1-2)
    (m-1-2) edge (m-1-3) 
   	(m-1-3) edge node [below] {$p$}(m-1-4) 
   	(m-1-4) edge (m-1-5) 
   	(m-1-4) edge[bend left = -50]
   			node [above] {$i$}(m-1-3);
\end{tikzpicture}\]
\end{center}
\end{figure}
For any $y \in \mathfrak{Lie}(G)(R)$ we denote by $\e^{\epsilon y}$ the image of $y$ in $G(R[\epsilon])$. In what follows the notation $\mathfrak{Lie}(G)(R)$ refers both to the kernel of $p$ and to its image in $G(R[\epsilon])$. The Lie-algebra of $G$ is $\mathfrak{Lie}(G)(A)$ and is denoted by $\Lie(G) := \mathfrak{g}$. According to \cite[II, \S 4, n\degree 4.8, Proposition]{DG} when $G$ is smooth or when $A$ is a field and $G$ is locally of finite presentation over $A$, the equality 
\[\Lie(G) \otimes_A R =  \mathfrak{Lie}(G)(A)\otimes_A R = \mathfrak{Lie}(G)(R) = \Lie(G_R)\] holds true for any $A$-algebra $R$ (these are sufficient conditions). When the aforementioned equality is satisfied the $A$-functor $\mathfrak{Lie}(G)$ is representable by $W(\mathfrak{g})$, where for any $A$-module $M$ and any $A$-algebra $R$ we set $W(M)(R) := M \otimes_A R$;
\item the $A$-group functor $G$ acts on $\mathfrak{Lie}(G)$ as follows: for any $A$-algebra $R$ the induced morphism is the following:
\begin{alignat*}{3}
\Ad_R : \: & G_R\: &\rightarrow \: & \Aut(\mathfrak{Lie}(G))(R),\\ 
\: & g \: & \mapsto \: & \Ad_R(g) : \mathfrak{Lie}(G)(R) \rightarrow \mathfrak{Lie}(G)(R) : x \mapsto i(g)xi(g)^{-1}.
\end{alignat*}
When $G$ is smooth (in particular when $\mathfrak{Lie}(G)$ is representable) the $G$-action on $\mathfrak{Lie}(G)$ defines a linear representation $G \rightarrow \GL(\mathfrak{g})$ (see \cite[II, \S 4, n\degree 4.8, Proposition]{DG}).
\end{enumerate}
With the above notations, given:
\begin{enumerate}
\item a closed subgroup $H\subseteq G$, the normaliser of $H$ in $G$ is the closed subscheme of $G$ defined for any $k$-algebra $A$ by
\[N_G(H)(A) := \{g\in G(A) \mid \Ad(g)(H(R)) =H(R) \text{ for any } A-\text{algebra } R.\}\]
\item a subalgebra $\mathfrak{h}\subseteq \mathfrak{g}$,
 \begin{enumerate}
 \item the normaliser of $\mathfrak{h}$ in $G$ is the closed subscheme of $G$ defined for any $k$-algebra $A$ by
\[N_G(\mathfrak{h})(A) := \{g\in G(A) \mid \Ad(g)(\mathfrak{h}(R)) =\mathfrak{h}(R) \text{ for any } A-\text{algebra } R.\}\]
\item the normaliser of $\mathfrak{h}$ in $\mathfrak{g}$ is the closed subscheme of $\mathfrak{g}$ defined for any $k$-algebra $A$ by
\[N_{\mathfrak{g}}(\mathfrak{h})(A) := \{g\in \mathfrak{g}(A) \mid \ad(g)(\mathfrak{h}(R)) =\mathfrak{h}(R) \text{ for any } A-\text{algebra } R.\}\]
\end{enumerate}
\end{enumerate} 

The following result is a quick reminder of useful facts for the proofs that follow:
\begin{lemma}
Let $G$ be a reductive $k$-group. Then:
\begin{enumerate}
\item let $d$ be the dimension of a minimal faithful representation for $G$. If $p>2^{2d}$ then all normalisers of subspaces of $\mathfrak{g}$ are smooth (see \cite[Theorem A]{HS});
\item consider a subset $\mathfrak{h}\subset \mathfrak{g}$, then one has $\Lie(N_{G}(\mathfrak{h})) = N_{\mathfrak{g}}(\mathfrak{h})$ (see \cite[Exercise 2.4.8]{C} and for instance \cite[Lemma 6.4]{JEA1} for a proof; note that the additional assumption of $\mathfrak{h}$ being a subalgebra required in the statement is not used in the proof);
\item let $H\subseteq G$ be a closed subgroup then: \begin{enumerate}
	\item the inclusion $N_G(H) \subseteq N_G(\mathfrak{Lie}(H))$ is satisfied. In particular if $H$ is smooth this leads to the inclusion $N_G(H)(R) \subseteq N_G(\mathfrak{h}_R)$ for any $A$-algebra $R$ (see \cite[Lemma 6.3]{JEA1});

	\item if moreover $H\subseteq G$ is smooth, then:
\begin{itemize}
		\item in general only the inclusion $\Lie(N_G(H)) \subseteq N_{\mathfrak{g}}(\mathfrak{h})$ holds true,
		\item if $H(k)$ is Zariski-dense in $H$ then
\[\Lie(N_G(H)) = \lbrace x \in \mathfrak{g} \mid \Ad(h)(x) - x \in \mathfrak{h} \ \forall h \in H(k)\rbrace.\]
\end{itemize} 
(see \cite[Lemma 6.4]{JEA1}).
\end{enumerate}
\end{enumerate}
\label{reminder_normaliser}
\end{lemma}

\subsection{Radicals of a restricted \texorpdfstring{$p$}{Lg}-Lie algebra}
\label{radicaux_p_alg_restreinte}
\begin{lemma}
Let $k$ be a perfect field of characteristic $p\geq 3$ and let $G$ be an algebraic $k$-group. Assume that $G$ is an extension of a $k$-group of multiplicative type $S$ by a smooth connected algebraic $k$-group $H$ such that $\iota(\Rad_U(H))$ is a normal subgroup of $G$:  \begin{figure}[H]
\begin{center}
\[\begin{tikzpicture} 

 \matrix (m) [matrix of math nodes,row sep=2em,column sep=4.8em,minimum width=2em,  text height = 1.5ex, text depth = 0.25ex]
  {
    1 & H & G & S & 1.\\
   };
  \path[-stealth]
  	(m-1-1) edge (m-1-2)
    (m-1-2) edge node[above]{$\iota$} (m-1-3) 
    (m-1-3) edge node[above]{$\pi$}(m-1-4) 
    (m-1-4) edge (m-1-5)
   ;
    	
\end{tikzpicture}\]
\end{center}
\end{figure}
\noindent Then the $p$-radical of $\mathfrak{g}$ identifies with the set of all $p$-nilpotent elements of $\rad(\mathfrak{g})$. Namely one has $\rad_p(\mathfrak{g})= \{x \in \rad(\mathfrak{g}) \mid \ x \text{ is } p\text{-nilpotent}\}= \Lie(\iota)(\rad_p(\mathfrak{h}))$.
\label{p_rad_inf_sat}
\end{lemma}

\begin{proof}
One inclusion is clear: the $p$-radical of $\mathfrak{g}$ is a $p$-nil $p$-ideal of $\rad(\mathfrak{g})$. Thus $\rad_p(\mathfrak{g})$ is contained in the set of all $p$-nilpotent elements of $\rad(\mathfrak{g})$ (see \cite{JEA1}[Lemma 2.5 ii)]).

The derived sequence:
\begin{figure}[H]
\begin{center}
\[\begin{tikzpicture} 

 \matrix (m) [matrix of math nodes,row sep=2em,column sep=4.8em,minimum width=2em,  text height = 1.5ex, text depth = 0.25ex]
  {
    0 & \mathfrak{h} & \mathfrak{g} & \mathfrak{s} & 0,\\
   };
  \path[-stealth]
  	(m-1-1) edge (m-1-2)
    (m-1-2) edge node[above]{$\Lie(\iota)$}(m-1-3) 
    (m-1-3) edge node[above]{$\Lie(\pi)$} (m-1-4) 
    (m-1-4) edge (m-1-5)
   ;
    	
\end{tikzpicture}\]
\end{center}
\end{figure}
\noindent is still exact as $H$ is smooth and the morphism $G\rightarrow S$ is surjective (see \cite[II, \S 5, Proposition 5.3]{DG}). This is moreover an exact sequence of $p$-morphisms of restricted $p$-Lie algebras (see \cite[II, \S7 n\degree 2.1 and n\degree 3.4]{DG}). 

Let $x$ be a $p$-nilpotent element of $\rad(\mathfrak{g})$. The image $\Lie(\pi)(x)$ is a $p$-nilpotent element of $\mathfrak{s}$ which is toral because it is the Lie algebra of a group of multiplicative type. The $p$-power map is therefore injective over $\mathfrak{s}$, hence the vanishing of $\Lie(\pi)(x)=0$. The exactness of the derived sequence ensures that $x$ belongs to the image of $\Lie(\iota)$. Thus there exists an element $z \in \mathfrak{h}$ such that $x = \Lie(\iota)(z)$. As $\iota$ is a $p$-morphism of restricted $p$-Lie algebras and $x$ is $p$-nilpotent (by assumption), there exists an integer $m$ such that
\[0 = x ^{[p^m]} = \iota(z)^{[p^m]} = \iota(z^{[p^m]}).\] 
This implies that $z^{[p^m]}=0$, so that $z$ is $p$-nilpotent because $\iota$ is injective. The element $x$ thus belongs to the set of $p$-nilpotent elements of $\rad(\mathfrak{g}) \cap \im(\Lie(\iota))$ which is:
\begin{itemize}
\item a solvable ideal of $\im(\Lie(\iota))$,
\item contained in  $\rad(\im(\Lie(\iota))(\mathfrak{h})\cong \rad(\mathfrak{h})$.
\end{itemize}
In other words, any $p$-nilpotent element of $\rad(\mathfrak{g})$ is the image of a $p$-nilpotent element of $\rad(\mathfrak{h}) = \rad_p(\mathfrak{h}) = \radu(H)$ according to \cite[Lemma 2.13]{JEA1} (and \cite[Remark 2.10]{JEA1} for the characteristic $2$ case).

By assumption, the image of $\radu(H) = \rad_p(\mathfrak{h})$ is an ideal of $\mathfrak{g}$. It is a $p$-nil $p$-ideal of $\mathfrak{g}$ because $\Lie(\iota)$ is a $p$-morphism of restricted $p$-Lie algebras. In other words the inclusion $\Lie(\iota)(\radu(H)) \subseteq \rad_p(\mathfrak{g})$ holds true.

To summarise, we have shown the following inclusions \[\{x \in \rad(\mathfrak{g}) \mid \ x \text{ is } p\text{-nilpotent}\} \subseteq \Lie(\iota)(\radu(H)) \subseteq \rad_p(\mathfrak{g}),\] whence the desired equality.
\end{proof}

\begin{remark}
If $G$ is a reductive $k$-group such that $\mathfrak{g}$ is endowed with a non degenerate $G$-equivariant symmetric bilinear form, then $\mathfrak{g}$ is semisimple according to \cite[I, \S8, Corollary]{SEL}: its radical (thus its center) is trivial. In other words, the adjoint representation $\ad : \mathfrak{g} \rightarrow \mathfrak{\gl(g)}$ is faithful. Therefore, under this assumption any $\ad$-nilpotent element is $p$-nilpotent. In particular the equality of $p$-ideals $\nil(\mathfrak{h}) = \rad_p(\mathfrak{h})$ holds true for any restricted $p$-Lie subalgebra $\mathfrak{h} \subseteq \mathfrak{g}$.
\label{semisimple_rad_p_nil}
\end{remark}

\begin{lemma}
Let $k$ be an algebraically closed field of characteristic $p>0$. Let $G$ be a reductive $k$-group such that $\mathfrak{g}$ is endowed with a non degenerate $G$-equivariant symmetric bilinear form $\kappa$, and let $\mathfrak{b} \subseteq \mathfrak{g}$ be a Borel subalgebra. Then one has $\nil(\mathfrak{b}) = \mathfrak{b}^{\perp}$.
\label{nil_borel_orthogonal}
\end{lemma}

\begin{proof}
The inclusion $\nil(\mathfrak{b})\subseteq \mathfrak{b}^{\perp}$ is clear according to \cite[\S 4, n\degree 4, Proposition 6]{BOU1}.

Let us thus show the converse inclusion. According to \cite[Remark 2.14]{JEA1} the $p$-radical of $\mathfrak{b}$ is the Lie algebra of the unipotent radical of $B$ (Table \ref{Hypothèses_sur_la caractéristiqu_ pour_un_groupe_simple} allows to check that requiring the existence of $\kappa$ leads to avoid pathological cases). According to the third point of \cite[Lemma 2.5]{JEA1} this $p$-radical is the nilradical of $\mathfrak{b}$ if and only if $\mathfrak{z_b} \subseteq \mathfrak{b}$. This inclusion holds true here: indeed, we assumed that $\mathfrak{g}$ was endowed with a non degenerate $G$-equivariant symmetric bilinear form $\kappa$. The Lie algebra $\mathfrak{g}$ is therefore semisimple (see \cite[I, \S8, Corollary]{SEL}). Thus, in particular, its center is trivial (see \cite[Corollary 5.6]{FS}). Then, by making use of the inclusions $Z^0_G \subseteq Z_B \subseteq Z_G$ (see for instance \cite[Corollary 11.11]{BORlag}), one obtains by derivation that the Lie algebra $\mathfrak{z_b}$ is trivial. This is so because $\Lie(Z_B) = \mathfrak{z_b}$ according to \cite[Remark 6.2]{JEA1}.

To summarise, we have shown that $\nil(\mathfrak{b}) = \rad_p(\mathfrak{b}) = \Lie(\Rad_U(B)) = \bigoplus_{\alpha \in \Phi^+} \mathfrak{g}^{\alpha}$, where:
\begin{itemize}
\item we denote by $\Phi^+$ a positive root system associated to $B$. 
\item the first equality comes from \cite[Lemma 2.5]{JEA1} iii), 
\item the second equality comes from \cite[Remark 2.14]{JEA1}.
\end{itemize} 
Moreover, as $\mathfrak{g}^{\alpha}$ is orthogonal to $\mathfrak{g}^{\mu}$ as soon as $\mu \neq - \alpha$, one necessarily has $\mathfrak{b}^{\perp} \subseteq \nil(\mathfrak{b})$, whence the desired equality. 
\end{proof}

\section{First steps}
\label{sections_premiers_resultats}
Unsurprisingly under some assumptions that allow to reproduce the characteristic $0$ framework, an analogue of Morozov's Theorem can be easily shown to hold true in characteristic $p>0$. Especially the characteristic $p$ has to be quite big. These results are already instructive as they underline obstructions specific to the positive characteristic setting. 

	\subsection{When all the normalisers are smooth}
When $k$ is an algebraically closed field of characteristic $0$ or $p>0$ really high, Statement \ref{analogue_morozov} turns out to be true as it is an immediate application of Veisfeiler--Borel--Tits Theorem (see \cite[corollaire 3.2]{TB}). More precisely: 

\begin{proposition}
Let $k$ be an algebraically closed field and $G$ be a reductive $k$-group. We denote by $d$ the dimension of a minimal faithful representation of $G$. Assume that $k$ is of characteristic $0$ or $p > 2^{2d}$. Finally, let $\mathfrak{u}\subset \mathfrak{g}$ be a restricted $p$-nil $p$-subalgebra such that $\mathfrak{u}=\rad_p(N_{\mathfrak{g}}(\mathfrak{u}))$. Then $N_{\mathfrak{g}}(\mathfrak{u})$ is a parabolic Lie subalgebra of $\mathfrak{g}$.
\label{morozov_Herpel_Stewart}
\end{proposition}

\begin{remark}
When $G = \GL_n$ the assumptions of Proposition \ref{morozov_Herpel_Stewart} amount to require $p> 2^{2n}$.
\end{remark}

\begin{proof}
Assumptions of the proposition warrant the existence of an exponential map, allowing:
\begin{itemize}
	\item to integrate the nil Lie subalgebra $\mathfrak{u}$ into a unipotent smooth connected subgroup $U \subseteq G$ such that $\Lie(U) = \mathfrak{u}$,
	\item to make sure that this integration is compatible with the adjoint representation, implying in particular that the normalisers $N_G(U)$ and $N_G(\mathfrak{u})$ are the same (according to \cite[Lemma 4.3]{HS}).	 
\end{itemize}  
\noindent This is detailed:
 \begin{itemize}
		\item for instance in \cite[II, \S6, Corollary 3.4]{DG} when $k$ is of characteristic $0$,
		\item in \cite[\S 6]{BDP} and \cite[Proposition 3.1]{JEA1} when $p>2^{2d}$, for the existence of an integration for $\mathfrak{u}$ (because $p>2\h(G)-2 \geq \h(G)$ according to Lemma \ref{borne_p_grand}), and in \cite[Lecture 4, Theorem 5]{SER} and \cite[\S 4.6]{BDP} for the compatibility between the truncated exponential map and the adjoint representation (once again one can apply these results because $p>2\h(G)-2$). For this last point see also Remark \ref{rem_car_0_ou_grande} below.
	\end{itemize}
Let $V$ be the unipotent radical of $N_G(U)^0$. Note that it is well defined because $N_G(U)$ is smooth. Indeed one has $N_G(U)= N_G(\mathfrak{u})$, and the latter is smooth under the assumptions of the proposition. Indeed, this is
\begin{itemize}
\item well known in characteristic $0$ according to P. Cartier' theorem which states that every affine group over a field of characteristic zero is smooth, 
\item ensured by \cite[Theorem A]{HS} when $p> 2^{2d}$.
\end{itemize}

The inclusion $U \subseteq V$ is clear because $U$ is a unipotent, smooth, connected normal subgroup of $N_G(U)$. So one has $\mathfrak{u} \subseteq \mathfrak{v}$. Moreover the equalities of Lie algebras $\Lie(N_G(U)) = \Lie(N_G(\mathfrak{u})) = N_{\mathfrak{g}}(\mathfrak{u})$ are satisfied according to Lemma \ref{reminder_normaliser}. As one has $V:= \Rad_U(N_G(U)^0_{\red})$, the subalgebra $\mathfrak{v}:= \Lie(V)$ is a $p$-nil $p$-ideal of $\mathfrak{g}$ (see \cite[Lemma 2.12]{JEA1}). By assumption $\mathfrak{u}$ is the set of all $p$-nilpotent elements of $\rad(N_{\mathfrak{g}}(\mathfrak{u}))$, whence the inclusion $\mathfrak{v} \subseteq \mathfrak{u}$. So we have shown that $\mathfrak{u} = \mathfrak{v}$. Moreover, as $U$ and $V$ are smooth connected subgroups of $G$, this last equality of Lie algebras allows to conclude that $U=V$ (see \cite[II, \S5, 5.5]{DG}). To summarise, the unipotent smooth and connected subgroup $U = V$ is the unipotent radical of the connected  component of its smooth normaliser. According to Veisfeiler--Borel--Tits Theorem (\cite{WEI} and \cite[Corollary 3.2]{TB}), the latter is thus a parabolic subgroup of $G$ and $\Lie(N_G(U))= N_{\mathfrak{g}}(\mathfrak{u})$ is therefore a parabolic subalgebra of $\mathfrak{g}$.
\end{proof}

\begin{remark}
When $\h(G)<p\leq 2^{2d}$, even if an exponential map still exists and allows to integrate the restricted $p$-nil $p$-subalgebra $\mathfrak{u}$ into a unipotent smooth connected subgroup $U \subset G$, none of the steps arising in the previous proof is immediate. More precisely, the compatibility between the exponential map and the adjoint representation does no longer hold true (see for example \cite[Lemma 4.3]{HS}, this is a consequence of \cite[4.1.7]{SER1}), neither does the smoothness of normalisers (see \cite[Lemma 11]{HS}). 
\end{remark}

	\subsection{Conditional statements when the Killing form is non degenerate}
A careful reading of the proof of V. Morozov's Theorem provided in \cite[VIII, \S10 Corollaire 2]{BOU} leads to obtain the following conditional statement:

\begin{theorem}
Let $k$ be an algebraically closed field of characteristic $p>0$ and let $G$ be a reductive $k$-group. Assume that $p>\h(G)$ and that $\mathfrak{g}$ is endowed with a non degenerate $G$-equivariant symmetric bilinear form $\kappa$. Let $\mathfrak{u}$ be a subalgebra of $\mathfrak{g}$ such that:
\begin{enumerate}
\item the Lie algebra $\mathfrak{u}$ is the $p$-radical of $N_{\mathfrak{g}}(\mathfrak{u})$, 
\item the orthogonal of $\mathfrak{u}$ for the form $\kappa$ is a subalgebra of $\mathfrak{g}$.
\end{enumerate} 
Then $N_g(\mathfrak{u})$ is a parabolic subalgebra of $\mathfrak{g}$.
\label{th_si_killing}
\end{theorem}

The following lemma will be useful to prove Theorem \ref{th_si_killing}:
	\begin{lemma}
	Let $\mathfrak{h}$ and $\mathfrak{u}$ be two restricted $p$-Lie subalgebras of $\mathfrak{g}$ such that:
	\begin{itemize} 
	\item the subalgebra $\mathfrak{u}$ is the $p$-radical of $\mathfrak{h}$,
	 	\item the Killing form over $\mathfrak{g}$, denoted by $\kappa$, is non degenerate,
	 	\item the subalgebra $\mathfrak{h}$ is the normaliser of $\mathfrak{u}$ in $\mathfrak{g}$, namely one has $\mathfrak{h} = N_{\mathfrak{g}}(\mathfrak{u})$,
	 	\item the orthogonal of $\mathfrak{u}$ for $\kappa$ is a subalgebra of $\mathfrak{g}$, 
	 	\end{itemize}
then $\mathfrak{u}$ is the orthogonal of $\mathfrak{h}$ for $\kappa$.
	\label{orthogonal}
	\end{lemma}
	
	\begin{proof}
	Let $\mathfrak{q}$ be the orthogonal of $\mathfrak{u}$ for the Killing form over $\mathfrak{g}$. Let us show that $\mathfrak{q} \subseteq \mathfrak{h}$. The vanishing condition $\kappa([x,q_1],q_2) = \kappa(x, [q_1, q_2]) = 0,$ is satisfied for any pair $(q_1, q_2) \in \mathfrak{q}$ and any $x \in \mathfrak{u}$. Therefore the Lie bracket $[x,q_1]$ lies in $\mathfrak{q}^{\perp}:= (u^{\perp})^{\perp}$, where $\mathfrak{q}^{\perp}$ is the orthogonal of $\mathfrak{q}$ for $\kappa$. Hence we have obtained that:
\begin{itemize}
\item for any $q_1 \in \mathfrak{q}$ the Lie bracket $[x, q_1]$ belongs to $\mathfrak{u}$, 
\item the Lie algebra $\mathfrak{q}$ normalises $\mathfrak{u}$.
\end{itemize} 
In other words we have shown that $\mathfrak{q} \subseteq N_{\mathfrak{g}}(\mathfrak{u}) = \mathfrak{h}$.

Moreover and by assumption, the Lie algebra $\mathfrak{u}$ is an ideal of $\mathfrak{h}$ that only consists in $p$-nilpotent elements. It is therefore contained in the nilradical of $\mathfrak{h}$, and it is thus orthogonal to $\mathfrak{h}$ for the Killing form over $\mathfrak{h}$ (according to \cite[\S4, n\degree 4, Proposition 6]{BOU1}). In other words we have shown that $\mathfrak{h} \subseteq \mathfrak{q} := \mathfrak{u}^{\perp}$, hence the equality $\mathfrak{h} = \mathfrak{q}$.
	\end{proof}
	
\begin{proof}[Proof of Theorem \ref{th_si_killing}]
As $p>\h(G)$, the restricted $p$-nil $p$-subalgebra $\mathfrak{u}$ is integrable into a unipotent smooth connected subgroup $U\subset G$ (see \cite[6]{BDP} but also \cite[Proposition 3.1]{JEA1}). Let $N_0 := N_G(U)_{\red}^0$ be the smooth connected part of the normaliser of $U$ in $G$, set $V_1 := \Rad_U(N_0)$ and let us consider the tower of smooth connected normalisers of $U$ in $G$. For dimensional reasons this tower converges to a limit object $N_{\infty}$ which turns out to be the smooth connected part of the normaliser of its unipotent radical $V_{\infty}$. According to \cite[Théorème 3.2]{TB} the smooth connected normaliser $N_{\infty}$ is a parabolic subgroup of $G$, denoted by $P_G(U)$. There are two possibilities:
	\begin{enumerate}
		\item either $P_G(U) = G$, in this case $V_{\infty}=\{0\}$ is trivial, so is $U=\{0\}$, thus the vanishing of $\Lie(U) = \mathfrak{u}$ and the result follows,
		\item or $P_G(U)$ is a proper parabolic subgroup of $G$. It therefore contains a Borel subgroup $B \subseteq G$. As $k$ is algebraically closed, thus perfect, the unipotent smooth connected subgroup $V_{\infty}$ is $k$-embeddable into the unipotent radical of a Borel subgroup $B \subseteq G$ (see \cite[Corollary 3.7]{TB}). Moreover as the Lie algebras involved here are of finite dimension we have the following inclusions: 
		\[\mathfrak{u} \subseteq \mathfrak{v}_{\infty} \subseteq \radu(B) \subseteq \nil(\mathfrak{b}) \subseteq (\mathfrak{b})^{\perp}\]
\noindent where the last inclusion is ensured by \cite[\S 4 n\degree 4, Lemma]{BOU1}. As by assumption $\mathfrak{g}$ is endowed with a non degenerate form one can deduce that:
		\[(\mathfrak{b}^{\perp})^{\perp} = \mathfrak{b} \subseteq  \nil(\mathfrak{b})^{\perp} \subseteq  \mathfrak{v}_{\infty}^{\perp} \subseteq \mathfrak{u}^{\perp} = N_{\mathfrak{g}}(\mathfrak{u}),\] 
\noindent where the last equality is ensured by Lemma \ref{orthogonal} and Remark \ref{semisimple_rad_p_nil}. By assumption $p>\h(G)$, in other words $p \neq 2$ when $G$ is of type $A_1$ and $p> 3$ when $G$ is not of $A_n$-type (see Table \ref{Hypothèses_sur_la caractéristiqu_ pour_un_groupe_simple}). In this case the parabolic subalgebras of $\mathfrak{g}$ are exactly those that contain the Lie algebra of a Borel subgroup (see Section \ref{ssalg_parab}). The Lie algebra $N_{\mathfrak{g}}(\mathfrak{u})$ satisfies this condition according to what precedes because $\mathfrak{p_g}(U) \subseteq N_{\mathfrak{g}}(\mathfrak{u})$ is a parabolic Lie subalgebra of $\mathfrak{g}$, thus it contains a Borel subalgebra of $\mathfrak{g}$. So $N_{\mathfrak{g}}(\mathfrak{u})$ is indeed a parabolic subalgebra of $\mathfrak{g}$.
	\end{enumerate}
\end{proof}

\section{Hilbert--Mumford--Kempf--Rousseau method and the optimal parabolic subgroup}
\label{Kempf-Rousseau}

In what follows and unless otherwise stated, the field $k$ is algebraically closed of characteristic $p>0$ and $G$ is a reductive $k$-group. We assume that $p$ is not of torsion for $G$ (let us underline that this condition is in particular always satisfied when $p>\h(G)$ (see section \ref{hypothèses_car})). In what follows and by convention, a $k$-variety is a separated $k$-scheme of finite type. In particular a variety is not assumed to be reduced a priori. If $Y$ is a $k$-variety we denote by $\bar{Y}$ its reduced closure. 

	\subsection{Construction}
	\label{construction}
Set $X := \mathfrak{g}^{\oplus^d}$ and let $\Ad(G)^{\Delta}$ be the diagonal adjoint action of $G$ on $X$.
Let $\mathfrak{u} \subseteq \mathfrak{g} $ be a restricted $p$-nil $p$-subalgebra of dimension $d>0$ and let $\{x_1, \hdots, x_d\}$ be a basis of $\mathfrak{u}$. As explained in this section, the Hilbert--Mumford--Kempf-Rousseau method (see \cite{K}) allows us to obtain a parabolic subgroup that satisfies the following properties:
\begin{itemize}
\item it is of the form $P_G(\lambda_{\mathfrak{u}})$, where $\lambda_{\mathfrak{u}}$ is a not uniquely defined cocharacter of $G$, optimal for $\mathfrak{u}$ in a certain sense, 
\item the Lie algebra $\mathfrak{u}$ is a Lie subalgebra of $\Lie(\Rad_U(P_G(\lambda_{\mathfrak{u}}))):= \mathfrak{rad_u}(P_G(\lambda_{\mathfrak{u}}))$. 
\end{itemize}  
Set $x = (x_1, \hdots, x_n) \in X$ and $S := \left(\overline{\Ad(G)^{\Delta} x} \setminus \Ad(G)^{\Delta} x\right)_{\red} \subseteq X$. 

\begin{lemma}
The subvariety $S\subseteq X$ is closed, it is $\Ad(G)^{\Delta}$-invariant and not empty. Moreover the vector $x$ does not belong to $S$. 
\label{fond}
\end{lemma}

	\begin{proof}[Proof of Lemma \ref{fond}]
Only the emptiness of $S$ is not clear. To prove it this is enough to check that $0\in S$: according to Corollary \ref{corollaire_LMT}, there exists a Borel subgroup $B \subseteq G$ such that $\mathfrak{u}$ is a Lie subalgebra of $\radu(B)$. Let $T\subseteq B\subseteq G$ be a maximal torus. Any dominant regular cocharacter $\lambda$ of $G$ with respect to $T$ satisfies the equality $\lim_{t \to 0} \lambda(t)\cdot z = 0$ for any $z \in \radu(B)^{\oplus^d}$. In particular the following holds true \[0 \in \left(\overline{\Ad(G)^{\Delta} x} \setminus \Ad(G)^{\Delta} x\right)_{\red} \subseteq X.\]
	\end{proof}
	 
Let us denote by $\lvert X, x \rvert := \left\lbrace \lambda \in X_{*}(G) \mid \lim_{t\to 0} \Ad^{\Delta}(\lambda(t))x \text{ exists}\right\rbrace$ the attractor of $X$ (see \cite[1.4]{DriGai}).
For any $\lambda \in \lvert X, x \rvert$ set $x_0^{\lambda} := \lim_{t\to 0} \Ad^{\Delta}(\lambda(t)) x \in S$ (or simply $x_0:= x_0^{\lambda}$ when there is no ambiguity on the chosen cocharacter $\lambda$). Finally let us set
\[\alpha_{S,x}(\lambda) :=  \begin{cases}
 \text{vanishing order of } t \mapsto \Ad^{\Delta}(\lambda(t))x-x_0 : \mathbb{A}^1 \rightarrow X \text{ if } x_0 \in S,   \\
  0 \text{ otherwise.}
 \end{cases} \]
In particular $\alpha_{S,x} : X_*(G) \rightarrow \mathbb{N}$ has strictly positive values if and only if $ x_0^{\lambda} \in S$.  

As a reminder, any cocharacter $\lambda : \mathbb{G}_m \rightarrow k$ defines both a parabolic subgroup $P_G(\lambda)$ and its unipotent radical $U_G(\lambda) := \Rad_U(P_G(\lambda))$ such that for any $k$-algebra $A$ one has $P_G(\lambda)(A):= \lbrace g \in G(A) \mid \lim_{t \to 0} \lambda(t)\cdot g \text{ exists}\rbrace$ and $U_G(\lambda)(A):= \lbrace g \in P_G(\lambda)(A) \mid \lim_{t \to 0} \lambda(t)\cdot g = 1\rbrace$. The following theorem has been established by G. R. Kempf in the following more general setting in which we consider:
	\begin{itemize}
		\item a finite dimensional linear representation $(\rho,V)$ of $G$,
		\item a closed $\rho(G)^{\Delta}$-invariant subvariety $S \subset V^{\oplus d}$.
	\end{itemize}	  
It allows to associate to $\mathfrak{u}$ optimal cocharacters and the parabolic subgroup they define.
  
\begin{theorem}[Kempf, {\cite[Theorem 3.4]{K}}]
The function $\lambda \mapsto \alpha_{S,x}(\lambda)/ \left\lVert \lambda \right \rVert$ admits a maximum $B >0$ on $\lvert X, x\rvert$. Set:
\[\Delta_{S,x} = \{\lambda \in \lvert X, x \rvert \mid \alpha_{S,x}(\lambda) = B \cdot \left\lVert \lambda \right \rVert \text{ and } \lambda \text{ is indivisible}\}.\]
Such a cocharacter $\lambda \in \Delta_{S,x}$ is said to be optimal. Under these conditions:
	\begin{enumerate}
		\item the set $\Delta_{S,x}$ is non-empty,
		\item there exists a parabolic subgroup $P_G(\lambda_{\mathfrak{u}}) \subseteq G$ (denoted by $P_{S,x}$ in \cite{K}), and such that $P_G(\lambda_{\mathfrak{u}}) = P_{G}(\lambda)$ for any $\lambda \in \Delta_{S,x}$. The parabolic subgroup $P_G(\lambda_{\mathfrak{u}})$ is the optimal parabolic subgroup of $\mathfrak{u}$, 
		\item the set $\Delta_{S,x}$ is a principal homogeneous space under $\Rad_U(P_G(\lambda_{\mathfrak{u}}))(k)$, 
		\item any maximal torus of $P_G(\lambda_{\mathfrak{u}})$ contains a unique cocharacter that belongs to $\Delta_{S,x}$.
	\end{enumerate}
\label{Kempf}
\end{theorem}
 
\begin{lemma}
Let $\mathfrak{u}$ and $P_G(\lambda_{\mathfrak{u}})$ be as in Theorem \ref{Kempf}. Let us denote by $\mathfrak{u_g(\lambda_u)} := \mathfrak{\rad_u}(P_G(\lambda_u))$. Then the inclusion $\mathfrak{u} \subseteq \mathfrak{u_g(\lambda_u)}$ holds true.
\label{unip_parab}
\end{lemma}

	\begin{proof}
		According to \cite[I, Proposition 2.1.8]{CGP} the Lie algebra $\mathfrak{p_g(\lambda_u)}$ writes $\mathfrak{c_g(\lambda_u)} \oplus \mathfrak{u_g(\lambda_u)}$ where: 
			\begin{itemize}
				\item the factor $\mathfrak{c_g(\lambda_u) }:= \mathfrak{g}_0$ is the Lie algebra of the centraliser of the torus $\lambda(\mathbb{G}_m)$ (which is denoted by $C_G(\lambda)$),
				\item and $\mathfrak{u_g(\lambda_u)} := \bigoplus_{n>0} \mathfrak{g}_n$, with $\mathfrak{g}_n := \{v \in \mathfrak{g} \mid \lambda(t) \cdot v = t^n v\}$.
			\end{itemize}
The choice of $\lambda_{\mathfrak{u}}$ is such that $\alpha_{S,x}(\lambda_{\mathfrak{u}})$ is strictly positive. In particular the weights are all strictly positive for $x$, so they are for any vector of the basis of $\mathfrak{u}$ and hence for $\mathfrak{u}$. This allows us to obtain the inclusion $\mathfrak{u} \subseteq \mathfrak{u_g(\lambda_u)}$.
	\end{proof}

\begin{remark}
If $\mathfrak{u} \subseteq \mathfrak{g}$ is a nilpotent Lie algebra of dimension $1$, there always exists a parabolic subgroup $P_G(\lambda_{\mathfrak{u}}) \subseteq G$ such that $\mathfrak{u} \subseteq \mathfrak{u_g(\lambda_u)}$ whatever the characteristic of $k$. In this case indeed the $\Ad(G)$-orbite of $x$ is never closed (according to \cite[\S 2.10]{Jantzen2004})). Let also us remark that the restriction on $p$ in the general case (without restriction on the dimension), ensures that $\mathfrak{u}$ is embeddable into the Lie algebra of the unipotent radical of a Borel subgroup. This guarantee that $\Ad(G)$-orbit of  $x$ is not closed (see the proof of Lemma \ref{fond}).
\end{remark}

	\subsection{Obstruction}
	In what follows we assume that the restricted $p$-nil $p$-Lie subalgebra $\mathfrak{u} \subseteq \mathfrak{g}$ is the set of all $p$-nilpotent elements of $\rad(N_\mathfrak{g}(\mathfrak{u}))$.
	
		\subsubsection{Preliminary remarks}

As explained by the following lemma, the whole point here is to obtain the inclusion $N_{\mathfrak{g}}(\mathfrak{u})\subseteq \mathfrak{p_g(\lambda_u)}$: 

\begin{lemma}
Let us assume that the inclusion $N_{\mathfrak{g}}(\mathfrak{u}) \subseteq \mathfrak{p_g(\lambda_u)}$ holds true. Then under our assumptions on $\mathfrak{u}$ this normaliser is a parabolic subalgebra of $\mathfrak{g}$.
\label{norm_parab}
\end{lemma}
	
	\begin{proof}
The inclusion $\mathfrak{u} \subseteq \mathfrak{u_g(\lambda_u)}$ has been obtained in Lemma \ref{unip_parab} we have shown. Assume this inclusion to be strict. In this case $\mathfrak{u}$ is a proper ideal of $\mathfrak{u}' := N_{\mathfrak{u_g(\lambda_u)}}(\mathfrak{u}) \subseteq \mathfrak{u_g(\lambda_u)}$ (see for example \cite[Chapter 1, \S 3 exercise 7]{HUM}). 

The intersection $\mathfrak{u_g(\lambda_u)} \cap \N_{\mathfrak{g}}(\mathfrak{u})=\mathfrak{u}'$ is a restricted $p$-algebra because it derives from an algebraic group (in this case $\mathfrak{u}':= \Lie(N_{U_G(\mathfrak{\lambda_u})}(\mathfrak{u}))$, see Lemma \ref{reminder_normaliser}). This is a $p$-ideal of $\N_{\mathfrak{g}}(\mathfrak{u})$ (because $N_{\mathfrak{g}}(\mathfrak{u}) \subseteq \mathfrak{p_g(\lambda_u)}$ by assumption). It is $p$-nil according to \cite[Remark 2.14]{JEA1} because $\mathfrak{u_g(\lambda_u)}$ is the Lie algebra of the unipotent radical of $P_G(\lambda_{\mathfrak{u}})$. In particular $\mathfrak{u'}$ is a nilpotent (thus solvable) ideal of $N_{\mathfrak{g}}(\mathfrak{u})$, so it is contained in $\rad(N_{\mathfrak{g}}(\mathfrak{u}))$. As $\mathfrak{u}$ is the set of $p$-nilpotent elements of this radical, one necessarily has $\mathfrak{u}'\subseteq \mathfrak{u}$, thus the equality $ \mathfrak{u}= \mathfrak{u_g(\lambda_u)} \cap \N_{\mathfrak{g}}(\mathfrak{u})$. This contradicts the strictness of the inclusion $\mathfrak{u} \subsetneq \mathfrak{u_g(\lambda_u)}$. In other words  the equality $\mathfrak{u} = \mathfrak{u_g(\lambda_u)}$ is satisfied.
	\end{proof}

\begin{remark}
The inclusion required in Lemma \ref{norm_parab} is immediate in some cases, for example
when $\N_G(\mathfrak{u})$ is smooth (apply \cite[Corollary 9, (2)]{M2} to $\N_G(\mathfrak{u})^0$).
\label{cas_ok_inclusion_norm}
\end{remark} 

\begin{remark} 
When the subalgebra $\mathfrak{u} \subset \mathfrak{g}$ is generated by a $p$-nilpotent element $x$, and with no additional assumption, the normaliser $\N_{\mathfrak{g}}(\mathfrak{u})$ is necessarily the Lie algebra of a parabolic subgroup of $G$. Indeed, under this assumption and according to G. R. Kempf \cite[Proposition 18 (1)]{M1} there exists an optimal cocharacter $\lambda$ of $\mathfrak{u}$ associated to $x$. This means that:
			\begin{enumerate}
				\item the element $x$ is of weight $2$ for the conjugation by $\lambda$,
				
\item the image $\lambda(\mathbb{G}_m)$ is contained in the derived group of a Levi subgroup $L \subseteq G$ for which $x \in \Lie(L)$ is distinguished. As a reminder a Levi subgroup is distinguished if any torus contained in $C_G(x)$ is contained in the center of $G$.
			\end{enumerate}
The $\lambda$-conjugation on $\mathfrak{g}$ induces a grading on $\mathfrak{g} = \bigoplus_i \mathfrak{g}_i$. For such a grading the normaliser $N_{\mathfrak{g}}(\mathfrak{u})$ is contained in $\mathfrak{g}_0$. Indeed let $n := \bigoplus_{i \in \mathbb{Z}} n_i \in \N_{\mathfrak{g}}(x)$, thus the bracket $[n,x]$ lies in $\mathfrak{g}_2$. In particular, the bracket $[n_i,x]$ is necessarily trivial for any $i < 0$ because the bracket $[\mathfrak{g}_i, \mathfrak{g}_j]$ belongs to $\mathfrak{g}_{i+j}$. In other words $n_i \in C_{\mathfrak{g}}(x)$ for any $i<0$, where $C_{\mathfrak{g}}(x)$ is the centraliser of $x$ in $\mathfrak{g}$. But as $G$ satisfies the standard hypotheses, the equalities 
\[C_{\mathfrak{g}}(x) = \mathfrak{p}_G(\lambda) = \Lie(P_G(\lambda)) = \Lie(C_G(x))\] hold true (see \cite[5.9, Remark]{Jantzen2004}). Therefore $n_i = 0$ vanishes for any $i<0$ and the inclusion $\N_{\mathfrak{g}}(x) \subseteq \mathfrak{p}_G(\lambda)$ is satisfied. 	
\end{remark}

 	\subsubsection{Characterisation, proof of Theorem \ref{obstruction}}
The first point of Remark \ref{cas_ok_inclusion_norm} is actually an equivalence, as stated by Theorem \ref{obstruction}.

\begin{proof}[Proof of Theorem \ref{obstruction}]
		Let us denote by $U_G(\lambda_{\mathfrak{u}}) := \Rad_U(P_G(\lambda_{\mathfrak{u}}))$. 
		
		\begin{enumerate}
			\item[$1. \implies 3.$] There are two possibilities:
	\begin{itemize}
	\item either $P_G(\lambda_{\mathfrak{u}}) = G$, then the $p$-nil Lie algebra $\mathfrak{u}$ is trivial because so is $\mathfrak{u} \subseteq \mathfrak{u_g(\lambda_u)}$ (as a reminder this last inclusion is provided by Lemma \ref{unip_parab});
	\item or $P_G(\lambda_{\mathfrak{u}})$ is a proper parabolic subgroup of $G$. In this case the following holds true: 
	\[\mathfrak{u} \subseteq \mathfrak{u_g(\lambda_u)}=\rad_p(\mathfrak{p_g(\lambda_u)}) =\rad_p(N_{\mathfrak{g}}(\mathfrak{u}))\subset \rad(N_{\mathfrak{g}}(\mathfrak{u}))\]
\noindent where:
\begin{itemize}
\item the first inclusion is given by Lemma \ref{unip_parab},
\item the first equality is provided by \cite[Remark 2.14]{JEA1}, 
\item the second equality comes from the assumption $N_{\mathfrak{g}}(\mathfrak{u})=\mathfrak{p_g(\lambda_u)}$, \item and the last inclusion is provided by \cite[Lemma 2.5]{JEA1}. 
\end{itemize}
Furthermore the inclusion $\rad_p(N_{\mathfrak{g}}(\mathfrak{u})) \subseteq \mathfrak{u}$ is satisfied a priori because $\mathfrak{u}$ is the set of all $p$-nilpotent elements of $\rad(N_{\mathfrak{g}}(\mathfrak{u}))$. To summarise $\mathfrak{u} = \rad_p(N_{\mathfrak{g}}(\mathfrak{u}))$ and the equality $\mathfrak{u} = \mathfrak{u_g(\lambda_u)}$ is satisfied.
Moreover one has 
\[N_G(\mathfrak{u})^0 = N_G(\mathfrak{u_g(\lambda_u)})^{0} \supseteq N_G(U_G(\lambda_{\mathfrak{u}}))^0 = P_G(\lambda_{\mathfrak{u}})^0= P_G(\lambda_ {\mathfrak{u}})\] where 
\begin{itemize}
\item the inclusion is given by Lemma \ref{reminder_normaliser}, 
\item the first equality is ensured by assumption,
\item and the last equality is provided by \cite[XXII, Corollaire 5.8.5]{SGA33}.
\end{itemize}
In other words $P_G(\lambda_ {\mathfrak{u}})\hookrightarrow N_G(\mathfrak{u})^0$ is a monomorphism that induces a isomorphism of Lie algebras $P_{\mathfrak{g}}(\lambda_ {\mathfrak{u}})\hookrightarrow N_{\mathfrak{g}}(\mathfrak{u})^0:= \Lie(P_G(\lambda_{\mathfrak{u}}))$. Furthermore as $P_G(\lambda_{\mathfrak{u}})$ is a smooth $k$-group (because $P_G(\lambda_ {\mathfrak{u}})$ is a parabolic subgroup, see \cite[XXVI, Définition 1.1]{SGA33}). By \cite[II, \S5, 5.5]{DG} the aforementioned monomorphism of groups is therefore an open immersion. Finally the equality $P_G(\lambda_ {\mathfrak{u}}) = N_G(\mathfrak{u})^0$ holds true as $P_G(\lambda_ {\mathfrak{u}})$ is a closed subgroup of $N_G(\mathfrak{u})^0$. Hence $N_G(\mathfrak{u})^0$ is smooth over $k$.
	\end{itemize}
			\item [$3. \implies 2.$] This is clear because a parabolic subgroup is smooth by definition (see \cite[XXVI, Définition 1.1]{SGA33}).
			\item[$2. \implies 1.$] According to Lemma \ref{norm_parab} the key point is to show the inclusion $N_{\mathfrak{g}}(\mathfrak{u}) \subseteq \mathfrak{p_g(\lambda_u)}$. This is provided by \cite[Corollary 9, (2)]{M2} because $N_G(\mathfrak{u})^0$ is smooth by assumption. Let us detail the reasoning, the notations are those of section \ref{Kempf-Rousseau}. 
			
			The idea is to show that $N_{G}(\mathfrak{u})^{0}$ normalises $P_G(\lambda_{\mathfrak{u}})$ because then $N_{G}(\mathfrak{u})^{0} \subseteq N_{G}(P_G(\lambda_{\mathfrak{u}}))= P_G(\lambda_{\mathfrak{u}}),$
\noindent (the equality is immediate because $P_G(\lambda_{\mathfrak{u}})$ is its own normaliser, see \cite[XXII, corollaire 5.8.5]{SGA33}). This leads to the desired inclusion.

The smoothness of the normaliser $N_{G}(\mathfrak{u})^{0}$ is crucial: it allows us to restrict our proof to a reasoning on $k$-points by Zariski-density of $N_G(\mathfrak{u})^0(k)$ in $N_G(\mathfrak{u})^0$ (because $k$ is algebraically closed, see \cite[Corollaire 10.4.8]{EGA4c}). Thus one only needs to show that $\Ad(h) P_G(\lambda_{\mathfrak{u}}) = P_G(\lambda_{\mathfrak{u}})$ for any $h \in N_G(\mathfrak{u})^{0}(k)$. According to \cite[Corollary 3.5]{K} the subgroup $\Ad(h) P_G(\lambda_{\mathfrak{u}})$ and the optimal parabolic subgroup defined by an optimal cocharacter of $\Delta_{S,\Ad(h)x}$ satisfy the equality $\Ad(h) P_G(\lambda_{\mathfrak{u}}) =  P_{S,\Ad(h)x}$. 
Moreover, the sets $\Delta_{S,x}$ and $\Delta_{S, \Ad(h)x}$ are equal according to \cite[Corollary 7]{M2} because for any $h\in N_G(\mathfrak{u})^0(k) $:
	\begin{enumerate}
		\item the sets $\lvert X,x \rvert$ and $\lvert X, \Ad(h)(x)\rvert$ are the same,  
		 \item one has $\alpha_{S,x}(\lambda) = \alpha_{S,\Ad(h)x}(\lambda)$.
	\end{enumerate}
This leads to the equality $P_{S,\Ad(h)x} = P_{G}(\lambda_{\mathfrak{u}})$ for any $h \in N_G(\mathfrak{u})^0(k)$, whence the inclusion $N_G(\mathfrak{u})^0\subseteq P_G(\lambda_{\mathfrak{u}})$.
		\end{enumerate}
	\end{proof}

\begin{remark}
When $k$ is of characteristic $p>\h(G)$ we introduce in Section \ref{saturation_inf} the notion of infinitesimal saturation defined by P. Deligne and show, by making use of Lemma \ref{sat_inf_normalisateur}, that the subgroup $N_G(\mathfrak{u}) \subseteq G$ is infinitesimally saturated, see Definition \ref{def_sat_inf}. This ensures that Lemma \ref{p_rad_inf_sat} holds true in the context we are interested in (see \cite[Theorem 2.5]{DG}) and thus leads to get the equality $\mathfrak{u} = \radu(P_G(\lambda_{\mathfrak{u}}))$ more easily (as one only needs to apply the aforementioned lemma).
\end{remark}

\section{Analogue and corollaries of Morozov's Theorem in characteristic \texorpdfstring{$p>\h(G)$}{Lg}}
\label{chapitre_morozov_p_sup}

In this section $k$ is an algebraically closed field of characteristic $p>0$ and $G$ is a reductive $k$-group. In Section \ref{proof_Morozov_psup} we prove that the analogous statement of Morozov's Theorem holds true when $p>\h(G)$:

\begin{theorem}[Analogue of Morozov's Theorem when $p>\h(G)$]
Let $k$ be an algebraically closed field of characteristic $p>0$ and $G$ be a reductive $k$-group such that $p > \h(G)$ (where $\h(G)$ is the Coxeter number of $G$, see Section \ref{hypothèses_car}). Let $\mathfrak{u} \subseteq \mathfrak{g}$ be a Lie subalgebra of $\mathfrak{g}$ and assume that $\mathfrak{u}$ is the $p$-radical of $N_{\mathfrak{g}}(\mathfrak{u})$. Then:
	\begin{enumerate}
		\item the normaliser $N_\mathfrak{g}(\mathfrak{u})$ is a parabolic subalgebra of $\mathfrak{g}$,
		\item this parabolic subalgebra satisfies the equalities $N_\mathfrak{g}(\mathfrak{u})= \mathfrak{p_g}(\lambda_{\mathfrak{u}})= \mathfrak{p_g}(U)$ 
where $U \subset G$ is a unipotent smooth connected subgroup such that $\Lie(U) = \mathfrak{u}$ (when $p>\h(G)$ such a subgroup exists according to \cite[Section 6]{BDP}, see also \cite[Proposition 3.1]{JEA1}).		 
	\end{enumerate}
\label{Morozov_p_sup}
\end{theorem}

This theorem can be rephrased in terms of $p$-nilpotent elements of the radical of the normaliser, this is the point of section \ref{section_cor_p_grand}:

\begin{corollary}
Let $k$ be an algebraically closed field of characteristic $p>0$ and $G$ a reductive $k$-group such that $p > \h(G)$. Let $\mathfrak{u} \subset \mathfrak{g}$ be a restricted $p$-subalgebra of $\mathfrak{g}$ such that $\mathfrak{u}$ is the set of all $p$-nilpotent elements of the radical $\rad(N_{\mathfrak{g}}(\mathfrak{u}))$. Then $N_\mathfrak{g}(\mathfrak{u})$ is the Lie subalgebra of the parabolic subgroup obtained in Theorem \ref{Morozov_p_sup}. 
\label{cor_p_elements_p_sup}
\end{corollary}

In the same section we also derive a ``limit'' statement of Morozov's Theorem in characteristic $p>\h(G)$:

\begin{corollary}
Under the assumptions of Theorem \ref{Morozov_p_sup}, let $\mathfrak{u}\subset \mathfrak{g}$ be a restricted $p$-nil subalgebra, then:
\begin{enumerate}
	\item the tower of normalisers of $\mathfrak{u}$ converges to a parabolic subalgebra of $\mathfrak{g}$ denoted by $N_{\mathfrak{g}}(\mathfrak{u}_{\infty})$, 
	\item this limit object satisfies the following equalities $N_{\mathfrak{g}}(\mathfrak{u}_{\infty}) = \mathfrak{p_g(\lambda_{u_{\infty}})} = \Lie(N_G(V_{\infty}))$ where the subgroup $N_G(V_{\infty})$ is the limit object of the tower of smooth normalisers associated to $\Rad_U(N_G(\mathfrak{u})^0_{\red})$.
\end{enumerate}
\label{cor_towers_p_sup}
\end{corollary}

Let us in particular remark that Theorem \ref{Morozov_p_sup} above is a special case of this corollary (for which the tower of normalisers stabilises directly).

Section \ref{section_cor_Killing} is dedicated to the proof of Corollary \ref{cor_killing}, which rephrase Corollary \ref{cor_p_elements_p_sup} by involving a non-degenerate $G$-equivariant symmetric bilinear form on $\mathfrak{g}$.

Finally, Section \ref{section_cor_Mor_pgrand} is dedicated to prove that Statement \ref{cor_morozov} holds true when $p>\h(G)$:

\begin{corollary}[of Theorem \ref{Morozov_p_sup}]
Let $k$ be an algebraically closed field of characteristic $p>0$ and let $G$ be a reductive $k$-group. Assume that $p>\h(G)$. Let $\mathfrak{q} \subseteq \mathfrak{g}$ be a maximal proper restricted $p$-Lie subalgebra. Then $\mathfrak{q}$ is either $p$-reductive (see Definition \ref{def_pred}) or parabolic.   
\label{cor_Mor_pgrand}
\end{corollary}

Let $\mathfrak{u} \subset \mathfrak{g}$ be a restricted $p$-nil $p$-Lie subalgebra. See \cite[Section 2.2]{JEA1} for the terminology and the fundamental properties of the objects involved in this section. Under the assumption $p>\h(G)$ a theorem of J-P. Serre in \cite[Part II, Lecture 2, Theorem 3]{SER}, and detailed by V. Balaji, P. Deligne and A. J. Parameswaran in \cite[section 6]{BDP} provides the existence of an integration for any restricted $p$-nil $p$-Lie subalgebra of $\mathfrak{g}$ (see also \cite[Proposition 3.1]{JEA1}). Nevertheless, and as already mentioned in the introduction, such an integration is a priori no longer compatible with the adjoint representation. In particular: 
\begin{itemize}
\item if we denote by $U \subset G$ the unipotent smooth connected subgroup that integrates $\mathfrak{u}$, there is no good reasons for the equality $N_{G}(\mathfrak{u}) = N_G(U)$ to hold true a priori. 
\item Moreover, as mentioned in the introduction, the normalisers involved here are a priori no longer smooth. 
\end{itemize}
The notion of infinitesimal saturation (see definition \ref{def_sat_inf}) allows to settle these issues.

\subsection{Measuring the lack of smoothness of normalisers}
\subsubsection{Infinitesimal saturation}
	\label{saturation_inf}

The following lemmas are crucial in the proof of Theorem \ref{Morozov_p_sup}. Let us note that they are immediate as soon as $p>2\h(G)-2$, as detailed in Section \ref{rem_car_0_ou_grande}.

\begin{lemma}
Let $k$ be an algebraically closed field of characteristic $p>0$ and $G$ be a reductive $k$-group such that $p>\h(G)$. Let $\mathfrak{u} \subseteq \mathfrak{g}$ be a restricted $p$-nil subalgebra of $\mathfrak{g}$. The normaliser $N_G(\mathfrak{u})$ is infinitesimally saturated.
\label{sat_inf_normalisateur}
\end{lemma}

\begin{proof}
Let $x \in
\Lie(N_G(\mathfrak{u}))= N_{\mathfrak{g}}(\mathfrak{u})$ be a $p$-nilpotent element (note that the equality is provided by Lemma \ref{reminder_normaliser}). One needs to show that the $t$-power map factors through $N_G(\mathfrak{u})$. The following holds true:
	\begin{enumerate}
		\item if $x \in \mathfrak{u}$, as $U$ integrates $\mathfrak{u}$, the inclusions $\exp(tx) \in U \subseteq N_G(U) \subseteq N_G(\mathfrak{u})$ are satisfied (see for example \cite[Lemma 6.4]{JEA1}),
		\item if $x \in N_{\mathfrak{g}}(\mathfrak{u}) \setminus \mathfrak{u}$, let $\mathfrak{u}_{x}:= \mathfrak{u} \oplus kx$ be the restricted $p$-Lie algebra generated by $\mathfrak{u}$ and $x$. This is a Lie subalgebra because $x \in N_{\mathfrak{g}}(\mathfrak{u})$. It remains to show that it is $p$-nil. This means that we have to check that any element of the form $u + x$, with $u \in \mathfrak{u}$ is $p$-nilpotent. As $p>\h(G)$ this amounts to show that $(u +x)^{[p]}=1$. By assumption one has $\mathfrak{u} \oplus kx \subseteq \mathfrak{u}$. This assumption together with Jacobson formulae provide the desired result. 
According to \cite[Proposition 3.1]{JEA1} there exist unipotent smooth connected subgroups of $G$, respectively $U$ and $U_x$ such that $\Lie(U_x) = \mathfrak{u}_x$, respectively $\Lie(U) =\mathfrak{u}$. Note that these constructions are independent from the Borel subalgebra in whose $\mathfrak{u}$ and $\mathfrak{u}_x$ are embedded. In particular, it is not necessary to choose the same Borel subalgebra for both $\mathfrak{u}$ and $\mathfrak{u}_x$.

By construction we have that $U \subseteq U_x$. Furthermore these groups are smooth, so the following equalities $\dim \mathfrak{u}_x = \dim U_x$ and $\dim \mathfrak{u}= \dim U$ hold true. As $\mathfrak{u} \subseteq \mathfrak{u}_x$ is of codimension $1$ so is the unipotent smooth connected subgroup $U \subseteq U_x$. According to \cite[IV, \S 4 Corollaire 3.15]{DG} this is a normal subgroup of $U_x$ (this quotient is isomorphic to $\mathbb{G}_a$). In other words on has $U_x \subseteq N_G(U) \subseteq N_G(\mathfrak{u})$ where the last inclusion comes from Lemma \ref{reminder_normaliser}, so $\exp(tx) \in N_G(\mathfrak{u})$ and $N_G(\mathfrak{u})$ are infinitesimally saturated.
\end{enumerate}		
\end{proof}

The proof of Lemma \ref{sat_inf_normalisateur} allows us to show the following result:

\begin{lemma}
Under the previous assumptions, set $U:= \exp_{\mathfrak{b}}(\mathfrak{u})$ where $\mathfrak{b}\subseteq \mathfrak{g}$ is a Borel subalgebra such that $\mathfrak{u} \subseteq \radu(B)$. Then $N_G(U)$ is infinitesimally saturated.
\label{sat_inf_norm_groupe}
\end{lemma} 

\begin{proof}[Proof of Lemma \ref{sat_inf_norm_groupe}]
We have to show that, given any $p$-nilpotent element $x \in \Lie(N_{G}(U))\subset N_{\mathfrak{g}}(\mathfrak{u})$, the $t$-power map factors through $N_G(U)$. The equality is ensured by \cite[Lemma 6.4]{JEA1}. According to the proof of Lemma \ref{sat_inf_normalisateur}, the image of $x$ by the $t$-power map is an element of $N_{G}(U)$, hence the latter is infinitesimally saturated. 
\end{proof}

\begin{remark}
We just have shown in Lemma \ref{sat_inf_normalisateur} that the normaliser $N_{G}(\mathfrak{u})$ of any $p$-nil subalgebra $\mathfrak{u}\subset \mathfrak{g}$ is infinitesimally saturated. In particular the normaliser of Theorem \ref{Morozov_p_sup} is infinitesimally saturated. This condition is necessary for the statement to be true because any parabolic subgroup of $G$ is infinitesimally saturated (see \cite[Lemma 4.2]{JEA1}).
\end{remark}

	\subsubsection{Exponential map and normalisers}
	
Let $G$ be a reductive group over an algebraic closed field $k$ of characteristic $p>\h(G)$. When $U \subseteq G$ is a reductive subgroup only the inclusion $N_G(U)_{\red} \subseteq N_G(\mathfrak{u})_{\red}$ is a priori satisfied (see for instance \cite[Lemma 6.4 and Remark 6.7]{JEA1}). When $\mathfrak{u} \subseteq \mathfrak{g}$ is a restricted $p$-nil $p$-subalgebra, let $U =\exp_{\mathfrak{b}}(\mathfrak{u})$, where $\mathfrak{b}\subseteq \mathfrak{g}$ is a Borel subalgebra such that $\mathfrak{u} \subseteq \radu(B)$. Under our assumptions on $p$ the normalisers are equal. We start by showing the equality of the reduced parts, the preliminary remark is useful for what follows:

\begin{remark}
Remember that, given a $p$-nilpotent element $x \in \mathfrak{g}$, the exponential map induces the $t$-power map:
\begin{alignat*}{3}
\exp_x(\cdot) :  \: & \mathbb{G}_a \: & \rightarrow \: & G \\
\: & t \: & \mapsto \: & \exp(tx).
\end{alignat*}
that itself induces a morphism:
\begin{alignat*}{3}
\psi_{\mathfrak{u}} : W\: &(\mathfrak{u}) \:& \times \: & \mathbb{G}_a \: & \rightarrow \: & G \\
\: & (x,\: & \: & t) \: & \mapsto \: & \exp(tx)=:\exp_x(t).
\end{alignat*}
Denote by $\widetilde{U}$ the subgroup of $G$ generated by $\psi_{\mathfrak{u}}$ as an $\fppf$-sheaf (see \cite[VIB Proposition 7.1 and Remarque 7.6.1]{SGA31}). This subgroup $\widetilde{U}$ is:
\begin{itemize}
\item unipotent, according to \cite[IV, \S2, n\degree 2, Proposition 2.5]{DG} as it can be embedded into the unipotent radical of the Borel subgroup of $G$, 
\item connected, according to \cite[VIB Corollaire 7.2.1]{SGA31} because $\mathfrak{u} \cong W(\mathfrak{u})$ is geometrically reduced and geometrically connected;
\item and reduced, thus smooth, because it is the image of a smooth, thus reduced, group. 
\end{itemize} 
On the other hand, let $B\subset G$ be a Borel subgroup such that $\mathfrak{u}$ is a Lie subalgebra of $\mathfrak{\rad_u}(B)$ (this latter exists by virtue of \cite[Corollary 2.1 and Remark 2.2]{JEA1}). We denote by $U:=\exp_B(\mathfrak{u})$ the unipotent smooth connected subgroup obtained in \cite[Proposition 3.1]{JEA1}. The subgroups $\widetilde{U}$ and $U$ coincide, the inclusion $\widetilde{U} \subseteq U$ is clear by construction. Once again the equality of the considered groups follows from the equality of their Lie algebras. To obtain it, remember that
the exponential map $\exp_{\mathfrak{b}}$ induces the identity on the tangent spaces, hence the inclusion $\mathfrak{u} \subseteq \Lie(\widetilde{U})$ is necessarily satisfied. Therefore on actually has $\mathfrak{u} \subseteq \Lie(\widetilde{U})$, and so the $k$-groups $U$ and $\widetilde{U}$ are equal because $U$ and $\widetilde{U}$ are smooth and connected (see \cite[II, \S 5, 5.5]{BDP}).
\label{sous_groupe_engendre}
\end{remark}

	\begin{lemma}
With the previous notations one has $N_G(\mathfrak{u})_{\red} = N_G(U)_{\red}$. 
\label{egalite_normalisateurs_reduits}
\end{lemma}

\begin{proof}
Only the inclusion $N_G(\mathfrak{u})_{\red} \subseteq N_G(U)_{\red}$ has to be checked as the reverse inclusion is ensured by the equality $\Lie(U) = \mathfrak{u}$. 

By Zariski density as $k$ is algebraically closed and because the schemes we consider are locally of finite type we only need to do the reasoning on $k$-points (see \cite[Corollaire 10.4.8]{EGA4c}). Let $g \in N_G(\mathfrak{u})_{\red}(k)$, and let $h \in U(k)$. As $U := \exp_{\mathfrak{b}}(\mathfrak{u})$ there exists an element $x \in \mathfrak{u}(k)$ such that $h = \exp_{\mathfrak{b}}(x) = \exp(x)$ (where the last equality comes from the preamble of \cite[Section 3.1]{JEA1} and Remark \ref{sous_groupe_engendre} because $x \in \mathfrak{u}(k)$ is a closed point). We still make use of the aforementioned remark to obtain that $\Ad(g)(h) = \Ad(g)(\exp(x)) = \exp(\Ad(g)(x)) \in U(k)$ because $g \in N_G(\mathfrak{u})_{\red}(k)$ (where the last equality is obtained by making use of the $G$-equivariance of $\exp$).
\end{proof}

The $\fppf$-formalism of Remark \ref{sous_groupe_engendre} allows us to obtain the equality at the group level:

\begin{lemma}
The subgroup $N_G(\mathfrak{u})$ normalises $U$. In particular the equality $N_G(U) = N_G(\mathfrak{u})$ holds true.
\label{inclusion_normalisateurs_U_et_u}
	\end{lemma}
	
	\begin{proof}
	Let $R$ be a $k$-algebra. For any $g \in N_G(\mathfrak{u})(R)$ and $h \in U(R)$, there exists an $\fppf$-covering $S \rightarrow R$ such that $h_S = \psi(x_1,s_1) \cdots \psi(x_n,s_n)$ where $x_i \in \mathfrak{u}_R \otimes_R S $ and $s_i \in S$. The $G$-equivariance of $\exp_{x_i}$ for any $i\in \{1, \cdots,n\}$ then leads to
	\[(\Ad(g)h)_S = \prod_{i=1}^n \Ad(g_S)\psi(x_i,s_i) = \prod_{i=1}^n \psi\left(\Ad(g_S)x_i,s_i\right) \in U(S)\cap G(R) = U(R).\]
Note that the equality $U(S)\cap G(R) = U(R)$ is satisfied because $U$ is generated by $\psi$ as an $\fppf$-sheaf. So we have obtained the inclusion of normalisers $N_G(\mathfrak{u})(R) \subseteq N_G(U)(R)$ for any $k$-algebra $R$. Yoneda's Lemma then allows us to conclude that $N_G(\mathfrak{u}) \subseteq N_G(U)$. As the Lie algebra $\mathfrak{u}$ derives from a unipotent smooth connected subgroup $U \subseteq G$ according to \cite[Proposition 3.1]{JEA1}, the inclusion $N_G(U) \subseteq N_{G}(\mathfrak{u})$ is still satisfied (see for instance Lemma \ref{reminder_normaliser}), hence the equality of normalisers.
	\end{proof}

	\subsection{Proof of Theorem \ref{Morozov_p_sup}}
	\label{proof_Morozov_psup}

\begin{proof}[Proof of Theorem \ref{Morozov_p_sup}]
As $p>\h(G)$ the restricted $p$-nil $p$-subalgebra $\mathfrak{u}$ can be integrated into a unipotent smooth connected subgroup $U\subset G$ (see \cite[Proposition 3.1]{JEA1}). This in particular means that $U\subseteq N_G(U)^0_{\red}$. Let us denote by $V$ the unipotent radical of $N_G(U)^0_{\red}$. The inclusion $U \subseteq V$ is necessarily satisfied. At the Lie algebras level this leads to the inclusion $\mathfrak{u} \subseteq \mathfrak{v}$.

By Veisfeiler--Borel--Tits Theorem the first point of Theorem \ref{Morozov_p_sup} is obtained if one can show that $\mathfrak{u} = \mathfrak{v}$ holds true. Note that this allows us to conclude that $U= V$ as $U$ and $V$ are smooth connected subgroups of $G$ (see \cite[II, \S5, 5.5]{DG}). 
 
Set $W := \Rad_U(N_G(\mathfrak{u})^{0}_{\red})$. According to \cite[Lemma 2.12]{JEA1} the Lie algebra $\mathfrak{w}:= \Lie(W)$ is a $p$-nil $p$-ideal of $\Lie(N_G(\mathfrak{u})^0_{\red}) \subseteq \Lie(N_G(\mathfrak{u}))$. As the subgroup $N_G(\mathfrak{u})$  is infinitesimally saturated, \cite[Theorem 2.5]{BDP} ensures that $W$ is a normal subgroup of $N_G(\mathfrak{u})$. Hence $\mathfrak{w}$ is a $p$-nil $p$-ideal of $\Lie(N_G(\mathfrak{u})) = N_{\mathfrak{g}}(\mathfrak{u})$. It is therefore contained in the $p$-radical of $N_{\mathfrak{g}}(\mathfrak{u})$, which is itself contained in the set of all $p$-nilpotent elements of $\rad(N_{\mathfrak{g}}(\mathfrak{u}))$ according to \cite[Lemma 2.5]{JEA1} ii). The latter is nothing but $\mathfrak{u}$ by assumption, whence the conclusion.

To summarise, we have obtained the following inclusions: $\mathfrak{w} \subseteq \mathfrak{u} \subseteq \mathfrak{v}$.

According to Lemma \ref{egalite_normalisateurs_reduits} the reduced normalisers $N_G(U)_{\red}$ and $N_G(\mathfrak{u})_{\red}$ are the same, so are respectively their connected components and their unipotent radicals. In other words, the groups $V$ and $W$ are equal, whence the equality of Lie algebras $\mathfrak{w}= \mathfrak{v}= \mathfrak{u}$. The groups $U \subseteq V$ are smooth and connected, therefore the equality of their Lie algebra lifts to an equality of groups, namely we have shown that $U = V$ (see \cite[II, 5.5]{DG}).

In other words $U$ is equal to $\Rad_U((N_G(U))^{0}_{\red})$. Its smooth connected normaliser is therefore a parabolic subgroup of $G$ (see to \cite[Corollaire 3.2]{TB}), we denote it $P_G(U):= N_G(U)^0_{\red}$.

As the normaliser $N_G(U)$ is infinitesimally saturated, the subgroup $N_G(U)_{\red}^0$ is normal in $N_G(U)$ (see \cite[Theorem 2.5]{BDP}). In other words one has $N_G(U) \subseteq N_G(N_G(U)^0_{\red})$. This leads to the equality $N_G(U)^0_{\red}=N_G(U)$ because $N_G(U)^{0}_{\red}$ is its own normaliser (see \cite[XXII, Corollaire 5.8.5]{SGA33}). Whence $N_G(U)$ is a parabolic subgroup of $G$.

The same reasoning as above also allows us to conclude that $N_G(\mathfrak{u}) = N_G(\mathfrak{u})^0_{\red}$ is a parabolic subgroup of $G$ (which is proper if $\mathfrak{u}\neq 0$. Note that this reasoning can be done because $N_G(\mathfrak{u})$ is infinitesimally saturated (see Lemma \ref{sat_inf_normalisateur}). Therefore the following equalities are satisfied:
\[N_G(\mathfrak{u}) = N_G(\mathfrak{u})^0_{\red} = N_G(U)^0_{\red} = N_G(U) = P_G(U),\]
\noindent and $N_{\mathfrak{g}}(\mathfrak{u}) = \Lie(N_{G}(\mathfrak{u})) = \Lie(P_G(U))$ is indeed a parabolic subalgebra of $\mathfrak{g}$ (namely it is the Lie algebra of the parabolic subgroup of $G$ obtained by applying the Borel--Tits theorem to $U$).

Furthermore Theorem \ref{obstruction} holds true because the normaliser $N_G(\mathfrak{u})^0_{\red} = N_G(U)$ are smooth. The normaliser $N_G(U)$ is therefore the optimal parabolic subgroup $P_G(\lambda_{\mathfrak{u}})$ obtained by Kempf--Rousseau theory (see Section \ref{Kempf-Rousseau}).
\end{proof}

\begin{remark}
Lemma \ref{inclusion_normalisateurs_U_et_u}, together with the $\fppf$-formalism it introduces, allows to drastically simplify the proof of Theorem \ref{Morozov_p_sup}. Indeed it ensures the equality of $N_G(U)$ and $N_G(\mathfrak{u})$ (and not only that of their smooth part). Lemma \cite[2.13]{JEA1} then allows us to conclude. The notations are the same as in the proof above. Because $U$ is unipotent smooth and connected implied the inclusion $U \subseteq V:= \Rad_U(N_G(U)^0_{\red})$ holds true. This leads to an inclusion of Lie algebras $\mathfrak{u}\subseteq \mathfrak{v}$. By assumption $\mathfrak{u} = \rad_p(N_{\mathfrak{g}}(\mathfrak{u}))$. Now, as $\mathfrak{v}:= \Lie(\Rad_U(N_{\mathfrak{g}}(\mathfrak{u})) \subset N_{\mathfrak{g}}(\mathfrak{u})$ is a $p$-nil ideal (see \cite[Lemma 2.12]{JEA1}) the inclusion $\mathfrak{v}\subseteq \mathfrak{u}$ is trivially satisfied. This leads to the equality $\mathfrak{u}=\mathfrak{v}$. We deduce from the latter (and as in the proof), that $N_G(U)$ is a parabolic subgroup of $G$. At this stage of the proof one already knew that $N_G(U) = N_G(\mathfrak{u})$ so that no additional argument is no longer needed to show that $N_G(\mathfrak{u})$ is parabolic.
\label{remaruqe_egalite_normalisateurs}
\end{remark}

\subsubsection{Remarks on the characteristic  \texorpdfstring{$0$}{Lg} or \texorpdfstring{$p$}{Lg} high enough}
\label{rem_car_0_ou_grande}

Let $G$ be a simple and simply connected $k$-group. When $p= 0$ or $p> 2\h(G) - 2$, the adjoint representation is compatible with the exponential map. In other words, any $\mathfrak{g}$-nilpotent element $x \in \mathfrak{g}$ satisfies $\Ad(\exp(tX))\cdot y = \exp(t\ad(X))\cdot y$ for any $y \in \mathfrak{g}$. In characteristic $p>0$, the assumption on $p$ is required because the $p$-structure of $\mathfrak{g}$ in particular leads to the following identity for any $x$ and $y \in \mathfrak{g}$ (see \cite[II, \S 7, 3.3]{DG}):
\[(x+y)^{[p]} = x^{[p]} + y^{[p]} - W(x,y).\]
This implies that the equality $\exp(x)\exp(y) =\exp(x+y)$ is no longer satisfied in general, when $x$ and $y$ are any $p$-nilpotent elements of $\mathfrak{g}$ (see \cite[Remarque 4.1.7]{SER1}). The assumption $p>2\h(G)-2$ allows to settle this issue as then the adjoint representation $\Ad$ is of low heigh. The proof of \cite[Lemma 4.3]{HS} details the reasoning when $G$ is of type $A_{n+1}$, this easily extends to the other types: let $x$ be a $p$-nilpotent element of $\mathfrak{g}$, denote by $l$ the left multiplication by $tx$ and by $r$ the right multiplication by $-tx$. Then 
\[\Ad(\exp(tx)) =  \exp(l)\exp(r) = \exp(l + r - \sum_{n=1}^{p-1} c_n l^{[n]} r^{[p-n]}),\] where the $c_n \in k$ are non trivial coefficients, whence the vanishing identity
	\begin{alignat*}{3}
	W_p(l,r) \: & = \sum_{n=1}^{p-1} c_n l^{[n]}r^{[p-n]}\\
			 \: & = \sum_{n=1}^{\h(G)-1} c_n l^{[n]}r^{[p-n]} + \sum_{n=\h(G)}^{p-1} c_n l^{[n]}r^{[p-n]} = 0
	\end{alignat*}
as in each sum one of the factor is raised to a power greater than $\h(G)$, thus vanishes, because when $p>\h(G)$, any nilpotent element $y \in \mathfrak{g}$ is killed by the $[p]$-power, namely $y^{[p]} = 0$ (see \cite[4.4]{MAUS}). 

Such a compatibility drastically simplifies the proof of Lemmas \ref{sat_inf_normalisateur} et \ref{sat_inf_norm_groupe}. Let us explain why in more details. 

Let $\mathfrak{u}$ be a nilpotent subalgebra of $\mathfrak{g}$. According to \cite[Proposition 3.1]{JEA1} this subalgebra integrates into a unipotent smooth connected subgroup of $G$, denoted by $U:= \exp_{\mathfrak{b}}(\mathfrak{u})$, where $\mathfrak{b}:=\Lie(B)\subseteq \mathfrak{g}$ is a Borel subalgebra such that $\mathfrak{u} \subseteq \radu(B)$. A priori only the inclusion $\Lie(N_G(U)) \subseteq N_{\mathfrak{g}}(\mathfrak{u})$ holds true (see Lemma \ref{reminder_normaliser}) but the compatibility between the adjoint representation and the exponential map ensures that these normalisers are actually equal:
\begin{lemma}
Let $k$ be an algebraically closed field of characteristic $p>0$ and $G$ be a simple and simply connected $k$-group such that $p>2\h(G)-2$. Let also $\mathfrak{u}$ be a $p$-nil subalgebra of $\mathfrak{g}$. Then the equality $N_{\mathfrak{g}}(\mathfrak{u})= \Lie(N_G(U))$ holds true.
\label{égalité_lie_normalisateurs}
\end{lemma}

\begin{proof}
One needs to show that any $x \in \in N_{\mathfrak{g}}(\mathfrak{u})$ actually belongs to $\Lie(N_G(U))$. According to Lemma \ref{reminder_normaliser}, this amounts to show that $\Ad(u)x - x \in \mathfrak{u}$ for any $u \in U(k)$. Let us thus consider $u \in U(k)$. As $U = \exp_{\mathfrak{b}}(\mathfrak{u})$ integrates $\mathfrak{u}$, the same reasons as in the proof of Lemma \ref{egalite_normalisateurs_reduits}, ensures the existence of an element $y \in \mathfrak{u}$ such that $u = \exp(y)$. Therefore we have:
\begin{alignat*}{3}
\Ad(u)x - x \: & = \Ad(\exp(y))(x) - x\\
 \: & = \exp(\ad(y))(x) -x, \text{ because } \Ad \text{ and } \exp \text{ are compatible,} \\
 \: & = \sum_{0 \leq n<p}\frac{\ad^n(y)}{n!} (x) -x \\
 \: & = \sum_{1\leq n<p}\frac{\ad^{n}(y)}{n!} (x) \in \mathfrak{u} \\
\end{alignat*}
because $x \in N_{\mathfrak{g}}(\mathfrak{u})$ by assumption. When $\Ad$ and $\exp$ are compatible (so in particular when $p>2\h(G)-2$) the equality $N_{\mathfrak{g}}(\mathfrak{u}) = \Lie(N_G(U))$ is thus satisfied.
\end{proof}

When $p>2\h(G)-2$ Lemmas \ref{sat_inf_normalisateur} and \ref{sat_inf_norm_groupe} are thus the same and their proof is immediate:

\begin{proof}[Proof of Lemmas \ref{sat_inf_normalisateur} and \ref{sat_inf_norm_groupe} when $p>2\h(G)-2$]
According to what precedes, and by making use of Lemma \ref{reminder_normaliser}, one can show that the equalities $\Lie(N_G(\mathfrak{u})) = N_{\mathfrak{g}}(\mathfrak{u})= \Lie(N_G(U))$ are satisfied. The purpose is to show that given any $p$-nilpotent element $x \in N_{\mathfrak{g}}(\mathfrak{u})$, the $t$-power map $t \mapsto \exp(tx)$ factors through $N_{G}(\mathfrak{u})$. Let $y\in \mathfrak{u}$ and $x \in \mathfrak{g}$ be a $p$-nilpotent element. As the adjoint representation is compatible with the exponential map, it is so with the $t$-power map. This leads to:
	 \begin{alignat*}{3}
	 \Ad(\exp(tx))(y) \: & = \exp(\ad(tx))(y)\\
	 \: & = \sum_{0\leq n <p} \frac{\ad(tx)^n}{n!}(y) \in \mathfrak{u}\\
	 \end{alignat*}
because $x \in N_{\mathfrak{g}}(\mathfrak{u})$ by assumption.
\end{proof}

\begin{remark}
When $G$ is any reductive $k$-group, the following assumptions allow, in a certain framework, to weaken the condition on $p$ without losing the compatibility between a representation of $G$ and the exponential map.

		Let us denote by $\haut_G(\mathfrak{g}):= \max\{\sum_{\alpha>0}\langle \lambda, \alpha^{\vee}\rangle\}$ the Dynkin height of the adjoint representation. The restriction on $p$ amounts to require the adjoint representation to be of low height. This means that $p > \haut_G(\mathfrak{g})$ as $\haut_G(\mathfrak{g}) = 2 \times \sum_{i=1}^{n} \bar{\omega}_i = 2\h(G)-2$, where $\bar{\omega}_i$ are the fundamental weights for the $T$-action (see \cite[Lecture 4, example of Theorem 8]{SER}). 
		
		More generally, one can weaken the restrictions on $p$ by considering an almost faithful representation $\rho$ (i.e. a representation whose kernel is of multiplicative type) and whose Dynkin height is minimal. In this case, the exponential map and the representation $\rho$ are compatible (see \cite[4.5 et 4.6]{BDP}) and the previous reasoning can be extended to this context. Let us note that the adjoint representation is always almost faithful when $G$ is a reductive $k$-group because its kernel is nothing but the center of $G$, which is of multiplicative type. 
		When $G$ is simple and simply connected, if $G$ is not of type $F_4$, $E_6$, $E_7$ or $E_8$, a low height almost faithful representation always exists when $p>\h(G)$ (see the remark that follows Theorem 4.7 in \cite{BDP}).

	\label{petite_hauteur}
\end{remark}

\begin{remark}
The lack of compatibility between the exponential map and the adjoint representation has been studied in \cite[Chapter 4]{Kra}. In paragraph $4.2$, the author defines an operator, denoted by $\had$, compatible with the truncated exponential.

\begin{defn}
A subalgebra $\mathfrak{u} \subseteq \mathfrak{g}$ is $\had$-weakly saturated if and only if $\had(x)(\mathfrak{u}) \subseteq \mathfrak{u}$ for any $p$-nilpotent element $x \in N_{\mathfrak{g}}(\mathfrak{u})$. 
\end{defn}

In practice when $p>\h(G)$, a restricted $p$-nil $p$-subalgebra $\mathfrak{u} \subset \mathfrak{g}$ is $\had$-weakly saturated if and only if its normaliser $N_G(\mathfrak{u})$ is infinitesimally saturated. One inclusion is clear: assume $\mathfrak{u}$ to be $\had$-weakly saturated and let $y \in \mathfrak{u}$. For any $p$-nilpotent element $x \in N_{\mathfrak{g}}(\mathfrak{u})$ and for any $t \in R$ (where $R$ is a $k$-algebra) we have:
\begin{alignat*}{3}
\Ad(\exp(tx))(y) \:& = \exp(t\had(x))(y)\\
\:& = \left(\sum_{i=0}^{p-1} \frac{(t \had(x))^i}{i!}\right)(y) \in \mathfrak{u},\\
\end{alignat*}
\noindent as any coefficient $\had^i(x)(y) \in \mathfrak{u}$. Conversely if $N_{G}(\mathfrak{u})$ is infinitesimally saturated let $x$ be a $p$-nilpotent element of $N_{\mathfrak{g}}(\mathfrak{u})$. For any $y \in \mathfrak{u}$ and $t \in \mathbb{G}_a$ we have:
\begin{alignat*}{3}
\Ad(\exp(tx))(y)\: & = \left(\sum_{i=0}^{p-1} \frac{(t \had(x))^i}{i!}\right)(y) \in \mathfrak{u},\\
\end{alignat*}
\noindent because $N_G(\mathfrak{u})$ is infinitesimally saturated. 
In particular $\had(x)(y) \in \mathfrak{u}$ thus $\mathfrak{u}$ is weakly $\had$-saturated. Let us remark that in particular Lemma \ref{sat_inf_normalisateur} allows us to conclude that if $p>\h(G)$ any restricted $p$-nil $p$-subalgebra of $\mathfrak{g}$ is $\had$-saturated.
\end{remark}

\subsection{Proof of Corollaries \ref{cor_p_elements_p_sup} and \ref{cor_towers_p_sup}}
\label{section_cor_p_grand}

Corollary \ref{cor_p_elements_p_sup} follows from Theorem \ref{Morozov_p_sup} thanks to the following result:

\begin{corollary}[of {\cite[Lemma 2.13]{JEA1}}]
Let $k$ be an algebraically closed field of characteristic $p>0$ and let $G$ be a reductive $k$-group. Assume that $p>\h(G)$. Let $H \subseteq G$ be an infinitesimally saturated group. Then the following equalities hold true:
\begin{alignat*}{2}
\rad_p(\mathfrak{h}) \: & = \{x \in \rad(\mathfrak{h}) \mid  x \text{ is } p\text{-nilpotent}\} \\
\: & =  \{x \in \rad_p(\Lie(H_{\red})) \mid  x \text{ is } p\text{-nilpotent}\} = \rad_p(\Lie(H_{\red})) =\radu(H^{0}_{\red}).
\end{alignat*}
\label{sat_inf_et_rad_p}
\end{corollary}

\begin{proof}
Let us note that the assumptions of Lemma \ref{p_rad_inf_sat} are satisfied here as:
\begin{enumerate}
\item the unipotent radical of $H^0_{\red}$ is normal in $H$ and the quotient $H/H^0_{\red}$ is a group of multiplicative type according to \cite[Theorem 2.5]{BDP}, because $H$ is infinitesimally saturated,
\item the characteristic is always as requested in Lemma \ref{p_rad_inf_sat} because we assume the characteristic $p$ to be greater than $\h(G)$. This can be checked with Table \ref{Hypothèses_sur_la caractéristiqu_ pour_un_groupe_simple}.
\end{enumerate}
By making use of the identity $\Lie(H^0_{\red}) = \Lie(H_{\red})$, we obtain the desired equalities.
\end{proof}

\begin{proof}[Proof of Corollary \ref{cor_p_elements_p_sup}]
The normaliser of $\mathfrak{u}$ is infinitesimally saturated according to Lemma \ref{sat_inf_normalisateur} because $p>\h(G)$. Corollary \ref{sat_inf_et_rad_p} thus ensures that $\mathfrak{u}$ is a subalgebra of $N_{\mathfrak{g}}(\mathfrak{u})$ given by the set of all $p$-nilpotent elements of $\rad(N_{\mathfrak{g}}(\mathfrak{u}))$. Theorem \ref{Morozov_p_sup} hence holds true and allows to conclude.
\end{proof}

\begin{remark}
To show Corollary \ref{cor_p_elements_p_sup} one needs to combine Theorem \ref{Morozov_p_sup} with Proposition \ref{sat_inf_et_rad_p}. This in particular underlines that, when $k$ is of characteristic $p>0$, the $p$-radical of a Lie algebra that derives from a reductive $k$-group is indeed the good analogue of the nilradical of a semisimple Lie algebra in characteristic $0$. Nevertheless, a much simpler proof is possible, by making use of the exponential map as defined in \cite[Proposition 3.1]{JEA1}: by definition the $p$-radical of $N_{\mathfrak{g}}(\mathfrak{u})$ is a $p$-nil $p$-ideal. According to \cite[Proposition 3.1]{JEA1} the ideal $\mathfrak{u}$ can be integrated into a unipotent smooth connected subgroup $U\subseteq G$. By construction $U$ is normal in $N_G(U)_{\red}^0$ it is therefore necessarily contained into the unipotent radical of $N_G(U)^0_{\red}$, denoted by $V$. The inclusion being preserved by derivation, this leads to the inclusion of Lie algebras $\mathfrak{u} \subseteq \mathfrak{v}$. According to Lemma \ref{sat_inf_norm_groupe} the group $N_G(U)$ is infinitesimally saturated. The subgroups $V$ and $N_G(U)_{\red}^0$ are therefore normal in $N_G(U)$ (see \cite[Theorem 2.5]{BDP}). Thus $\mathfrak{v}$ is an ideal of $\Lie(N_G(U))$. As the normalisers $N_G(U)$ and $N_G(\mathfrak{u})$ are equal (see Lemma \ref{inclusion_normalisateurs_U_et_u}), we actually have shown that $\Lie(N_G(U)) = \Lie(N_{G}(\mathfrak{u})) = N_{\mathfrak{g}}(\mathfrak{u})$ (where the last equality is ensured by Lemma \ref{reminder_normaliser}). The Lie algebra of $V$ is therefore a $p$-nil $p$-ideal of $N_{\mathfrak{g}}(\mathfrak{u})$ (see \cite[Lemma 2.12]{JEA1}). Hence it is necessarily contained in $\rad_p(N_{\mathfrak{g}}(\mathfrak{u}))$ which is the Lie algebra $\mathfrak{u}$ itself by assumption. In other words we have shown the equality $\mathfrak{u} = \mathfrak{v}$, whence the equality of the smooth connected groups $U$ and $V$ (by making use of \cite[II, \S5, 5.5]{DG}). As in the proof of Theorem \ref{Morozov_p_sup}, Veisfeiler--Borel--Tits Theorem (see \cite[Corollaire 3.2]{TB}) allows us to conclude.
\end{remark}

The proof of Corollary \ref{cor_towers_p_sup} is now immediate:

\begin{proof}[Proof of Corollary \ref{cor_towers_p_sup}]
By n\oe therianity of $\mathfrak{g}$ there exists an integer $i$ for which the sequences $(\mathfrak{u}_i)_{i \in \mathbb{N}}$ and $(\mathfrak{q}_i)_{i\in \mathbb{N}}$ of the statement stabilise. Let us denote by $\mathfrak{u}_{\infty}$, respectively $\mathfrak{q}_{\infty}$, the limit objects.

According to Corollary \ref{sat_inf_et_rad_p}, as $N_{\mathfrak{g}}(\mathfrak{u}_{\infty})$ is infinitesimally saturated (see Lemma \ref{sat_inf_normalisateur}) the Lie algebra $\mathfrak{u}_{\infty}$ is the $p$-radical of its normaliser $N_{\mathfrak{g}}(\mathfrak{u}_{\infty})$. According to Theorem \ref{Morozov_p_sup} this normaliser is a parabolic subalgebra of $\mathfrak{g}$. The aforementioned Corollary ensures that:
\begin{itemize}
\item this algebra derives from the subgroup $P(U_{\infty})$ obtained by applying Veisfeiler--Borel--Tits Theorem (see \cite[Corollaire 3.2]{TB}) to the group $U_{\infty}$ that integrates $\mathfrak{u}_{\infty}$,
\item this parabolic subalgebra is also the Lie algebra of the optimal parabolic subgroup of $\mathfrak{u}$ denoted by $P_G(\lambda_{\mathfrak{u_{\infty}}})$.
\end{itemize}
\end{proof}

%

%

\subsection{Proof of Corollary \ref{cor_Mor_pgrand}}
\label{section_cor_Mor_pgrand}
Theorem \ref{Morozov_p_sup} allows us to show Corollary \ref{cor_Mor_pgrand}:  

\begin{proof}[Proof of Corollary \ref{cor_Mor_pgrand}]
Let $\mathfrak{u}:= \rad_p(\mathfrak{q})$, there are two possibilities:
\begin{enumerate}
	\item either $\mathfrak{u}=\{0\}$ and so $\mathfrak{q}$ is $p$-reductive;
	\item or $\mathfrak{u}\neq\{0\}$. In this case let us note that the inclusion $\mathfrak{q} \subseteq N_{\mathfrak{g}}(\mathfrak{u})$ is satisfied and that $\mathfrak{q}$ is maximal by assumption. There are again two possibilities:
		\begin{enumerate}
			\item either $N_{\mathfrak{g}}(\mathfrak{u}) = \mathfrak{g}$ which is impossible because $\mathfrak{g}$ is reductive thus $p$-reductive,
			\item or $N_{\mathfrak{g}}(\mathfrak{u}) = \mathfrak{q}$. In this case, as $p>\h(G)$, Theorem \ref{Morozov_p_sup} allows us to conclude that $N_{\mathfrak{g}}(\mathfrak{u})$ is the Lie algebra of a parabolic subgroup of $G$.
		\end{enumerate}
\end{enumerate}
\end{proof}

\section{Analogue and corollaries of Morozov's Theorem in separably good characteristics}
\label{section_morozov_p_inf}

Let $k$ be an algebraically closed field of characteristic $p>0$ and $G$ be a reductive $k$-group. This section is dedicated to show that Morozov's Theorem admits an analogue and corollaries when $p$ is separably good for $G$. Before stating the results we are going to prove, let us briefly remind of some specificities raised by this setting.

We are going to reproduce and adapt arguments developed in Section \ref{chapitre_morozov_p_sup} (in which we assumed $p>\h(G)$). These required the existence of a unipotent smooth connected subgroup of $G$ which integrates $\mathfrak{u}$. Unfortunately, as mentioned in the introduction of this article, when $p$ is separably good for $G$, not any restricted $p$-nil $p$-subalgebra $\mathfrak{u} \subset \mathfrak{g}$ can be integrated into a unipotent smooth connected subgroup of $G$ (see also \cite[Section 3.3]{JEA1}). However, assuming $p$ to be separably good for $G$ ensures the existence of a Springer isomorphism $\phi : \mathcal{N}_{\red}(\mathfrak{g}) \rightarrow \mathcal{V}_{\red}(G)$ and allows us to lift nilpotent elements of $\mathfrak{g}$ to unipotent elements of $G$. This is detailed, for instance, in \cite[Section 3.1]{JEA1}. For any $p$-nilpotent element $x \in \mathfrak{g}$ this punctual integration induces a morphism, called a $t$-power map:
	\begin{alignat*}{3}
		\phi_x : \: & \mathbb{G}_a \: & \rightarrow \: & G,\\
			\:& t \: & \mapsto \: & \phi(tx).
	\end{alignat*} 
\noindent This, combined with the $\fppf$-formalism developed in \cite[VIB, Proposition 7.1]{SGA31}, leads to associate to any restricted $p$-nil $p$-subalgebra $\mathfrak{v} \subset \mathfrak{g}$, a unipotent smooth connected subgroup $J_{\mathfrak{v}}$. This group can be huge in general, but when the subalgebra $\mathfrak{v}$ fulfils the requirements of Statement \ref{analogue_morozov} the subgroup $J_{\mathfrak{v}}$ actually integrates $\mathfrak{v}$ (see \cite[Lemma 5.1]{JEA1}). Moreover as already seen in the previous section, even if this integration exists, nothing ensures that it is compatible with the adjoint representation. This also explains why some extra conditions are required on the normaliser of $N_G(\mathfrak{u})$ in order to settle these issues.

The lack of smoothness of normalisers is also an issue specific to the characteristic $p>0$ framework. Under the assumption $p > \h(G)$, we have shown that any normaliser of any restricted $p$-nil $p$-Lie subalgebra is infinitesimally saturated (see Lemma \ref{sat_inf_normalisateur}). This allowed to settle the lack of smoothness of normalisers by making use of \cite[Theorem 2.5]{BDP}. This property is satisfied when $p>\h(G)$. This can be shown by making use of a structure theorem for unipotent groups because any restricted $p$-nil $p$-subalgebra $\mathfrak{v} \subset \mathfrak{g}$ is integrable into a unipotent smooth connected subgroup $V \subset G$. 

As explained in introduction, the notion of infinitesimal saturation naturally extends to separably good characteristics with that of $\phi$-infinitesimal saturation, where $\phi : \mathcal{N}_{\red}(\mathfrak{g}) \rightarrow \mathcal{V}_{\red}(G)$ is a Springer isomorphism for $G$. This allows us to obtain a variation of a result of P. Deligne (see \cite[Theorem 1.1 and Proposition 4.12]{JEA1}). However the difficulty raised in this paragraph is only partially settled as it is known that some restricted $p$-nil $p$-Lie subalgebra of $\mathfrak{g}$ cannot be integrated when $p<\h(G)$. This prevents us from using the aforementioned structural argument useful in the proof of Lemma \ref{sat_inf_normalisateur} and justifies the additional assumption in the statement of Theorem \ref{morozov_p_inf} (with respect to the statement in characteristic $p>\h(G)$). 

In Section \ref{section_mor_p_inf} we show that Statement \ref{analogue_morozov} holds true in separably good characteristics under an extra assumption on $\phi$-infinitesimal saturation of some normalisers: 
\begin{theorem}[Analogue of Morozov's Theorem when $p$ is separably good]
Let $k$ be an algebraically closed field of characteristic $p>0$ and $G$ be a reductive $k$-group. Assume that $p$ is separably good for $G$. Let $\phi: \mathcal{N}_{\red}(\mathfrak{g}) \rightarrow \mathcal{V}_{\red}(G)$ be a Springer isomorphism for $G$. If $\mathfrak{u} \subseteq \mathfrak{g}$ is the $p$-radical of $N_{\mathfrak{g}}(\mathfrak{u})$ and if $N_{G}(\mathfrak{u})$ is $\phi$-infinitesimally saturated then:
	\begin{enumerate}
		\item the normaliser $N_\mathfrak{g}(\mathfrak{u})$ is a parabolic subalgebra of $\mathfrak{g}$,
		\item this parabolic subalgebra satisfies the equalities $N_\mathfrak{g}(\mathfrak{u})= \mathfrak{p_g}(J_{\mathfrak{u}}) = \mathfrak{p_g}(\lambda_{\mathfrak{u}})$ 
where $J_{\mathfrak{u}} \subset G$ is a unipotent smooth connected subgroup such that $\Lie(J_{\mathfrak{u}}) = \mathfrak{u}$ (such a subgroup exists according to \cite[Lemma 5.1]{JEA1}).			
	\end{enumerate}
\label{morozov_p_inf}
\end{theorem}

The conditions required in this theorem are by far less restrictive than the one required in Theorem \ref{Morozov_p_sup}. Furthermore, let us stress out that $N_{G}(\mathfrak{u})$ must be $\phi$-infinitesimally saturated for Theorem \ref{morozov_p_inf} to hold to be true as any parabolic subgroup is $\phi$-infinitesimally saturated according to \cite[Lemma 4.2]{JEA1}. 

As in the characteristic $p>\h(G)$ case, Theorem \ref{morozov_p_inf} admits the following reformulation by means of the $p$-nilpotent elements of the radical of the normaliser:
\begin{corollary}
Let $k$ be an algebraically closed field of characteristic $p>0$ and $G$ be a reductive $k$-group. Assume that $p$ is separably good for $G$. Let $\mathfrak{u} \subseteq \mathfrak{g}$ be a restricted $p$-Lie subalgebra such that $\mathfrak{u}$ is the set of $p$-elements of $\rad(N_{\mathfrak{g}}(\mathfrak{u}))$. Let us also assume that $N_{G}(\mathfrak{u})\subseteq G$ is $\phi$-infinitesimally saturated. Then $N_\mathfrak{g}(\mathfrak{u})$ is the Lie subalgebra obtained in Theorem \ref{morozov_p_inf}. 
\label{cor_p_nilp_elemnt_p_inf}
\end{corollary}
This is shown in Section \ref{section_cor_p_inf}. Furthermore, Theorem \ref{morozov_p_inf} is a particular case of the following limit corollary, for which the tower of normalisers stabilises immediately:
	
\begin{corollary}
Under the assumptions of Theorem \ref{morozov_p_inf}, let $\mathfrak{u}$ be a restricted $p$-nil $p$-subalgebra of $\mathfrak{g}$. Then:
\begin{enumerate}
	\item the tower of normalisers of $\mathfrak{u}$ converges to a parabolic subalgebra of $\mathfrak{g}$, 
	\item this limit object satisfies the following equalities $N_{\mathfrak{g}}(\mathfrak{u}_{\infty}) = \mathfrak{p_g(\lambda_{u_{\infty}})} = \Lie(N_G(J_{\mathfrak{u}_{\infty}}))$ where the subgroup $N_G(J_{\mathfrak{u}_{\infty}})$ is the limit object of the tower of smooth normalisers associated to $\Rad_U(N_G(\mathfrak{u})^0_{\red})$.
\end{enumerate}
\label{cor_parab_associé_inf}
\end{corollary}

Finally, in Section \ref{section_cor_Mor_pgrand} we show that Statement \ref{cor_morozov} holds true in separably good characteristics:

\begin{corollary}
Let $k$ be an algebraically closed field of very good characteristic. Let $G$ be a reductive $k$-group and assume that $p$ is separably good for $G$. Let $\mathfrak{q} \subsetneq \mathfrak{g}$ be a maximal proper restricted $p$-nil subalgebra. Then $\mathfrak{q}$ is either $p$-reductive or parabolic.
\end{corollary}

Finally, section \ref{section_cor_Mor_pinf} is dedicated to prove the following corollary of Theorem \ref{morozov_p_inf}:

\begin{corollary}[of Theorem \ref{morozov_p_inf}]
Let $k$ be an algebraically closed field of characteristic $p>0$ and let $G$ be a reductive $k$-group. Assume that $p>3$ is very good for $G$. Let $\mathfrak{q} \subseteq \mathfrak{g}$ be a maximal proper restricted $p$-Lie subalgebra. Then $\mathfrak{q}$ is either $p$-reductive (see Definition \ref{def_pred}) or parabolic.   
\label{cor_Mor_pinf}
\end{corollary}

\subsection{Proof of Theorem \ref{morozov_p_inf}}
\label{section_mor_p_inf}
Let us remind the reader of the notations of Theorem \ref{morozov_p_inf}: in what follows $k$ is an algebraically closed field of characteristic $p>0$ and $G$ is a reductive $k$-group. The characteristic $p$ is assumed to be separably good for $G$. Let $\mathfrak{u}\subseteq \mathfrak{g}$ be a restricted $p$-subalgebra such that $\mathfrak{u}$ is the set of $p$-nilpotent elements of $\rad(N_{\mathfrak{g}}(\mathfrak{u}))$.

\begin{proof}[Proof of Theorem \ref{morozov_p_inf}]
The proof is the same as the one of the analogue of Morozov's Theorem in characteristic $p>\h(G)$ (see Theorem \ref{Morozov_p_sup}). According to \cite[Lemma 5.1]{JEA1} the restricted $p$-nil $p$-Lie algebra $\mathfrak{u}$ integrates into a unipotent smooth connected subgroup $J_{\mathfrak{u}}\subset G$, and we aim to show that $N_G(\mathfrak{u}) = P_G(J_{\mathfrak{u}})$. Let us note that this directly implies the smoothness of $N_G(\mathfrak{u})^0$ because a parabolic subgroup is smooth (see \cite[XXVI, Définition 1.1]{SGA33}). According to Section \ref{Kempf-Rousseau}, and in particular the equivalent conditions $2.$ and $3.$ of Theorem \ref{obstruction}, what precedes is equivalent to require the equality $N_{\mathfrak{g}}(\mathfrak{u})= \mathfrak{p_g}(\lambda_{\mathfrak{u}})$ to hold true.

As by assumption the subgroup $N_G(\mathfrak{u})$ is $\phi$-infinitesimally saturated and satisfies hypotheses of \cite[Theorem 1.1]{JEA1}, the connected component of its reduced part $N_G(\mathfrak{u})^0_{\red}$ is normal in $N_G(\mathfrak{u})$. In particular if $N_G(\mathfrak{u})^0_{\red}$ is a parabolic subgroup of $G$ then one has $N_G(\mathfrak{u}) = N_G(\mathfrak{u})^0_{\red}$ because a parabolic subgroup is its own normaliser (see \cite[XXII, Corollaire 5.8.5]{SGA33}). This is enough to show that $N_G(\mathfrak{u})^0_{\red}$ is a parabolic subgroup of $G$. 

According to \cite[Lemma 3.5]{JEA1} the normalisers $N_G(\mathfrak{u})$ and $N_G(J_{\mathfrak{u}})$ are equal, so are the connected components of their reduced parts. Hence it only remains to show that $N_G(J_{\mathfrak{u}})^0_{\red}$ is a reduced parabolic subgroup of $G$. Mimicking the proof of Theorem \ref{Morozov_p_sup} we aim to make use of Veisfeiler--Borel--Tits Theorem (\cite[Corollaire 3.2]{TB}) to conclude that $N_G(J_{\mathfrak{u}})^0_{\red}= P_G(J_{\mathfrak{u}})$. Let $V:= \Rad_U(N_G(J_{\mathfrak{u}})^0_{\red})$, this requires to show that $N_G(J_{\mathfrak{u}})^0_{\red} = N_G(V)^0_{\red}$.

The subgroup $J_{\mathfrak{u}}$ is a unipotent smooth connected normal subgroup of $N_G(J_{\mathfrak{u}})^0_{\red}$ (see the preamble of \cite[Section 3.2]{JEA1}) contained in the unipotent radical of $N_G(J_{\mathfrak{u}})^0_{\red}$, whence the inclusion of Lie algebras $\mathfrak{j_u} \subseteq \mathfrak{v}$.

Finally, let us note that the unipotent subgroup $V$ is normal in $N_G(\mathfrak{u})$ because by assumption $N_G(\mathfrak{u})= N_G(J_{\mathfrak{u}})$ is $\phi$-infinitesimally saturated (see \cite[Theorem 1.1]{JEA1}). The Lie algebra of $V$ is therefore a $p$-nil $p$-ideal of $N_{\mathfrak{g}}(\mathfrak{u})$ (see \cite[Lemma 2.12]{JEA1}). The situation is hence the following $\mathfrak{v} \subseteq \rad_p(N_{\mathfrak{g}}(\mathfrak{u})) \subseteq \mathfrak{u} =\Lie(\mathfrak{j_u})$ where we make use of \cite[Lemma 2.5]{JEA1} because $\mathfrak{u}$ is the set of $p$-nilpotent elements of $\rad(N_{\mathfrak{g}}(\mathfrak{u}))$. Let us note that the last equality is ensured by \cite[Lemma 5.1]{JEA1} which ensures that the subgroup $J_{\mathfrak{u}}$ integrates $\mathfrak{u}$. As a conclusion we have shown that $\mathfrak{v} = \mathfrak{j_u}$.
The subgroups $V$ and $J_{\mathfrak{u}}$ are smooth, connected and their Lie algebras are equal, so they are according to \cite[II, \S5, 5.5]{DG}. Thus the subgroup $V$ is the unipotent radical of its smooth and connected normaliser. The latter is therefore the parabolic subgroup $P_G(J_{\mathfrak{u}}) \subseteq G$ obtained by applying Veisfeiler--Borel--Tits Theorem \cite[Corollaire 3.2]{TB}).
\end{proof}

\begin{remarks}
The proof of Theorem \ref{morozov_p_inf} leads to the two following remarks:
\begin{enumerate}
	\item the assumption of $\phi$-infinitesimal saturation a priori ensures that the smooth connected part of $N_G(\mathfrak{u})$ contains a torus. This is actually a necessary condition for the statement to hold true. Indeed, let $H \subseteq N_G(\mathfrak{u})$ be a maximal connected subgroup of multiplicative type and assume that $N_G(\mathfrak{u})^0_{\red}$ does not contain any torus. Then, the intersection $H \cap N_G(\mathfrak{u})^0_{\red}$ is non trivial and $N_G(\mathfrak{u})^0_{\red}$ is unipotent. The last point can for instance be deduced from the maximality of $H$ (that ensures that $N_G(\mathfrak{u})^0_{\red}$ has no factors of type $\mu_p$) coupled with the assumption of $k$ (which is algebraically closed). So we have that $J_{\mathfrak{u}} = N_G(\mathfrak{u})^0_{\red}$ in the proof of Theorem \ref{morozov_p_inf}, in other words $N_G(\mathfrak{u})^0_{\red}$ is both parabolic  and  unipotent, which is absurd.

	\item Moreover, the infinitesimal version of \cite[Theorem 2.5]{BDP} obtained in \cite[Proposition 4.14]{JEA1}  already allowed us to show an infinitesimal version of Theorem \ref{morozov_p_inf}, stated below.
	\end{enumerate}

\end{remarks}

	\begin{proposition}
	Let $k$ be an algebraically closed field and $G$ be a reductive $k$-group. Assume that $p$ is separably good for $G$ and let $\phi: \mathcal{N}_{\red}(\mathfrak{g}) \rightarrow \mathcal{V}_{\red}(G)$ be a Springer isomorphism for $G$. Let $\mathfrak{u} \subset \mathfrak{g}$ be a restricted $p$-subalgebra and assume that:
	\begin{enumerate}
	\item the subalgebra $\mathfrak{u}$ is the set of all $p$-nilpotent elements of $N_{\mathfrak{g}}(\mathfrak{u})$;
	\item the normaliser $N_{G}(\mathfrak{u})$ is $\phi$-infinitesimally saturated.
\end{enumerate}	
Then the normaliser $N_{\mathfrak{g}}(\mathfrak{u})\subseteq \mathfrak{g}$ is a parabolic Lie subalgebra.
\label{morozov_infinitesimal}
	\end{proposition} 
	
	\begin{proof}
	The restricted $p$-nil $p$-subalgebra $\mathfrak{u} \subset \mathfrak{g}$ of the statement integrates into a unipotent smooth connected subgroup $J_{\mathfrak{u}}\subset G$. According to \cite[Lemma 3.5]{JEA1} the normalisers $N_G(\mathfrak{u})$ and $N_G(J_{\mathfrak{u}})$ are equal. The subgroup $J_{\mathfrak{u}}$ is unipotent smooth connected and normal in $N_G(\mathfrak{u})$, thus in $N_G(\mathfrak{u})^0_{\red}$, hence the inclusion $J_{\mathfrak{u}}\subseteq V:=\Rad_U(N_G(\mathfrak{u}))^{0}_{\red}$. One needs to get the reverse inclusion. Once again, and as the groups involved here are smooth and connected, this can be checked on their Lie algebras (see \cite[II, \S 5, n\degree 5.5]{DG}). According to \cite[Lemma 2.12]{JEA1} the Lie algebra $\mathfrak{v}$ is a restricted $p$-nil $p$-subalgebra of $N_{\mathfrak{g}}(\mathfrak{u})$. As $N_G(\mathfrak{u})^0$ is $\phi$-infinitesimally saturated, the Lie algebra $\mathfrak{v}$ is an ideal of $\Lie(N_{G}(\mathfrak{u}))= N_{\mathfrak{g}}(\mathfrak{u})$ (see the second point of \cite[Proposition 4.14]{JEA1}). In other words $\mathfrak{v}$ is a $p$-nil $p$-ideal of $N_{\mathfrak{g}}(\mathfrak{u})$ thus is contained in $\mathfrak{u}$ (which is nothing but the set of all $p$-nilpotent elements of the radical of $N_{\mathfrak{g}}(\mathfrak{u})$ according to \cite[Lemma 2.5]{JEA1}). Hence $J_{\mathfrak{u}}= V$ is the unipotent radical of its smooth connected normaliser. According to \cite[Corollaire 3.2]{TB} this normaliser is the parabolic subgroup $P_G(J_{\mathfrak{u}})\subseteq G$. In particular this means that $\mathfrak{u}$ is the Lie algebra of the unipotent radical of a parabolic subgroup. It is therefore the Lie algebra of the unipotent radical of the optimal parabolic subgroup associated to it (and obtained in section \ref{Kempf-Rousseau} via the Hilbert--Mumford--Kempf--Rousseau method). In other words we have shown that $\mathfrak{p_g(\lambda_u)} = \mathfrak{p_g(j_u)}$, whence the infinitesimal version of Theorem \ref{morozov_p_inf}. 
\end{proof}

\begin{remark}
	When $G$ is of type $(RA)$, \cite[Proposition 4.14]{JEA1} is even enough to show the ``strong'' version of Morozov's Theorem \ref{morozov_p_inf}. We detail the reasoning here. Let us also emphasize that, in practice, this assumption together with the additional assumption of separably good characteristic amounts to avoid factors of type $A_{pm-1}$ in the decomposition of $G$.
	
In order to simplify the notations we set $P:=P_G(J_{\mathfrak{u}})$ so $\mathfrak{p}:=\mathfrak{p_g(j_u)} \subseteq \mathfrak{g}$. Assume $G$ to be of type $(RA)$. Then any of its parabolic subgroup $Q\subseteq G$ is such that $Q = N_G(\mathfrak{q})^0$ (see \cite[XXII, Proposition 5.3.4]{SGA33}). In other words the situation is the following: 
	\[\mathfrak{p} \unlhd \N_{\mathfrak{g}}(\mathfrak{u}) \subseteq N_{\mathfrak{g}}(\mathfrak{p}) = \Lie(N_G(\mathfrak{p})^0), \]
by virtue of the first point of \cite[Proposition 4.14]{JEA1}. This leads to the following inequalities: 
\[\dim(\mathfrak{p}) = \dim(P) \leq \dim(N_{G}(\mathfrak{u}))\leq \dim(N_{\mathfrak{g}}(\mathfrak{u}) \leq \dim(N_G(\mathfrak{p})^0)= \dim(P)\]
and allows us to conclude that $P_G(J_{\mathfrak{u}}) = N_{G}(\mathfrak{u})^0_{\red} = N_G(\mathfrak{u})$. The equality $N_G(\mathfrak{u}) = P_G(\lambda_{\mathfrak{u}})$ follows, as explained in the above proof.
	
\end{remark}

\subsection{Proofs of Corollaries \ref{cor_p_nilp_elemnt_p_inf} and \ref{cor_parab_associé_inf}}
\label{section_cor_p_inf}
The preceding paragraphs aimed to adapt techniques and notions developed in Section \ref{chapitre_morozov_p_sup} when $p>\h(G)$ to separably good characteristics. This work being done, proofs of Corollaries \ref{cor_p_nilp_elemnt_p_inf} and \ref{cor_parab_associé_inf} will be the same as those of Corollaries \ref{cor_p_elements_p_sup} and \ref{cor_towers_p_sup}. In particular Corollary \ref{cor_p_nilp_elemnt_p_inf} follows from Theorem \ref{morozov_p_inf} thanks to the following result:

\begin{corollary}[of {\cite[Lemma 2.13]{JEA1}}]
Let $k$ be an algebraically closed field of characteristic $p\geq 3$ and let $G$ be a reductive $k$-group. Assume that $p$ is separably good for $G$ and let $\phi : \mathcal{N}_{\red}(\mathfrak{g}) \rightarrow \mathcal{V}_{\red}(G)$ be a Springer isomorphism for $G$. Let $H \subseteq G$ be a $\phi$-infinitesimally saturated group. Then the following equalities hold true:
\begin{alignat*}{2}
\rad_p(\mathfrak{h}) \: & = \{x \in \rad(\mathfrak{h}) \mid  x \text{ is } p\text{-nilpotent}\} \\
\: & =  \{x \in \rad(\Lie(H_{\red})) \mid  x \text{ is } p\text{-nilpotent}\} = \rad_p(\Lie(H_{\red})) =\radu(H^{0}_{\red}).
\end{alignat*}
\label{phi_sat_inf_et_rad_p}
\end{corollary}

\begin{proof}
Let us note that assumptions of Lemma \ref{p_rad_inf_sat} are satisfied here. Indeed:
\begin{enumerate}
\item the unipotent radical of $H^0_{\red}$ is normal in $H$ and the quotient $H/H^0_{\red}$ is a group of multiplicative type because $H$ is infinitesimally saturated (see to \cite[Theorem 1.1]{JEA1}), 
\item the characteristic is always as requested in the statement of Lemma \ref{p_rad_inf_sat} as $p$ is assumed to be separably good for $G$. This can be checked with Table \ref{Hypothèses_sur_la caractéristiqu_ pour_un_groupe_simple}.
\end{enumerate}
By making use of the identity $\Lie(H^0_{\red}) = \Lie(H_{\red})$, we obtain the desired equalities.
\end{proof}

\begin{proof}[Proof of Corollary \ref{cor_p_nilp_elemnt_p_inf}]
Note that the main point here was to obtain Corollary \ref{phi_sat_inf_et_rad_p} above. This being done the proof of Corollary \ref{cor_p_nilp_elemnt_p_inf} is now exactly the same as the proof of Corollary \ref{cor_p_elements_p_sup} (see Section \ref{section_cor_p_grand})
\end{proof}

The proof of Corollary \ref{cor_parab_associé_inf} is now immediate:

\begin{proof}[Proof of Corollary \ref{cor_parab_associé_inf}]
Let us consider the towers introduced in the statement and denote by $\mathfrak{u}_{\infty}$, respectively $\mathfrak{q}_{\infty}$, their limit objects. As by assumption the normaliser $N_{\mathfrak{g}}(\mathfrak{u}_{\infty})$ is $\phi$-infinitesimally saturated, the proof of Theorem \ref{morozov_p_inf} allows us to identify the Lie algebra $\mathfrak{u}_{\infty}$ with the $p$-radical of its normaliser $N_{\mathfrak{g}}(\mathfrak{u}_{\infty})$. According to Corollary \ref{morozov_p_inf} this normaliser is a parabolic Lie subalgebra of $\mathfrak{g}$. The same statement ensures that:
\begin{itemize}
 \item it derives from a subgroup $P(U_{\infty})$ obtained by applying Veisfeiler--Borel--Tits Theorem (see \cite[Corollaire 3.2]{TB}) to the group $U_{\infty}$ that integrates $\mathfrak{u}_{\infty}$,
 \item it is also the Lie algebra of the optimal parabolic subgroup of $\mathfrak{u}$, denoted by $P_G(\lambda_{\mathfrak{u_{\infty}}})$.
 \end{itemize}
\end{proof} 

\subsection{Proof of Corollary \ref{cor_Mor_pinf}}
\label{section_cor_Mor_pinf}

\begin{proof}[Proof of Corollary \ref{cor_Mor_pinf}]
Let $\phi: \mathcal{N}_{\red}(\mathfrak{g}) \rightarrow \mathcal{V}_{\red}(G)$ be a Springer isomorphism for $G$.
By maximality of $\mathfrak{q}$ only two possibilities can occur:
\begin{enumerate}
	\item either $N_{\mathfrak{g}}(\mathfrak{q}) = \mathfrak{g}$ so $\mathfrak{q} \vartriangleleft \mathfrak{g}$. In this case, we claim that $\mathfrak{q}$ is $p$-reductive. Indeed, by making use of the following exact sequence: 
	\begin{figure}[H]
\begin{center}
\[\begin{tikzpicture} 

 \matrix (m) [matrix of math nodes,row sep=2em,column sep=4.8em,minimum width=2em,  text height = 1.5ex, text depth = 0.25ex]
  {
    1 & (Z_G^0)_{\red} & G & G^{\Ss} & 1,\\
  };
  \path[-stealth]
  	(m-1-1) edge (m-1-2) 
  	(m-1-2) edge (m-1-3) 
    (m-1-3) edge (m-1-4) 
    (m-1-4) edge (m-1-5)
    ;  	
\end{tikzpicture}\]
\end{center}
\end{figure} 
where $(Z_G^0)_{\red}$ is the reduced part of the connected center of $G$, and $G^{\Ss}$ is the resulting semisimple group, we can reduce ourselves to the case $G = G^{\Ss}$ (as $\Lie((Z_G^0)_{\red})$ will not provide any $p$-nilpotent element). The group $G^{\Ss}$ then decomposes into an almost direct product of $n$ almost semisimple groups $G_i$, whose Lie algebras are denoted by $\mathfrak{g_i}$. According to \cite[Theorem 3.2]{Ho} there exist ideals $\mathfrak{q}_i \unlhd \mathfrak{g}^{\Ss}$ such that $ \mathfrak{q}_i = \mathfrak{q} \cap \mathfrak{g}_i$. Our assumptions on the characteristic imply that then $\mathfrak{q}_i$ is equal to $\mathfrak{g_i} $ for all $i$ except one, for which it is $0$ (see \cite[Corollary 2.7 a)]{Ho}, note that here it is crucial to assume $p>2$). We label the $\mathfrak{g}_i$'s in such a way that $\mathfrak{q}_1 = 0$. Then, according to \cite[Table 1]{Ho} together with \cite[Theorem 3.2]{Ho} one has $\sum_{i=2}^n \mathfrak{g}_i \subset \mathfrak{q} \subset \sum_{i=2}^n \mathfrak{g}_i + \mathfrak{z}_{\mathfrak{g}_1}$, hence $\mathfrak{q} = \sum_{i=2}^n \mathfrak{g}_i$ as $\mathfrak{z}_{\mathfrak{g}_1} = 0$ by virtue of \cite[XII Théorème 4.7 d) and Proposition 4.11]{SGA32} (see also \cite[Remarque C.1.4]{JEAthese}). So in this case the Lie algebra $\mathfrak{q}$ is $p$-reductive. 

	\item Otherwise, one has $N_{\mathfrak{g}}(\mathfrak{q}) = \mathfrak{q}$. Denote by $\mathfrak{u}$ the $p$-radical of $\mathfrak{q}$. There are several possibilities:
\begin{enumerate}
\item either $\mathfrak{u}:=\rad_p(\mathfrak{q})$ is trivial, then by definition $\mathfrak{q}$ is a $p$-reductive subalgebra of $\mathfrak{g}$;
\item or $\mathfrak{u}$ is not trivial. In this case note that necessarily $N_{\mathfrak{g}}(\mathfrak{u}) = \mathfrak{q}$ by maximality of $\mathfrak{q}$. Indeed, as $\mathfrak{u} = \rad_p(\mathfrak{q})$ one has that $\mathfrak{q} \subseteq N_{\mathfrak{g}}(\mathfrak{u})= \Lie(N_G(\mathfrak{u}))$ and this normaliser is a restricted $p$-nil $p$-subalgebra as it derives from an algebraic group. Therefore either $N_{\mathfrak{g}}(\mathfrak{u}) = \mathfrak{q}$ or $N_{\mathfrak{g}}(\mathfrak{u}) =\mathfrak{g}$. This last option cannot occur because $\mathfrak{g}$ is the Lie algebra of a reductive group, hence is $p$-reductive (thus has no $p$-nil $p$-ideal). So $N_{\mathfrak{g}}(\mathfrak{u}) = \mathfrak{q}$ and it remains to see why Theorem \ref{morozov_p_inf} (or one of its corollaries) apply.
\end{enumerate}
We start by showing that $\mathfrak{u}$ is integrable (so that one actually has $\mathfrak{j_u}:=\Lie(J_{\mathfrak{u}})=\mathfrak{u}$, where $J_{\mathfrak{u}}$ is as above). According to \cite[Lemma 3.4]{JEA1} and Lemma \ref{reminder_normaliser} (2.(i)) the following inclusions are satisfied $N_{G}(\mathfrak{u}) \subseteq N_G(J_{\mathfrak{u}}) \subseteq N_G(\mathfrak{j_u})$. At the Lie algebra level this leads to:
\[N_{\mathfrak{g}}(\mathfrak{u})=\mathfrak{q} \subseteq \Lie(N_G(J_{\mathfrak{u}})) \subseteq N_{\mathfrak{g}}(\mathfrak{j_u}).\]
By maximality of $\mathfrak{q}$, either $N_{\mathfrak{g}}(\mathfrak{j_u})= \mathfrak{q}$ or $N_{\mathfrak{g}}(\mathfrak{j_u})= \mathfrak{g}$. This second option cannot occur as this would imply that $\mathfrak{j_u}$ is a restricted $p$-nil $p$-ideal of $\mathfrak{g}$ (because it is the Lie algebra of $J_{\mathfrak{u}}$ which is unipotent according to \cite[Lemma 4.4]{JEA1}). So we have that $N_{\mathfrak{g}}(\mathfrak{j_u})= \mathfrak{q}$ and the situation is the following:
\[\mathfrak{u}:= \rad_p(\mathfrak{q}) \subseteq \mathfrak{j_u}\subseteq \rad_p(N_\mathfrak{g}(\mathfrak{j_u})) \subseteq \mathfrak{q}= N_{\mathfrak{g}}(\mathfrak{u})=N_{\mathfrak{g}}(\mathfrak{j_u})\]
where the first inclusion comes from \cite[Lemma 4.4]{JEA1}. Hence we actually have shown that $\mathfrak{u} = \mathfrak{j_u}$, so $J_{\mathfrak{u}}$ integrates $\mathfrak{u}$. 

Let $X$ be the reduced $k$-sub-scheme of $p$-nilpotent elements of $\mathfrak{q} = N_{\mathfrak{g}}(\mathfrak{u}) = W(N_{\mathfrak{g}})(\mathfrak{u})$. The $t$-power map, defined for any $p$-nilpotent $x \in \mathfrak{g}$ by the map,
\begin{alignat*}{3}
\phi_x : \: & \mathbb{G}_a \: & \rightarrow \: & G\\
\:&  t \: & \mapsto \: & \phi(tx).
\end{alignat*}
induces a morphism
\begin{alignat*}{4}
\psi_X : \: & X \: & \times \: & \mathbb{G}_a \: & \rightarrow \: & G\\
\:& (x,\:& \: & t) \: & \mapsto \: & \phi(tx).
\end{alignat*}
Let $J_X$ be the subgroup generated by the image of $\psi_X$. Note that $J_{X}$ is smooth (according to \cite[VIB, Proposition 7.1 (i)]{SGA31}) and connected (according to \cite[VIB, Corollaire 7.2.1]{SGA31}). Moreover $N_G(\mathfrak{q})\subseteq N_G(X) \subseteq N_G(J_X)$ where:
\begin{enumerate}
\item the first inclusion comes from the fact that $p$-nilpotency is preserved under $\Ad$, 
\item the second inclusion is deduced from the $\fppf$-formalism \cite[Lemma 3.4]{JEA1}. Note that this result is applicable here because its proof does not make use of the Lie algebra structure of the $p$-nil-subspace of $\mathfrak{g}$ one starts with.
\end{enumerate} 
Thus by derivation we finally get that $\mathfrak{q} \subseteq N_{\mathfrak{g}}(X) \subseteq N_{\mathfrak{g}}(\mathfrak{j}_X)$, because $\Lie(N_G(X)) = N_{\mathfrak{g}}(X)$ according to \ref{reminder_normaliser} ii). So by maximality of $\mathfrak{q}$:
\begin{enumerate}
\item either $N_{\mathfrak{g}}(\mathfrak{j}_X) =  \mathfrak{q}$, we then show that the normaliser $N:= N_{G}(\mathfrak{j}_X)$ is $\phi$-infinitesimally saturated. Indeed, let $X_{\mathfrak{n}}$ be the set of $p$-nilpotent elements of $\mathfrak{n}:= \Lie(N)$. Note that $\Lie(N) = N_{\mathfrak{g}}(\mathfrak{j}_X) =  \mathfrak{q}$ according to \ref{reminder_normaliser} iii) and by assumption. Thus $X = X_{\mathfrak{n}}$ and $J_{X_\mathfrak{n}}= J_X \subset N_G(J_X) \subset N_G(\mathfrak{j}_X)$, so $N_G(\mathfrak{j}_X)$ is $\phi$-infinitesimally saturated. Let $V := \Rad_U(N_G(\mathfrak{j}_X)^0_{\red})$ be the unipotent radical of $N_G(\mathfrak{j}_X)$. As $N_G(\mathfrak{j}_X)$ is $\phi$-infinitesimally saturated and $p\geq 3$, Lemma \ref{p_rad_inf_sat} allows us to conclude that $\mathfrak{v} \subseteq \mathfrak{u} = \rad_p(\mathfrak{q}) = \rad_p(N_{\mathfrak{g}}(\mathfrak{j}_X)) = \rad_p(\Lie(N_G(J_X))^{0}_{\red})$. The groups $J_{\mathfrak{u}}$ and $V$ being smooth and connected, and their Lie algebras being equal, they coincide. Now, note that:
\begin{enumerate}
	\item at the group level one has that $N_G(\mathfrak{j}_X) \subseteq  N_G(\mathfrak{v})$ (by \cite[Theorem 1.1]{JEA1} as $N_G(\mathfrak{j}_X)$ is $\phi$-infinitesimally saturated),
	\item at the Lie algebra level this leads to the inclusion $N_{\mathfrak{g}}(\mathfrak{j}_X) \subseteq  N_{\mathfrak{g}}(\mathfrak{v})$. Thus either $\mathfrak{q} = N_{\mathfrak{g}}(\mathfrak{j}_X) = N_{\mathfrak{g}}(\mathfrak{v})$ or $N_{\mathfrak{g}}(\mathfrak{v}) = \mathfrak{g}$, and only the first equality can occur, because $\mathfrak{q}$ is maximal and $\mathfrak{g}$ is $p$-reductive. Once again, the normaliser $N_G(\mathfrak{v})$ is $\phi$-infinitesimally saturated (this can be shown exactly in the same way as for $N_G(\mathfrak{j}_X)$). We just have shown that $\mathfrak{v}$ is the $p$-radical of its normaliser $N_{\mathfrak{g}}(\mathfrak{v})$ and that the latter is the Lie algebra of a $\phi$-infinitesimally saturated group. Corollary \ref{morozov_p_inf} then holds true and allows us to conclude that $\mathfrak{q}$ is parabolic.
\end{enumerate}

\item or $N_{\mathfrak{g}}(\mathfrak{j}_X) = \mathfrak{g}$. Then, consider the following short exact sequence:
\begin{figure}[H]
\begin{center}
\[\begin{tikzpicture} 

 \matrix (m) [matrix of math nodes,row sep=2em,column sep=4.8em,minimum width=2em,  text height = 1.5ex, text depth = 0.25ex]
  {
    1 & (Z_G^0)_{\red} & G & G^{\Ss} & 1,\\
  };
  \path[-stealth]
  	(m-1-1) edge (m-1-2) 
  	    (m-1-2) edge (m-1-3) 
    (m-1-3) edge (m-1-4) 
    (m-1-4) edge (m-1-5)
    ;  	
\end{tikzpicture}\]
\end{center}
\end{figure} 
where $(Z^0_G)_{\red}$ is the reduced part of the connected component of the center of $G$. This leads, after derivation, to the following exact sequence of Lie algebras:
\begin{figure}[H]
\begin{center}
\[\begin{tikzpicture} 

 \matrix (m) [matrix of math nodes,row sep=2em,column sep=4.8em,minimum width=2em,  text height = 1.5ex, text depth = 0.25ex]
  {
    0 & \Lie((Z_G^0)_{\red}) & \mathfrak{g} & \mathfrak{g}^{\Ss} & 0.\\
  };
  \path[-stealth]
  	(m-1-1) edge (m-1-2) 
  	  	    (m-1-2) edge (m-1-3) 
    (m-1-3) edge node[above] {$\pi$} (m-1-4) 
    (m-1-4) edge (m-1-5)
    ;  	
\end{tikzpicture}\]
\end{center}
\end{figure} 
The group $G^{\Ss}$ is an almost direct product of almost simple groups $G_i$, let $\mathfrak{g}_i$ the corresponding Lie algebra. The image of $\mathfrak{j}_X$ is an ideal of $\mathfrak{g}^{\Ss}$ because $\pi$ is surjective, thus, according to \cite[Theorem 3.2]{Ho}, there exist ideals $(\mathfrak{j}_X)_i\unlhd \mathfrak{g}_i$ and thus $N_{\mathfrak{g}}(J_X) = \mathfrak{g}$. For what follows we can hence assume that $G$ is one of the $G_i$'s. So $\mathfrak{j}_X$ is an ideal of the Lie algebra of $G$ which is assumed to be almost-simple, hence $\mathfrak{j}_X = \mathfrak{g}$, according to \cite[corollary 2.7 a)]{Ho} and because of the assumption we made on the characteristic. Therefore one also has $J_X= G$ because both groups are smooth and connected (\cite[II, \S5, 5.5]{DG}). Once again there are two possibilities and we will show that both of them are absurd:
\begin{itemize}
	\item either $N_{\mathfrak{g}}(X) = \mathfrak{q}$, in this case let $\lambda_X$ be a destabilizing one-parameter subgroup associated to $X$ via the Hilbert--Mumford--Kempf--Rousseau's method. As all Lie algebras considered here are finitely generated it is enough to destabilise a well-chosen vector with a finite number of components. By construction (see section \ref{Kempf-Rousseau} above) one has $X \subset \radu(\mathfrak{p_g}(\lambda_X))$. Also, note that the inclusion $J_X \subset P_G(\lambda_X)$ is always satisfied, because of the $\fppf$-formalism. Indeed, let $R$ be a $k$-algebra, and consider $h \in J_{X}(R)$. By definition of $J_{X}$ there exists an $\fppf$-covering $S \rightarrow R$ such that $h_S = \psi_{X}(x_1,s_1) \cdots \psi_{X}(x_n,s_n)$ for $x_i \in X_R \otimes_R S $ and $s_i \in S$. But then one has 
	\[\lim_{t \to 0}(\lambda_X(t) \cdot h)_S = \prod_{i=1}^n \lim_{t \to 0}\lambda_X(t) \cdot \psi_{X}(x_i,s_i) = \prod_{i=1}^n \psi_{X}\left(\lim_{t \to 0}\lambda_X(t) \cdot x_i,s_i\right),\] 
\noindent which exists as any limit $\lim_{t \to 0}\lambda_X(t) \cdot x_i$ exists. Thus $P_G(\lambda_X) = G$ (as $J_X = G$) and $\radu(\mathfrak{p_g}(\lambda_X))$ is trivial, so are $X$ and $\mathfrak{u}$, hence a contradiction.
	\item Otherwise, one has $N_{\mathfrak{g}}(X) = \mathfrak{g}$, in this case let $\mathfrak{h}_X$ be the restricted $p$-Lie algebra generated by $X$. Then $N_{\mathfrak{g}}(X) \subset N_{\mathfrak{g}}(\mathfrak{h}_X)$, thus $N_{\mathfrak{g}}(\mathfrak{h}_X) = \mathfrak{g}$, hence $\mathfrak{h}_X= \mathfrak{g}\supsetneq \mathfrak{q}$, which is also absurd.
\end{itemize}
\end{enumerate}  

\end{enumerate}
\end{proof}

\section{An application: characterisation of the canonical parabolic subgroup of a reductive group}
	\label{parab_canonique}
	
	\subsection{Proof of Corollary \ref{cor_killing}}
\label{section_cor_Killing}
	
	Corollary \ref{cor_killing} of Theorem \ref{Morozov_p_sup} below is crucial to derive the application presented in Section \ref{parab_canonique} below. We assume here that $\mathfrak{g}$ is endowed with a non-degenerate $G$-equivariant symmetric bilinear form, denoted by $\kappa$. Under this assumption, and according to \cite[I, \S8, Corollary]{SEL}, the Lie algebra $\mathfrak{g}$ is semisimple.

\begin{corollary}[of Corollary \ref{cor_p_elements_p_sup}]
Let $k$ be an algebraically closed field and $G$ be a reductive $k$-group. Assume moreover that $k$ is of characteristic $p>\h(G)$ and that the Killing form $\kappa$ is non-degenerate over $\mathfrak{g}$. Let $\mathfrak{p} \subseteq \mathfrak{g}$ be a Lie subalgebra such that $\mathfrak{p}^{\perp}$ is a nilpotent subalgebra of $\mathfrak{q}$ made of $p$-nilpotent elements of $\mathfrak{g}$. Then:
	\begin{enumerate}
		\item the subalgebra $\mathfrak{p}$ is a parabolic subalgebra of $\mathfrak{g}$,
		\item it is the parabolic subalgebra obtained in Theorem \ref{Morozov_p_sup}. Namely one has:
		\[\mathfrak{p} = \Lie(P_G(\exp(\mathfrak{p}^{\perp}))) = \mathfrak{p_g(\lambda_{p^{\perp}})}.\]
	\end{enumerate} 
	\label{cor_killing}
\end{corollary}

\begin{proof}
Let us first show that $\mathfrak{p} \subseteq N_{\mathfrak{g}}(\mathfrak{p}^{\perp})$: let $x \in \mathfrak{p}^{\perp}$ and $p_1 \in \mathfrak{p}$. For any $p_2 \in \mathfrak{p}$, the following vanishing condition $\kappa([x,p_1],p_2) = \kappa(x,[p_1,p_2]) = 0$ is satisfied. In other words $[x,p_1] \in \mathfrak{p}^{\perp}$ and $\mathfrak{p}$ is therefore a subalgebra of $N_{\mathfrak{g}}(\mathfrak{p}^{\perp})$. 

Now, remark that $\mathfrak{p}^{\perp}$ is contained in $\nil(\mathfrak{p})$ because $\mathfrak{p}^{\perp}$ is a nilpotent ideal of $\mathfrak{p}$. On the other hand the nilradical of $\mathfrak{p}$ is orthogonal to $\mathfrak{p}$ according to \cite[\S 4, n\degree 3, Proposition 4d]{BOU1}. Hence one eventually has $\mathfrak{p}^{\perp} = \nil(\mathfrak{p})$. As by assumption $\mathfrak{p}^{\perp}$ is made of $p$-nilpotent elements, the nilradical $\nil(\mathfrak{p})$ is a $p$-nil $p$-ideal of $\mathfrak{p}$, therefore one has $\nil(\mathfrak{p}) = \rad_p(\mathfrak{p})=:\mathfrak{u}$

Moreover, by definition $\mathfrak{p} \subseteq N_{\mathfrak{g}}(\mathfrak{u})$, so that $N_{\mathfrak{g}}(\mathfrak{u})^{\perp} \subseteq \mathfrak{p}^{\perp} = \mathfrak{u}$. Applying  \cite[\S 4, n\degree 3, Proposition 4d]{BOU1}, one also gets the inclusions $\nil(N_{\mathfrak{g}}(\mathfrak{u})) \subseteq (N_{\mathfrak{g}}(\mathfrak{u}))^{\perp}$. To summarise the situation is the following:
\[\mathfrak{u} \subseteq \rad_p(N_{\mathfrak{g}}(\mathfrak{u}) \subseteq \nil(N_{\mathfrak{g}}(\mathfrak{u})) \subseteq (N_{\mathfrak{g}}(\mathfrak{u}))^{\perp} \subseteq \mathfrak{p}^{\perp} = \mathfrak{u}.\]
Therefore $\mathfrak{u}$ is the $p$-radical of its normaliser, and so $N_G(\mathfrak{u})$ is parabolic according to Theorem \ref{Morozov_p_sup}. The second point directly follows from Theorem \ref{Morozov_p_sup}.

\end{proof}

	\subsection{Canonical parabolic subgroup}
	
In what follows we use the language of \cite{SGA33}. In particular the theory we refer to is the one of reductive groups over an arbitrary basis. Let us remind the reader of some conventions: a group scheme $G$ over a basis $S$ is reductive if it is affine and smooth over $S$, with connected and reductive fibres (see \cite[XIX, Définition 2.7]{SGA33}). A definition of Borel subgroups in this framework can be found in \cite[XXII Exemples 5.2.3.a)]{SGA33}, see also \cite[XXVI Définition 1.1]{SGA33} for a definition of parabolic subgroups. A reference for the study of the scheme of automorphisms of a group scheme is \cite[XXIV]{SGA33}. 

Let $k$ be a field and $C$ be a smooth geometrically connected projective $k$-curve. In \cite{BEH}, K. A. Behrend defines the notion of the canonical parabolic subgroup of a reductive $C$-group $G$ which is the twisted form of a constant $C$-group $G_0$ (see \cite[Definition 5.6]{BEH}). Let $E$ be the associated $G_0$-torsor over $X$ so that $G= ^{E}G_0$. The author shows that such a group has a unique canonical parabolic subgroup, denoted by $P_{G}^{\can}$: it is the maximal element (for the inclusion) of the set of parabolic subgroups of $G$ whose degree is the degree of instability of $G$ (see Theorem 7.3 and Definition 4.4 ibid.). As a reminder (see Definition 4.1 and Note 4.2 ibid.), the degree of a smooth connected $C$-group is the degree of its Lie algebra seen as a vector bundle over $C$. The degree of a reductive group is thus always equal to $0$. This defines a canonical reduction of the structure group of $E$ that generalises that of $G$ as defined by M. Atiyah and R. Bott in \cite[\S 10]{AB} when $k$ is of characteristic $0$ (as explained and showed in \cite{BH}.

Let us set $U_G^{\can}:= \Rad_U(P_G^{\can})$. As mentioned in the introduction, when $p>2 \dim(G)-2$ and under the above assumptions on $G$, Theorem \ref{Morozov_p_sup} allows to provide a new proof that the definition of M. Atiyah and R. Bott still make sense in characteristic $p>0$. In this framework, the parabolic subgroups of $G=\prescript{E}{}{G_0}$ obtained respectively by K. A. Behrend in \cite[Theorem 7.3]{BEH} (and that are always defined), and by M. Atiyah and R. Bott in \cite[\S 10]{AB} (that is defined only for $p = 0$ or $p>2 \dim(G)-2$), coincide. This is the point of this section.

Let $V$ be a vector bundle over a curve $C$. As a reminder:
\begin{itemize}
\item the slope of $V$, denoted by $\mu(V)$, is the quotient $\deg(V)/\rank(V)$ where $\deg(V)$ and $\rank(V)$ are respectively the degree and the rank of the vector bundle $V$. Let us in particular underline that $V$ is of positive (respectively negative) slope if and only if $V$ has positive (respectively negative) degree;
\item any vector bundle over a curve admits a Harder--Narasimhan filtration (see \cite[1.3.9]{HN}), that is a filtration by vector subbundles such that the successive quotient are semistable with strictly decreasing slope. 
\end{itemize}  
In what follows $\mu_{\min}(V)$ (respectively $\mu_{\max}(V)$) is the minimum (respectively the maximum) of slopes of the successive quotients that occur in this filtration.

In order to avoid heavy notations, we denote:
\begin{itemize}
\item by $P:= P_G^{\can}$ the canonical parabolic subgroup of $G$, 
\item by $I$ the type of $P$,
\item and by $P_0$ the parabolic subgroup of $G_0$ from which $P$ is a twisted form. 
\end{itemize}
The couple $(\prescript{E}{}{G_0}, \prescript{E_{P_0}}{}{P_0})$ naturally identifies with $(G,P)$, where $E$ is the $\Aut(G_0, P_0)_{k}$-torsor $\Isom((G_0,P_0),(G,P))$, and $E_{P_0}$ is its restriction to $P_0$ (see \cite[Exposé XXIV]{SGA33}, in particular Corollary 2.2). This restriction of structure is said to be canonical (see \cite[\S 8]{BEH} as well as the preamble and Lemma 4 of \cite{HEI}). The associated $P_0/\Rad_U(P_0)$-torsor is thus semi-stable (which means that the reductive group $P/\Rad_U(P)$ it defines is semistable). As a reminder a reductive $k$-group $H$ is semistable if any parabolic subgroup $Q \subseteq H$ is of negative or $0$ degree, see \cite[Definition 4.4]{BEH}).

Let us go back to our notations. The Lie algebra $\mathfrak{g}:= \prescript{E}{}{\mathfrak{g}_0}$ is a vector bundle over $C$. It thus admits a Harder--Narasimhan filtration:
 \[0 \subsetneqq E_{-r} \subsetneqq E_{-r+1} \subsetneqq \hdots \subsetneqq E_{-1} \subsetneqq E_{0} \subsetneqq E_{1} \subsetneqq \hdots \subsetneqq E_{l} = \mathfrak{g}, \]
\noindent where the indices are chosen in such a way that:
\begin{itemize}
\item the slopes $\mu(E_i/E_{i-1})$ are non positive for $0 \leq i \leq l$,
\item the slopes $\mu(E_{-i}/E_{-i-1})$ are positive for $1 \leq i \leq r-1$.
\end{itemize}  

The following lemma provides a first comparison between $\mathfrak{p_g^{\can}}$ and $E_0$. Given a constant $C$-group  $H_0$, its proof requires to compare the semi-stability of a $H_0$-torsor and that of the vector bundle it induces. This is the point of \cite[Theorem 3.1]{IMP} when $H_0$ is semisimple. The general case of a reductive group when the connected component of the center of $H_0$ acts by homothetie is detailed in \cite[Proposition 6.9]{BH}.

\begin{lemma}
When $p>2\h(G)-2$ the Lie algebra of the canonical parabolic subgroup as defined by K. A. Behrend satisfies the inclusion $\mathfrak{p_g^{\can}} \subseteq E_0$.
\label{p_dans_E0}
\end{lemma}

\begin{proof}
As previously mentioned, the $P_0/\Rad_{U}(P_0)$-torsor induced by $P_0$ is semistable because $P$ is the canonical parabolic subgroup of $G = \prescript{E}{}{G_0}$. Let us remark that one has $\h(G_0) = \h(G)$ and $\h(P_0) = \h(P)$. This comes from the definition of the Coxeter number (see subsections \ref{hypothèses_car} and \ref{ssalg_parab}). Let us moreover note that \cite[Proposition 6.9]{BH} holds true, because when $p>2\h(G)-2$ the Adjoint representation $\Ad :P_0/\Rad_{U}(P_0) \rightarrow  \GL(\mathfrak{p_0}/\radu(P_0))$ is of low height (see Remark \ref{petite_hauteur}). Therefore the induced vector bundle $\mathfrak{p_g^{\can}}/\radu(P_G^{\can})$ is semistable. It is thus of degree $0$ (because the quotient $P/\Rad_U(P)$ is reductive), so $\mu_{\min}(\mathfrak{p_g^{\can}}) = 0$. As by construction $\mu_{\max}(\mathfrak{g}/E_0) = \mu(E_1/E_0)<\mu(E_0/E_{-1})\leq 0$, the only possible morphism between $\mathfrak{p_g^{\can}}$ and $\mathfrak{g}/E_0$ is the trivial morphism. In other words the inclusion $\mathfrak{p_g^{\can}} \subseteq E_0$ holds true.  
\end{proof}

The main point of this section is to determine some conditions under which the equality $E_0 = \mathfrak{p_g^{\can}}$ is satisfied. A first criterion is obtained by making additional assumptions on the slopes of the successive quotients of the Harder--Narasimhan filtrations of both $\radu(P_G^{\can})$ and $\mathfrak{g/p_g^{\can}}$:

\begin{proposition}
When $p>2\h(G)-2$, if:
\begin{enumerate}
\item the Harder--Narasimhan filtration of $\mathfrak{u_g^{\can}}$ admits successive quotients whose slopes are strictly positive,
\item the Harder--Narasimhan filtration of $\mathfrak{g/p_g^{\can}}$ admits successive quotients of strictly negative slopes,
\end{enumerate} 
then $E_0$ is the Lie algebra of $P_G^{\can}$ while $E_{-1}$ is that of $U_G^{\can}$.
\label{egal_si_pentes}
\end{proposition}
\begin{proof}
In order to avoid heavy notations, we set $U:=\Rad_U(P_G^{\can})$ and $\mathfrak{u}:= \mathfrak{u_g^{\can}}$.

By assumption one has $\mu_{\min}(\mathfrak{u}) >0$ and the equality $\mu_{\max}(\mathfrak{g}/E_{-1})=\mu(E_0/E_{-1}) \leq 0$ holds true. Thus the only possible morphism between $\mathfrak{u}$ and $\mathfrak{g}/E_{-1}$ is the trivial one. Therefore we have shown the inclusion $\mathfrak{u} \subseteq E_{-1}$. 

By assumption $\mathfrak{g/p}$ admits a filtration by semistable quotients of strictly negative slope. As $P$ is the canonical parabolic subgroup of $G$, for the same reasons as the one involved in proof of Lemma \ref{p_dans_E0}, the quotient $\mathfrak{p}/\mathfrak{u}$ is semistable and of slope $0$, hence the equality $\mu_{\max}(\mathfrak{g/u}) =  \mu(\mathfrak{p/u}) = 0$. As by construction $\mu_{\min}(E_{-1}) = \mu(E_{-1}/E_{-2})>\mu(E_{-1}/E_0) >0$, the only possible morphism between $E_{-1}$ and $\mathfrak{g/u}$ is trivial, whence the equality $\mathfrak{u} = E_{-1}$. To summarise, the situation is as follows $E_{-2}\subsetneq E_{-1}= \mathfrak{u} \subsetneq \mathfrak{p} \subset E_0 \subsetneq E_1$ and it remains to show that $\mathfrak{p} = E_0$. 

By assumption the slopes of the Harder-Narasimhan filtration of $\mathfrak{g/p}$ are strictly negative, so $\mu_{\max}(\mathfrak{g/p})<0$. As one has $\mathfrak{u} = E_{-1}$ the quotient $\mathfrak{p/u}$ is a subbundle of $E_0/E_{-1}$, which is semistable. In other words the inequality $\mu(\mathfrak{p/u}) < \mu(E_0/E_{-1})$ is satisfied. The quotient $\mathfrak{p/u}$ is of slope $0$ because it is of degree $0$. As the slope $\mu(E_0/E_{-1}$ is not positive by construction, the following vanishing condition is actually satisfied $\mu_{\min}(E_0) = \mu(E_0/E_{-1}) = 0$. The only possible morphism between $E_0$ and $\mathfrak{g/p}$ is therefore the trivial morphism, so we have shown that $E_0 = \mathfrak{p}$.

\end{proof}
                                                                                                                                                                                                                                                                                                                                                                                                                                                                                                                                                     When $\mathfrak{g}$ is endowed with a non degenerate symmetric $G$-equivariant bilinar form $\kappa$ the Harder--Narasimhan filtration of $\mathfrak{g}$ has a very specific shape, namely:
                                                                                                                                                                                                                                                                                                                                                                                                                                                                                                                                                     
\begin{lemma}                                                                                                                                                                                                                                                                                                                                                                                                                                                                                                                                                     Under the above assumptions the Harder--Narasimhan filtration of $\mathfrak{g}$ is such that $l = r+1$ and $E_{-i} \cong E_{i-1}^{\perp}$ for all $1 \leq i \leq r$.
\label{HN_dual}
\end{lemma}
                                                                                                                                                                                                                                                                                                                                                                                                                                                                                                                                                 
\begin{proof}
If $\kappa$ is non degenerate then $\mathfrak{g} \cong \mathfrak{g}^{\vee}$ and the Harder--Narasimhan filtration of $\mathfrak{g}^{\vee}$ is given by:
\[0 \subsetneqq (E^{\vee})_{-r} \subsetneqq \hdots \subsetneqq (E^{\vee})_{-1} \subsetneqq (E^{\vee})_{0} \subsetneqq (E^{\vee})_{1} \subsetneqq \hdots \subsetneqq (E^{\vee})_{l} = \mathfrak{g}^{\vee},\]
\noindent where $(E^{\vee})_{-i} = (\mathfrak{g}/ E_{i+l-r-1})^{\vee}$ if $0 \leq i \leq r$ and $(E^{\vee})_{i} = (\mathfrak{g}/ E_{-i+l-r+1})^{\vee}$ if $0 \leq i \leq l$.
Indeed: 
\begin{itemize}
	\item the bundle $(E^{\vee})_{-r} = (\mathfrak{g}/ E_{l-1})^{\vee}$ is semistable, because $\mathfrak{g}/E_{l-1}$ is (this is the first quotient of the Harder--Narasimhan filtration of $\mathfrak{g}$), 
	\item the quotient $(E^{\vee})_{-i}/(E^{\vee})_{-i-1}$ is isomorphic to $(E_{i+l-r} / E_{i+l-r-1})^{\vee}$ (by application of Snake Lemma), therefore $(E^{\vee})_{-i}/(E^{\vee})_{-i-1}$ is semistable for all $0 \leq i \leq r$. The same reasoning holds for any $0 \leq i \leq l$ as one can show on the same way that $(E^{\vee})_{i}/(E^{\vee})_{i-1}$ is isomorphic to $(E_{-i+l-r+1} / E_{i+l-r})^{\vee}$, 
	\item the preceding isomorphisms ensure the desired strictness of the inequalities for slopes.
\end{itemize}
The equalities $\mathfrak{g}\cong \mathfrak{g}^{\vee}$ and $E_{-i} \cong (E^{\vee})_{-i} = (E/E_{i-1})^{\vee} \cong E_{i-1}^{\perp}$ then follow from the uniqueness of the Harder-Narasimhan filtration. 
 
\end{proof} 

\begin{remark}
Lemma \ref{HN_dual} allows us to remark that under the conditions of Proposition \ref{E0_canonique}, the slopes of the Harder--Narasimhan filtration of $\mathfrak{g/p_g^{\can}} = \mathfrak{g}/E_0$ are strictly negative and those of the Harder--Narasimhan de $\mathfrak{u_g^{\can}}$ are strictly positive. In other words, the assumptions required at Proposition \ref{E0_canonique} imply those of Proposition \ref{egal_si_pentes}.
\end{remark}

As mentioned in the introduction, when $k$ is of characteristic $0$, the equality $E_0 = P$ is proven in \cite[\S 10]{AB}. Given a constant reductive group $G_0$ over a $k$-curve $C$, Atiyah-Bott definition was actually the definition of the canonical parabolic subgroup of a $G_0$ torsor $G$ before K. A. Behrend's work (so such an object was at first only defined in characteristic $0$). In characteristic $p>0$, the equality $E_0 = P$ is given:
\begin{itemize} 
\item by V. B. Mehta in \cite[Theorem 2.6]{Meh} when $p> 2 \dim(G)$ (see also the paper of V. B. Mehta et S. Subramanian \cite[Proposition 3.4]{MS} when $p> \max(2\dim G, 4(\h(G)-1))$), 
\item and by A. Langer in \cite[Proposition 3.3]{LAN} when $E$ admits a strong Harder--Narasimhan filtration (see \cite[Definition 3.1]{LAN}). This last condition is in particular satisfied when $p>3$ and the minimal slope of the tangent bundle of $X$ is strictly positive (see \cite[Corollary 6.4]{LAN} and \cite[Theorem 4.1]{MS}). 
\end{itemize}
Theorem \ref{Morozov_p_sup} allows us to show the equality of Proposition \ref{egal_si_pentes} and provides a new proof for it when $p>2\dim(G)-2$ by mimicking the one already known in characteristic $0$. 

\begin{remark}
Some of the results stated in the article of V. Balaji, P. Deligne and A. J. Parameswaran \cite{BDP} already allow to show Proposition \ref{egal_si_pentes} when $G$ admits an almost faithful representation of low height and $p>\h(G)$. This amounts to require $p> 2\h(G) -2$ if $G$ is simple and allows to show \cite[Proposition 3.4]{MS} without the additional assumption $p> 4(\h(G)-1)$ (this assumption allowed the authors to make use of a result of J-P. Serre coming from representation theory to reduce the proof to the case $G :=\GL_n$).
\end{remark}

	\subsection{Proof of Proposition \ref{E0_canonique}}

Let us remind the reader of the notations of this section: in what follows and unless otherwise stated $k$ is a field and $C$ is a projective, smooth and geometrically connected $k$-curve. We denote by $K$ its function field whose algebraic closure is denoted by $\bar{K}$. Let $G$ be a reductive $C$-group which is the twisted form of a constant $C$-group $G_0$ and assume that $\mathfrak{g}$ is endowed with a non degenerate Killing form. Let us moreover denote:
\begin{itemize}
\item by $P$ the canonical parabolic subgroup of $G$ (as defined by K. A. Behrend in \cite[Theorem 7.3]{BEH}), 
\item by $P_0$ the parabolic subgroup of $G$ from which $P$ is a twisted form, 
\item by $E$ the $\Aut(G_0,P_0)_C$-torsor $\Isom((G_0,P_0),(G,P))$ and by $E_{P_0}$ its restriction to $P_0$.
\end{itemize} 
The Lie algebra $\mathfrak{g} = \prescript{E}{}{\mathfrak{g}_0}$ seen as a vector bundle over $C$ admits a Harder--Narasimhan filtration:
 \[0 \subsetneqq E_{-r} \subsetneqq E_{-r+1} \subsetneqq \hdots \subsetneqq E_{-1} \subsetneqq E_{0} \subsetneqq E_{1} \subsetneqq \hdots \subsetneqq E_{l} = \mathfrak{g}, \]
\noindent where the indices are chosen such that $\mu(E_i/E_{i-1}) \leq 0$ for $0 \leq i \leq l$ and $\mu(E_{-i}/E_{-i-1}) > 0$ for $1 \leq i \leq r-1$. According to Lemma \ref{HN_dual}, any factor of the filtration satisfies $E_{-i} \cong E_{i-1}^{\perp}$ because $\mathfrak{g}$ is endowed with a non degenerate Killing form.

The proof of Proposition \ref{E0_canonique} requires to show that:
\begin{enumerate}
\item the factor $E_0$ is a Lie subalgebra of $\mathfrak{g}$ and the factor $E_{-1}$ is a nilpotent ideal of $E_0$. Let us note that a priori one would have to show that $E_{-1}$ is actually $p$-nil, but the assumption of the existence of a non degenerate bilinear form on $\mathfrak{g}$ guarantees the equivalence between $p$-nilpotency and $\ad$-nilpotency (see Remark \ref{semisimple_rad_p_nil});
\item the factor $E_0$ derives from a parabolic subgroup of $G$. 
\end{enumerate}
To prove the first point we make use of results that ensure a compatibility between the tensor product and the semi-stability of its factors. We then need Corollary \ref{cor_killing} to prove the second point, as this result allows us to show, on the generic geometric fibre, that $(E_{0})_{\bar{K}}$ is a parabolic subalgebra.

	\subsubsection{Semi-stability of tensor products}
	\label{par_semi_stab}
	The point of this paragraph is to show:
	\begin{enumerate}
	\item that $E_0$ is a subalgebra of $\mathfrak{g}$,
	\item that $E_{-1}$ is a nilpotent ideal of $E_0$,
	\item the third point of Proposition \ref{E0_canonique},
	\item elements of $E_{-1}$ are $p$-nil in $\mathfrak{g}$.
	\end{enumerate}
We start by showing the following lemma:

\begin{lemma}
When $p > 2 \dim G - 2$, for any couple of indices $-r \leq i, \ j \leq l$ in the above Harder--Narasimhan filtration of $\mathfrak{g}$ the slope identity $\mu_{\min}(E_i \otimes E_j) = \mu_{\min}(E_i) + \mu_{\min}(E_j)$ is satisfied.
\label{semistab_prodtens}
\end{lemma}

\begin{proof}
This is a direct consequence of the following results:
\begin{enumerate}
	\item if $k$ is of characteristic $0$, the tensor product of semistable $C$-vector bundles $V_1$ and $V_2$ is still semistable and of slope $\mu(V_1) + \mu(V_2)$ (see \cite[Lemma 10.1]{AB} see also \cite[Theorem 6]{Mac}), the result extends to $k$ non algebraically closed), 
  \item if $k$ is a field of characteristic $p>0$, let $V_1$ and $V_2$ be two semistable $C$-vector bundles and assume that $p+2 > \rank (V_1) + \rank (V_2)$. Then the vector bundle $V_1 \otimes V_2$ is still semistable of slope $\mu(V_1) + \mu (V_2)$ (see \cite[Remark 3.4]{IMP}, and \cite[Theorem 8]{Mac}). 
 \end{enumerate}
\end{proof}

This lemma allows us to show that $E_0$ is a subalgebra of $\mathfrak{g}$. Let us consider the Lie bracket:
\[\phi : E_0 \otimes E_0 \rightarrow \mathfrak{g} \rightarrow \mathfrak{g}/E_0,\]
\noindent where we made use of the isomorphism $(\mathfrak{g}^{\vee})^{\vee} \cong \mathfrak{g}$ because $\kappa$ is non degenerate.

As $E_{-1}\cong E_0^{\perp}=(\mathfrak{g}/ E_{0})^{\vee}$ (according to Lemma \ref{HN_dual}) the degree of $E_{-1}$ satisfies \[\deg(E_{-1}) = -(\deg(\mathfrak{g})-\deg(E_0)) = \deg(E_0),\] where the last equality is induced by the fact that $\deg(\mathfrak{g})=0$. So one actually has obtained the vanishing condition $\mu_{\min}(E_0)=\mu(E_0/E_{-1}) = 0$. 

Moreover, according to Lemma \ref{semistab_prodtens}, 
\[\mu_{\min}(E_0 \otimes E_0) = \mu_{\min}(E_0) + \mu_{\min}(E_0) = 0,\] 
and $\mu_{\max}(\mathfrak{g}/E_0) = \mu(E_1/E_0) < \mu(E_0/E_{-1})=0$. This implies that the morphism $\phi$ is necessarily trivial: the vector bundle $E_0$  is thus a Lie algebra bundle. 

This kind of reasoning also allow to show that $E_{-1}$ is a nilpotent ideal of $\mathfrak{g}$. For any $0 \leq j \leq r$ let us consider the bracket:
\[\phi : E_{-1} \otimes E_{-j} \rightarrow \mathfrak{g}/E_{-j-1},\]
with the convention that $E_{-r-1} = 0$. As by assumption $p>2\dim G-2$ we have that \begin{alignat*}{1}
 \mu_{\min}(E_{-1} \otimes E_{-j}) = \: & \mu_{\min}(E_{-1}) + \mu_{\min}(E_{-j}) \\
 \: = & \mu(E_{-1}/E_{-2}) + \mu(E_{-j}/E_{-j-1}) \\
 \: > & \mu(E_{-j}/E_{-j-1}) = \mu_{\max}(E_0/E_{-j-1}),
 \end{alignat*} 
because $\mu(E_{-1}/E_{-2}) > \mu(E_{0}/E_{-1}) =0$. Indeed, the morphism $\phi$ is necessarily trivial, thus $[E_{-1}, E_{-j}] \subset E_{-j-1}$ for any $0\leq j$, and $E_{-1}$ is a nilpotent ideal of $E_0$.

In order to show the third point of Proposition \ref{E0_canonique} we once again consider the composition of the Lie bracket together with the canonical projection for any $0\leq i \leq r$:
\[E_{-i} \otimes E_{0} \rightarrow E_0 \rightarrow E_0/E_{-i}.\]
\noindent Note that
 \begin{alignat*}{1}
 \mu_{\min}(E_{0} \otimes E_{-i}) = \: & \mu_{\min}(E_{0}) + \mu_{\min}(E_{-i}) \\
 \: = & \mu(E_{0}/E_{-1}) + \mu(E_{-i}/E_{-i-1}) \\
 \: > & \mu_{\max}(E_0/E_{-i}) = \mu(E_{-i+1}/E_{-i}),
 \end{alignat*} 
\noindent as $\mu(E_{0}/E_{-1}) =0$ and $\mu(E_{-i}/E_{-i-1}) > \mu(E_{-i+1}/E_{-i})$ by assumption. The morphism $\phi$ is therefore necessarily trivial and $E_{-i}$ is an ideal of $E_0$, nilpotent because so is $E_{-1}$.

Finally, we show the last point. As we already know that $E_{-1}$ is a $p$-nil $p$-ideal of $E_0$, to show that any element of $E_{-1}$ is $p$-nil one only needs to show that the Lie bracket with any element of $\mathfrak{g} \setminus E_{0}$ is $p$-nilpotent. This can be done by considering, for any $0\leq i \leq r$, the composition of the Lie bracket together with the canonical projection:
\[E_{i} \otimes E_{-1} \rightarrow \mathfrak{g}/E_{i-1}.\]
\noindent One has
 \begin{alignat*}{1}
 \mu_{\min}(E_{i} \otimes E_{-1}) = \: & \mu_{\min}(E_{i}) + \mu_{\min}(E_{-1}) \\
 \: = & \mu(E_{i}/E_{i-1}) + \mu(E_{-1}/E_{-2}) \\
 \: > & \mu_{\max}(\mathfrak{g}/E_{i-1}) = \mu(E_{i}/E_{i-1}),
 \end{alignat*} 
because $ \mu(E_{-1}/E_{-2})>0$ and $\mu(E_{i}/E_{i-1})<0$ by assumption. So for any $0\leq i \leq r$ the Lie bracket with an element of $E_{-1}$ sends any $E_{i}$ to $E_{i-1}$, so that taking $i$-times the Lie bracket with $E_{-1}$ will send any element of $E_{i}$ to $E_{0}$, where the Lie bracket with $E_{-1}$ turns out to be $p$-nilpotent. Therefore $E_{-1}$ is made of $p$-nilpotent elements of $\mathfrak{g}$.

\begin{remark}
The key point of the preceding proof is the equality  
\[\mu_{\min}(E_0 \otimes E_0) = \mu_{\min}(E_0) + \mu_{\min}(E_0) = 0,\]
which is satisfied when $p> 2 \rank E_0 -2$. This condition, and more generally the condition of compatibility between the stability of vector bundles and their tensor products is ensured in the articles of V. B. Mehta and S. Subramanian \cite{MS}, and V. B. Mehta \cite{Meh} by the assumption $p>2\dim(G)$. 
\end{remark}

	\subsubsection{Consequences of corollary \ref{cor_killing}}
\label{parab_E0}

Corollary \ref{cor_killing} allows us to show that $E_0$ is a parabolic subalgebra of $\mathfrak{g}$, this is the second part of the first point of Proposition \ref{E0_canonique}. It moreover provides a characterisation of such on the geometric fibres in terms of the instability parabolic subgroup of $E_{-1}$. This is the point of the second point of the aforementioned Proposition.

As $G$ is a reductive group, the Lie bracket on $\mathfrak{g}_{\bar{K}}$ comes from the Lie bracket on $\mathfrak{g}$. As a reminder and according to Lemma \ref{HN_dual}, the ideal $E_{-1}$ is orthogonal to $E_0$ for the Killing form on $\mathfrak{g}$. Thus we have:
\[(E_{-1})_{\bar{K}} := E_{-1} \times_C \Spec(\bar{K})= (E_0^{\perp})_{\bar{K}} = (E_0)_{\bar{K}}^{\perp}\]
and $E_{-1}$ is a nil ideal of $(E_0)_{\bar{K}}$.
According to Corollary \ref{cor_killing}, the subalgebra $(E_0)_{\bar{K}} \subseteq \mathfrak{g}_{\bar{K}}$ is thus a parabolic subalgebra of $\mathfrak{g}_{\bar{K}}$. It derives from a parabolic subgroup $\widetilde{Q}$ of $G_{\bar{K}}$. According to \cite[XXII, Corollaire 5.3.4]{SGA33} we have that $\widetilde{Q} = N_{G_{\bar{K}}}((E_0)_{\bar{K}})^0$ because $G_{\bar{K}}$ is of type $(RA)$ (as $p>2\dim(G)-2\geq 2\h-2$ the morphism $G\rightarrow G^{\Ad}$ is separable). This in particular implies the smoothness of the involved normaliser. 

Let us show that this parabolic subgroup is actually defined over $K$. As $E_0$ is defined over $C$, thus over $K$, we can consider the normaliser $N_{G_K}((E_0)_K)$ (which is representable according to \cite[II, \S 1 n\degree 3 Théorème 3.6 b)]{DG}). The equality $N_{G_{\bar{K}}}((E_0)_{\bar{K}})^0= N_{G_K}((E_0)_K)^0 \times_{K} \bar{K}$ holds true according to \cite[II, \S 1 n\degree 3]{DG}. As $N_{G_{\bar{K}}}((E_0)_{\bar{K}})^0$ is in particular smooth, the situation is compatible with base change so we have that:
\[(E_0)_K \times_K \bar{K} = (E_0)_{\bar{K}} =\Lie(\widetilde{Q})= \Lie(N_{G_{\bar{K}}}((E_0)_{\bar{K}})^0) = \Lie(N_{G_K}((E_0)_K)^0) \times_{K} \bar{K}.\]
In other words $(E_0)_K$ is the Lie algebra of a smooth subgroup of $G_K$ with parabolic geometric fibre, thus of a parabolic subgroup (by definition, see \cite[XXVI, Définition 1.1]{SGA33}).
 
The latter extends on a unique way to a parabolic subgroup $Q\subseteq G$ over $C$ (because the scheme of parabolic subgroups of $G$ is projective, see \cite[exposé XXVI, corollaire 3.5]{SGA33}). It remains to show that $\Lie(Q) = E_0$. As a reminder two $\mathcal{O}_C$-submodules of $\mathfrak{g} = E_l = E_{r-1}$ that are locally direct factors are equals if they are geometrically the same. This allows us to conclude because one has the following generic identity:
\[\Grass_S(n,E_0)(C) = \Grass_S(n,E_0)(K) =\{\text{subvector bundles of rank } n \text{ of } E_0\times_C K \},\] 
which is ensured by the projectivity of the $S$-scheme $\Grass_S(n,E_0)$ obtained for example in \cite[Definition 1.7.1 et Remarks 1.7.5]{KOL}. Let us remark that $\Grass_S(n,E_0)$ is well defined as $E_0$ is a locally direct factor, so is any factor of the above Harder--Narasimhan filtration.

By assumption (see also \cite[Planches]{BOU6}) the characteristic of $k$ satisfies $p>2\dim(G)>2(\h(G) -1)$. According to Lemma \ref{p_dans_E0} one thus has $\mathfrak{p} \subseteq E_0 = \mathfrak{q}$ and $E_{-1} = \radu(Q)$, that allows us to obtain the inclusions $E_{-1} \subseteq \radu(P) \subseteq \mathfrak{p} \subseteq E_0$. The quotient $E_0/E_{-1}$ is semi-stable hence $\mu(\mathfrak{p}/E_{-1}) < \mu(E_0/E_{-1}) = \mu(\mathfrak{q}/\radu(Q))=0$ and thus $ \deg(\mathfrak{p})\leq \deg(E_{-1}) = \deg(\mathfrak{q})$. The parabolic subgroup $P$ is the canonical parabolic subgroup of $G$, therefore it is maximal among the parabolic subgroups of $G$ of maximal degree (see \cite[Theorem 7.3]{BEH}), hence the equality  $P = Q$. This leads to the equality of Lie algebras $E_0 = \mathfrak{p}$ and $E_{-1} = \mathfrak{u}$ and ends the proof of the first point. Finally, according to Corollary \ref{cor_killing}, the equality $Q_{\bar{K}} = \tilde{Q} = P_{G_{\bar{K}}}(\lambda_{\mathfrak{p}_{\bar{K}}^{\perp}})$ holds on the generic geometric fibre. 
This settles the proof of the second point of Proposition \ref{E0_canonique}.

\begin{remark}
If $G$ is split over $X$, let $T\subseteq G$ be a maximal torus of $G$ and denote by $T_{\bar{K}} \subset G_{\bar{K}} \subset \widetilde{Q}$ the maximal subtorus obtained by base change. According to \cite[XXV, Corollaire 1.3]{SGA33} there exists a triple $(T_{\mathbb{Z}}, Q_{\mathbb{Z}}, G_{\mathbb{Z}})$ that provides a triple $(T_{\bar{K}},Q_{\bar{K}}=\widetilde{Q}, G_{\bar{K}})$ after base change. For any group of multiplicative type $H$ let us denote by $X^*(H)$ the group of cocharacters of $H$.

Let $\beta \in \mathbb{Q}_{>0}\lambda_{\mathfrak{p}_{\bar{K}}^{\perp}} \cap X^{*}(T_{\mathbb{Z}})$ be an indivisible cocharacter of the intersection. It defines a parabolic subgroup $P_{G_{\mathbb{Z}}}(\beta)$ such that $Q = P_G(\beta)$ et $(P_G(\beta))_{\bar{K}} = P_{G_{\bar{K}}}(\lambda_{\mathfrak{p}_{\bar{K}}^{\perp}})$. If $G$ is not split, the same reasoning is still valid locally, as $G$ is a twisted form of a constant group (thus $G$ can be pinned).	
\end{remark}

\paragraph{Acknowledgements}
The author would like to thank Philippe Gille for his precious help, his availability, his patience and all the fruitful conversations they had, Benoît Dejoncheere for his careful proofreading and advice, the reviewers of her PhD manuscript, Anne-Marie Aubert and Vikraman Balaji, for their useful remarks and corrections. Any critical remark must be exclusively addressed to the author of this paper.

\bibliography{bib}
\bibliographystyle{halpha}
\end{document}